\documentclass[11pt]{amsart}
\usepackage{amsxtra}
\usepackage{amsmath,amsthm,amssymb,enumerate}
\usepackage{setspace}
\newtheorem{theorem}{Theorem}[section]
\newtheorem*{theorem*}{Theorem}
\newtheorem{lemma}[theorem]{Lemma}
\newtheorem*{lemma*}{Lemma}
\newtheorem{example}[theorem]{Example}
\newtheorem*{example*}{Example}
\newtheorem{corollary}[theorem]{Corollary}
\newtheorem*{corollary*}{Corollary}
\newtheorem{definition}[theorem]{Definition}
\newtheorem*{definition*}{Definition}
\newtheorem{remark}[theorem]{Remark}
\newtheorem*{remark*}{Remark}
\newtheorem{notation}[theorem]{Notation}
\newtheorem*{notation*}{Notation}
\newtheorem{proposition}[theorem]{Proposition}
\newtheorem*{proposition*}{Proposition}

\def \g  {\mathfrak{g}}   % Lie algebra letters
\def \h  {\mathfrak{h}}

\def \n  {\mathfrak{n}}

\def \b  {\mathfrak{b}}

\def \sl {\mathfrak{sl}}

\def \t  {\mathfrak{t}}

\def \Cset {{\mathbb C}}
\def \Zset {{\mathbb Z}}
\def \Rset {{\mathbb{R}}}
\def \calA {{\mathcal{A}}}

\def \bfu {{\bf u}}

\def \sL {{\scriptscriptstyle{L}}}
\def \piL {\pi_{\sL}}
\def \sZ {{\scriptscriptstyle{Z}}}
\def \piZ {\pi_{\sZ}}
\def \sQ {{\scriptscriptstyle{Q}}}
\def \piQ {\pi_{\sQ}}

\def \sX {{\scriptscriptstyle{X}}}
\def \piX {\pi_{\sX}}
\def \sY {{\scriptscriptstyle{Y}}}
\def \piY {\pi_{\sY}}
\def \sU {{\scriptscriptstyle{U}}}

\def \pist {\pi_{\rm st}}
\def \Zu {Z_{\bfu}}
\def \la {\langle}
\def \ra {\rangle}
\def \lara {\la \, , \, \ra}
\def \O {\mathcal{O}}
\def \Upu {\Upsilon_{\bfu}}

\def \hs {\hspace{.2in}}
\def \ds {\dot{s}}

\def \pa {\partial}
\def \ga {\gamma}

\def \lal {\lambda_\alpha}
\def \sal {\sqrt{\lambda_\alpha}}
\def \lrw {\longrightarrow}

\def \T {{\mathbb{T}}}
\def \calC {{\mathcal{C}}}

\def \bk {{\bf k}}

\begin{document}
\setlength{\baselineskip}{1.1\baselineskip}

\title[Bott-Samelson varieties and Poisson Ore extensions]
{Bott-Samelson varieties and Poisson Ore extensions}
\author{Bal\'{a}zs Elek}
\address{
Department of Mathematics   \\
Cornell University\\
Ithaca, NY 14853, USA}              
\email{ bazse@math.cornell.edu}
\author{Jiang-Hua Lu}
\address{
Department of Mathematics   \\
The University of Hong Kong \\
Pokfulam Road               \\
Hong Kong}
\email{jhlu@maths.hku.hk}
\date{}
\begin{abstract} 
Let $G$ be a connected complex semi-simple Lie group, and let $Z_{{\bf u}}$ be 
an $n$-dimensional Bott-Samelson variety of $G$, where ${\bf u}$ is any sequence of simple reflections in 
the Weyl group of $G$.  We study the Poisson 
structure $\pi_n$ on $Z_{\bf u}$  defined by 
a standard multiplicative Poisson structure $\pi_{\rm st}$ on $G$.  We explicitly express 
$\pi_n$ on each of the $2^n$ affine coordinate
charts, one for every subexpression of ${\bf u}$, in terms of the root strings and the structure
constants of the Lie algebra of $G$. We show that 
the restriction of $\pi_n$ to each affine coordinate chart gives rise to a Poisson structure on the polynomial algebra 
${\mathbb{C}}[z_1, \ldots, z_n]$ which is an {\it iterated Poisson Ore extension} of $\mathbb{C}$ compatible with a rational action by 
a maximal torus of $G$. For canonically chosen $\pi_{\rm st}$, we show that
the induced Poisson structure on 
${\mathbb{C}}[z_1, \ldots, z_n]$ for every affine coordinate chart is in fact defined over ${\mathbb Z}$,
thus giving rise to 
an iterated Poisson Ore extension of any field ${\bf k}$ of arbitrary characteristic.  The special case of 
$\pi_n$ on the affine chart corresponding to the full subexpression of ${\bf u}$ yields an explicit formula for 
the standard Poisson structures on {\it generalized Bruhat cells} in Bott-Samelson coordinates. The paper 
establishes the foundation on generalized Bruhat cells and sets up the stage for their applications, some of which are
discussed in the Introduction of the paper. 
\end{abstract}

%\begin{abstract}
%Abstrait: 

%\end{abstract}

\maketitle
%\tableofcontents
%%%%%%%%%%%%%%%%%%%%   Introduction   %%%%%%%%%%%%%%%%%%%%%%%%%%%%

\section{Introduction}
\pagenumbering{arabic}
\label{Zw}
\subsection{Introduction}\label{intro-intro}
Poisson Ore extensions were introduced in \cite{Oh} as
Poisson analogs of Ore extensions in the theory of
non-commutative rings. 
Let $\bk$ be any field. A polynomial Poisson $\bk$-algebra 
\[
A = (\bk[z_1, z_2, \ldots, z_n], \;\{\, , \, \})
\]
is called an {\it iterated Poisson Ore extension (of $\bk$)} if for each $1 \leq i \leq n-1$, 
\begin{equation}\label{eq-zzz-0}
\{z_i, \; \bk[z_{i+1}, \ldots, z_n]\} \subset z_i \bk[z_{i+1}, \ldots, z_n] + \bk[z_{i+1}, \ldots, z_n],
\end{equation}
and such an extension is said to be {\it symmetric} if one also has
\[
\{\bk[z_{1}, \ldots, z_{i-1}], \; z_i\} \subset z_i \bk[z_{i}, \ldots, z_{i-1}] + \bk[z_1, \ldots, z_{i-1}], 
\hs 2 \leq i \leq n.
\]
When a split $\bk$-torus $\T$ acts rationally on a polynomial Poisson $\bk$-algebra by Poisson algebra automorphisms, one
may impose a certain compatibility condition between the $\T$-action and the iterations, leading to the notion of
{\it iterated $\T$-Poisson Ore extensions}. A sub-class of iterated $\T$-Poisson Ore extensions are 
the {\it symmetric iterated $\T$-Poisson Ore extensions}, which are also called {\it symmetric Poisson 
Cauchon-Goodearl-Letzter (CGL) extensions} by K. Goodearl and M. Yakimov
\cite{GY:Poi}. The Poisson bracket for a symmetric iterated $\T$-Poisson Ore extension  has the additional property that
\[
\{z_i, \; z_k\} = c_{i,k} z_iz_k + f_{i, k} \;\;\; \mbox{with} \;
c_{i, k} \in \Cset, \;f_{i,k} \in \bk[z_{i+1}, \ldots, z_{k-1}], \;\forall 
1 \leq i < k \leq n.
\]
We refer to $\S$\ref{subsec-Poi-Ore} for the precise definitions of (symmetric) iterated $\T$-Poisson Ore extensions. In the
rest of the Introduction, we will use the terms {\it symmetric iterated $\T$-Poisson Ore extensions} and {\it symmetric 
Poisson CGL extensions} interchangeably. 

Iterated $\T$-Poisson Ore extensions have been studied
in \cite{GL, L-L:prime} for their $\T$-invariant Poisson prime ideals, and are shown to arise 
as semi-classical limits of certain quantum coordinate rings. 
Poisson CGL extensions form 
an axiomized class of Poisson algebras introduced and studied in \cite{GY:Poi}
by K. Goodearl and M. Yakimov as the Poisson analogs of {\it CGL extensions}. 
CGL extensions  were, in turn, introduced
in \cite{LLR:CGL} by S. Launois, T. Lenagan and L. Rigal 
as a class of non-commutative unique factorization domains to which 
the {\it Cauchon deleting derivation theory} \cite{Cauchon} and the {\it Goodearl-Letzter stratification theory} 
\cite{Goodearl-Letzter} apply. 
Poisson CGL extensions and CGL extensions are 
the starting points of the remarkable body of work of K. Goodearl and M. Yakimov on 
classical and quantum cluster algebras related to Lie theory, and especially on the 
classical and quantum Berenstein-Zelevinsky conjectures \cite{Goodearl-Yakimov:PNAS, GY:BZ-conjecture,
GY:AMSbook, GY:Poi}.
In particular, 
it is shown in \cite{GY:Poi} that the structure of a symmetric Poisson CGL extension on the polynomial algebra $A$
gives rise to a cluster algebra structure on $A$ compatible with the Poisson structure in the sense
of Gekhtman-Shapiro-Vainshtein \cite{GSV:book}. Similar results for CGL extensions are obtained in \cite{Goodearl-Yakimov:PNAS,
GY:AMSbook}. 

In this paper, for any connected complex semisimple Lie group $G$ with a maximal torus $T$,
we construct a class of iterated $T$-Poisson Ore extensions and a sub-class of
symmetric Poisson CGL extensions 
associated to Bott-Samelson varieties of $G$. The extensions 
will be initially defined over $\Cset$, but we show that 
the polynomials $\{z_i, z_k\} \in \Cset[z_1, \ldots, z_n]$ all have
integer coefficients, resulting in iterated Poisson Ore extensions of an arbitrary field $\bk$.
 While the main text of the paper 
concentrates on these examples coming from Bott-Samelson varieties, we devote a large part of the Introduction discussing
applications, some of which have already been carried out.

Briefly, associated to any sequence $\bfu = (s_1, s_2, \ldots, s_n)$ of simple reflections in the Weyl group
$W$ of $G$, one has the $n$-dimensional Bott-Samelson variety $\Zu$ and a so-called {\it standard Poisson
structure} $\pi_n$ on $\Zu$. On the other hand, $\Zu$ has a natural atlas consisting of $2^n$ 
coordinate charts, one chart $\O^\gamma$ for each {\it subexpression} $\gamma$ of $\bfu$ (see $\S$\ref{coordcharts}) 
and all parametrized by $\Cset^n$. We prove that the
restriction of $\pi_n$ to {\it each} coordinate chart $\O^\gamma$ gives rise to 
an iterated  $T$-Poisson Ore extension, and for the case of $\gamma = \bfu$, 
a symmetric iterated $T$-Poisson Ore extension.
Moreover, the restriction of $\pi_n$ to $\O^\gamma$ for $\gamma = (e, e, \ldots, e)$ (see $\S$\ref{coordcharts})
is always log-canonical.
More importantly, 
we give explicit formulas for the Poisson structure $\pi_n$ in all the coordinate charts
in terms of the root strings and the structure constants
of the Lie algebra $\g$ of $G$. Such formulas made it possible for the first author to write a computer
program in the GAP  language \cite{GAP4}
which computes the Poisson bracket $\{\, , \, \}_\gamma$ on ${\mathbb{Z}}[z_1, \ldots, z_n]$ for {\it any triple} 
$(G, \bfu, \gamma)$, where $G$ is a connected complex simple Lie group (the results only depend on the
isogeny class of $G$), $\bfu$ is a length $n$ sequence of simple reflections
in the Weyl group of $G$, and $\gamma$ is a subexpression of $\bfu$. 
Some examples, such as when $\g$ is of type $G_2$,  are given in $\S$\ref{subsec-examples}.

The origin of the Poisson structure $\pi_n$ on an $n$-dimensional Bott-Samelson variety $\Zu$ is 
the so-called {\it standard multiplicative Poisson structure} $\pist$ on $G$, and 
the Poisson Lie group $(G, \pist)$ is the semi-classical limit of the much studied quantum group
\cite{CP, dr:quantum, Etingof-Schiffmann}
associated to $G$. Results of this paper have applications to quantum groups, 
integrable systems, cluster algebras, and the algebraic geometry of various important varieties associated to $G$. 
In the next $\S$\ref{subsec-BS-intro} and $\S$\ref{subsec-explicit-intro}, we give more details on the results of the paper.
The rest of the Introduction, $\S$\ref{subsec-quantum} - $\S$\ref{subsec-degeneration}, are devoted to
discussions on applications and future research problems.

%In $\S$\ref{subsec-BS-intro}, we introduce the Poisson structure $\pi_n$ on a Bott-Samelson
%variety $\Zu$ of $G$, where $u = (s_1, s_2, \ldots, s_n)$ is any sequence of simple reflections 
%in the Weyl group of $G$, and we give the outline of the results of the paper. 

\subsection{Bott-Samelson varieties and iterated Poisson Ore extensions}\label{subsec-BS-intro}
Let $G$ again be a connected complex semi-simple Lie group with a fixed Borel subgroup $B$ and a maximal torus $T \subset B$,
and let $\mathfrak{g}$,
${\mathfrak{b}}$, and $\mathfrak{h}$ be the respective Lie algebras of $G$, $B$, and $T$. Let
$\Delta_+\subset \mathfrak{h}^*$ be the set of positive roots determined by $\b$ and 
let $\Gamma\subset\Delta_+$ the corresponding set of simple roots.  
Let $W=N_G(T)/T$ be the Weyl group, where $N_G(T)$ is the normalizer of $T$ in $G$. For $\alpha \in \Gamma$, let
$s_\alpha \in W$ be the corresponding simple reflection. 

Let ${\bfu}=(s_1,s_2,\cdots ,s_n)$ be any sequence of simple reflections in $W$, and for $1\leq i \leq n$, 
let $P_{s_i}=B\cup Bs_iB$, the parabolic subgroup of $G$ containing $B$ that is associated to $s_i$.
Consider the product manifold
$P_{s_1}\times  \ldots \times P_{s_n}$ with the right action of $B^n$  by
\[
(p_1, p_2,\ldots ,p_n) \cdot (b_1,b_2,\ldots, b_n)=(p_1b_1,\, b_1^{-1}p_2b_2,\, \ldots, \, b_{n-1}^{-1}p_nb_n), 
\hs p_i \in P_{s_i}, \, b_i \in B.
\]
The quotient space, denoted by  $\Zu=P_{s_1}\times_B \ldots \times_B P_{s_n}/B$, is the Bott-Samelson variety associated to $\bfu$. For $(p_1, \ldots, p_n)\in P_{s_1} \times \ldots\times P_{s_n}$, let
$[p_1,  \ldots, p_n]\in \Zu$ denote the 
image of $(p_1, \ldots, p_n)$ in $\Zu$.
Multiplication in the group $G$ gives a well-defined map
\begin{equation}\label{eq-mu-intro}
\mu: \;\Zu\longrightarrow G/B:\;\; \mu([p_1, \, p_2,\, \ldots, \, p_n])= p_1p_2\cdots p_nB/B.
\end{equation}
When $\bfu$ is a reduced word, $\mu$ is a resolution of singularities of the Schubert variety 
$\overline{BuB/B}$ in $G/B$, where $u = s_1 s_2\cdots s_n \in W$.
Bott-Samelson varieties have been studied extensively in the literature and play an important role in
geometric representation theory. See, for example, \cite{Brion-flag, BK} and the references therein.

It is well-known (see, for example, \cite[$\S$1.5]{CP} or \cite[$\S$4.4]{Etingof-Schiffmann}) that
the choice of the pair $(B,T)$,
together with that of a symmetric non-degenerate invariant bilinear form $\lara$ on $\g$,
give rise to a {\it standard multiplicative holomorphic Poisson structure} $\pist$, making
$(G, \pist)$ into the {\it standard complex semisimple Poisson Lie group} (see $\S$\ref{subsec-pist}).
 Every parabolic subgroup of $G$ containing $B$ is a Poisson submanifold of $(G, \pist)$.
Consequently, for any sequence $\bfu = (s_1,  \ldots, s_n)$ of simple reflections in $W$, 
the restriction to $P_{s_1}\times\ldots \times P_{s_n} \subset G^n$ of the $n$-fold product Poisson structure
$\pi_{\rm st}^n =\pist \times\ldots \times \pist$ on $G^n$  projects to a well-defined Poisson structure,
denoted by $\pi_n$, on the Bott-Samelson variety $\Zu$ (see $\S$\ref{piBS} for details).
We refer to $\pi_n$ as {\it a standard Poisson structure on $\Zu$}. 

%The complex Poisson manifolds $(\Zu, \pi_n)$, $n \geq 1$, are the objects of study of this paper.

Fixing root vectors $\{e_\alpha: \alpha \in \Gamma\}$ and extending them to a Chevalley basis 
of $\g$, one obtains an
atlas 
\begin{equation}\label{eq-calA}
\calA = \{(\Phi^\gamma: \; \Cset^n \longrightarrow \O^\gamma \subset Z_\bfu): \;\; \gamma \in \Upu\},
\end{equation}
on $Z_\bfu$, where 
$\Upu$ is the set of all the $2^n$ {\it subexpressions} of $\bfu$ (see $\S$\ref{coordcharts}). 
While referring to $\S$\ref{coordcharts} for the precise definition of the
parametrization $\Phi^\gamma: \Cset^n \to \O^\gamma$ for an arbitrary $\gamma \in \Upu$, we point out here that for
$\gamma = \bfu$,
\begin{equation}\label{eq-OU}
\O^\bfu = \varpi(Bs_1B \times Bs_2B \times \cdots \times Bs_nB) \subset \Zu,
\end{equation}
where $\varpi: P_{s_1} \times \cdots \times P_{s_n} \to \Zu$ is the projection. The coordinate chart
$\Phi^\bfu: \Cset^n \to \O^\bfu$ will play a very special role in this paper and for applications of the
results in this paper.

In $\S$\ref{subsec-coor-I}, we give our first formula (Lemma \ref{le-zizj-1})
of the Poisson structure $\pi_n$ in each coordinate chart in terms of certain vector fields on Bott-Samelson
subvarieties of $\Zu$. It is also shown in $\S$\ref{subsec-log} that $\pi_n$ is log-canonical in some of the
coordinate charts. The first main result of the paper is Theorem \ref{th-zizj}, in which we further express the
vector fields in Lemma \ref{le-zizj-1}
in terms of root strings and the structure constants of $\g$. In particular, $\pi_n$ is algebraic in every coordinate chart
$\Phi^\gamma: \Cset^n \to \O^\gamma$.
 
For $\gamma \in \Upu$, 
let $\{\, , \, \}_\gamma$ be the Poisson bracket on the polynomial algebra
$\Cset[z_1, \ldots, z_n]$ defined by $\pi_n$ through the parametrization
$\Phi^\gamma: \Cset^n \to \O^\gamma$.
As consequences of Theorem
\ref{th-zizj}, we prove in $\S$\ref{sec-pi-gamma} the following prominent features of the Poisson polynomial algebras
$(\Cset[z_1, \ldots, z_n], \{\, , \, \}_\gamma)$ for every $\gamma \in \Upu$:

1)  The Poisson polynomial algebra $(\Cset[z_1, \ldots, z_n], \{\, , \, \}_\gamma)$   
is an {\it iterated $T$-Poisson Ore extension of ${\mathbb C}$} and is of the form
\begin{equation}\label{eq-z-f}
\{z_i, \; z_k\}_\gamma = c_{i, k} z_i z_k + b_i(z_k), \hs 1 \leq i < k \leq n,
\end{equation}
where $c_{i, k} \in \Cset$ and $b_i$ is a derivation of $\Cset[z_{i+1}, \ldots, z_k]$.
When $\gamma = \bfu$, the iterated $T$-Poisson Ore extension 
$(\Cset[z_1, \ldots, z_n], \{\, , \, \}_\bfu)$
of $\Cset$ is {\it symmetric}
(and is thus a {\it symmetric Poisson CGL extension} in the terminology of \cite{GY:Poi}), and in particular
\begin{equation}\label{eq-symmetric}
b_i(z_k) \in \Cset[z_{i+1}, \ldots, z_{k-1}], \hs \forall \;\; 1 \leq i < k \leq n. 
\end{equation}
See $\S$\ref{subsec-Poi-Ore} for the precise definitions and Theorem \ref{thm-Ore} for the precise statements.
We will also give more details on $\{\, , \, \}_\bfu$ in $\S$\ref{subsec-explicit-intro}.

2) Choose the bilinear form $\lara$ on $\g$ such that
$\frac{\la \alpha, \alpha\ra}{2} \in {\mathbb Z}$ for each root $\alpha$. Then for any $\gamma \in \Upu$
and $1 \leq i < k \leq n$, the polynomial $\{z_i, z_k\}_\gamma \in \Cset[z_1, \ldots, z_n]$ 
has integer coefficients, so $\{\, , \, \}_\gamma$ defines a Poisson bracket on $\Zset[z_1, \ldots, z_n]$.
Consequently, each $\gamma \in \Upu$ gives rise to an iterated Poisson Ore extension 
$({\bf k}[z_1, \ldots, z_n], \{\, , \, \}_\gamma)$ of any field
${\bf k}$ of arbitrary characteristic. In particular, when the shortest roots $\alpha$ satisfy
$\la \alpha, \alpha \ra = 2$, associated to $\gamma = \bfu$ one then 
has a symmetric Poisson CGL extension of any field ${\bf k}$
with ${\rm char}(k) \neq 2, 3$.
See Theorem \ref{th-Z} and Remark \ref{re-mod-2}.

When $u = (s_1, s_2, \ldots, s_n)$ is a reduced word, the map $\mu$ in \eqref{eq-mu-intro} restricts to an
isomorphism between $\O^\bfu$ to $BuB/B \subset G/B$, where $u = s_1s_2 \cdots s_n$. On the other hand, 
as $B$ is a Poisson Lie subgroup of $(G, \pist)$, the Poisson structure $\pist$ on $G$ projects to a well-defined Poisson structure on $G/B$,
denoted as $\pi_1$ (this notation will become clear in $\S$\ref{subsec-g-cell}), with respect to which $BuB/B$ is a Poisson submanifold \cite{Goodearl-Yakimov:GP}.
It then follows from the 
definition of $\pi_n$ and the multiplicativity of $\pist$ that 
\begin{equation}\label{eq-mu-O}
\mu|_{\O^\bfu}: \;\; (\O^\bfu, \; \pi_n) \longrightarrow (BuB/B, \; \pi_1)
\end{equation}
is a Poisson isomorphism. Referring to the coordinates $(z_1, \ldots, z_n)$ on $\O^\bfu$ as
the {\it Bott-Samelson coordinates on $BuB/B$ (defined by the reduced expression $u = s_1s_2 \cdots s_n$)}
via the isomorphism $\mu|_{\O^\bfu}$, 
Theorem \ref{thm-Ore} then says that the Poisson coordinate ring of $(BuB/B, \pi_1)$ becomes 
a symmetric Poisson CGL extension in the Bott-Samelson coordinates on $BuB/B$. 

In applications, however, it is crucial that we have a symmetric Poisson CGL extension
associated to an arbitrary (i.e., not necessarily reduced) sequence $\bfu = (s_1, s_2, \ldots, s_n)$
of simple reflections. See $\S$\ref{subsec-g-cell} - $\S$\ref{subsec-degeneration}.

\subsection{The explicit formulas in Theorem \ref{th-zizj}}\label{subsec-explicit-intro}
Let again $\bfu = (s_1, s_2, \ldots, s_n)$ be any sequence of simple reflections, not necessarily reduced. 
For $1 \leq j \leq n$, let $\alpha_j$ be the simple root such that
$s_j = s_{\alpha_j}$.
For each subexpression $\gamma$ of $\bfu$,  Theorem \ref{th-zizj} gives the explicit formulas for 
the Poisson brackets $\{z_i , \, z_k\}_\gamma$ among the coordinate functions $(z_1, \ldots, z_n)$ on $\O^\gamma$.
To show the nature of the explicit formulas, and in particular to show how root strings and the 
structure constants of the Lie algebra $\g$ of $G$ appear in the formulas, 
we (taking the risk of burdening the reader with too much detail in the Introduction) now
state the formulas for the case of $\gamma = \bfu$, which will be the most importance case for applications.

\begin{theorem}\label{th-intro}
In the coordinates $(z_1, \ldots, z_n)$ on  $\O^\bfu$, and for $1 \leq i < k \leq n$, one has
\[
\{z_i, \, z_k\}_\bfu = c_{i, k}  z_iz_k + b_i(z_k),
\]
where $c_{i, k}=-\la s_1s_2 \cdots s_{i-1}(\alpha_i), \; s_{1}s_{2} \cdots s_{k-1}(\alpha_k)\ra$, and
$b_i(z_k) \in \Cset[z_{i+1}, \ldots, z_{k-1}]$ is  given as follows:

1) If $k = i+1$, one has $b_i(z_{i+1}) = 0$ if $\alpha_{i+1} \neq \alpha_i$, and
$b_i(z_{i+1}) = -\la \alpha_i, \alpha_i\ra$ if $\alpha_{i+1} = \alpha_i$;

2) Assume that $k > i+1$. For $(j_{i+1}, \ldots, j_{k-1}) \in {\mathbb{N}}^{k-i-1}$
and $i+1 \leq l \leq k-1$, let 
\begin{align*}
\beta_{(j_{i+1}, \ldots, j_l)} &= s_l s_{l-1} \cdots s_{i+2}s_{i+1}(\alpha_i) -
j_{i+1} s_l s_{l-1} \cdots s_{i+2}(\alpha_{i+1}) - \cdots - j_{l-1} s_l(\alpha_{l-1}) - j_l\alpha_l\\
& = s_l(\beta_{(j_{i+1}, \ldots, j_{l-1})})-j_l\alpha_l  \in \h^*,
\end{align*}
where $\beta_{(j_{i+1}, \ldots, j_{l-1})} = \alpha_i$ if $l = i+1$.
Let $J_{i, k} \subset {\mathbb{N}}^{k-i-1}$ be given by
\begin{align*}
J_{{i, k}} &= \{(j_{i+1}, \ldots, j_{k-1}) \in {\mathbb{N}}^{k-i-1}: \; \beta_{(j_{i+1}, \ldots, j_l)} \in \Delta_+
 \; \forall \, i+1 \leq l \leq k-1 \; \mbox{and} \;\\
&\hspace{2in}\;\;  
\beta_{(j_{i+1}, \ldots, j_{k-1})} = \alpha_k\}. 
\end{align*}
If $J_{{i, k}} = \emptyset$, then $b_i(z_k) = 0$. Otherwise,
\begin{equation}\label{eq-bizk}
b_i(z_k) = -\la \alpha_i, \, \alpha_i\ra \sum_{(j_{i+1}, \ldots, j_{k-1}) \in J_{i,k}} c_{j_{i+1}, \ldots, j_{k-1}} z_{i+1}^{j_{i+1}} \cdots z_{k-1}^{j_{k-1}},
\end{equation}
where for  $(j_{i+1}, \ldots, j_{k-1}) \in J_{i,k}$,
\[
c_{j_{i+1}, \ldots, j_{k-1}}= c_{\alpha_{i+1}, \alpha_i}^{s_{i+1}, j_{i+1}} \, 
c_{\alpha_{i+2}, \beta_{(j_{i+1})}}^{s_{i+2}, j_{i+2}} \,\cdots \, 
c_{\alpha_{k-1}, \beta_{(j_{i+1}, \ldots, j_{k-2})}}^{s_{k-1}, j_{k-1}} \neq 0,
\]
and for $i+1 \leq l \leq k-1$, $c_{\alpha_l, \beta_{(j_{i+1}, \ldots, j_{l-1})}}^{s_l, j_l}$ is a certain binomial
coefficient with plus or minus sign, with the binomial coefficient being determined by the $\alpha_l$-string of roots through 
$\beta_{(j_{i+1}, \ldots, j_{l-1})}$ and the plus or minus sign determined by the signs of the structure constants of
$\g$ in the chosen Chevalley basis,  as in \eqref{constants-1}, \eqref{const} and \eqref{phibeta}.
\end{theorem}

Theorem \ref{th-intro} is extracted from Theorem \ref{th-zizj} and \eqref{eq-Vi-zk}
for the case of $\gamma = \bfu$. Formulas for the Poisson bracket $\{\, , \, \}_\gamma$
for the general case of $\gamma \in \Upsilon_\bfu$
are of similar nature but more involved. We refer to Theorem \ref{th-zizj} for detail.

In the remaining $\S$\ref{subsec-quantum} - $\S$\ref{subsec-degeneration} of the Introduction, we 
discuss applications via quantization and the ubiquitous presence of the so-called
{\it generalized Bruhat cells} in Lie theory (see $\S$\ref{subsec-g-cell} and $\S$\ref{subsec-building-blocks}). 

\subsection{The Poisson analog of the Levendorskii-Soibelman strengthening law for quantum Schubert cells}\label{subsec-quantum}
Consider again the case when $\bfu = (s_1, s_2, \ldots, s_n)$ is reduced, and let
$u = s_1s_2 \cdots s_n \in W$. 
Recall that the Bruhat cell $BuB/B$ is a Poisson submanifold of $G/B$ with respect to the Poisson structure $\pi_1$,
the projection of $\pist$ to $G/B$.  It is well-known that $(BuB/B, -\pi_1)$ is
the semi-classical analog of the quantum Schubert cell $\mathcal{U}^-[u]$ introduced by 
De Concini-Kac-Procesi \cite{DKP} and Lusztig \cite{Lusztig:quantum} (see \cite[Lemma 4.3]{Milen:strata}), and that the coordinates
$(z_1, \ldots, z_n)$ on $\O^\bfu$, now regarded as regular functions on $BuB/B$ via the 
isomorphism $\mu|_{\O^\bfu}: \O^\bfu \cong BuB/B$, are the semi-classical analogs of the Lusztig root vectors
$F_1, \ldots, F_n \in {\mathcal{U}}^-[u]$ (see \cite[$\S$9.2]{GY:AMSbook}). 
Recall \cite[$\S$9.2]{GY:AMSbook} \cite[I.6.10]{Brown-Goodearl:book} that the Lusztig
root vectors $F_1, \ldots, F_n$ satisfy the {\it Levendorskii-Soibelman straightening law}
\[
F_iF_k-q^{\la \beta_i, \beta_k\ra}F_kF_i = \sum_{(j_{i+1}, \ldots, j_{k+1}) \in {\mathbb{N}}^{k-i-1}}
\xi_{j_{i+1}, \ldots, j_{k+1}} F_{i+1}^{j_{i+1}} \cdots F_{k-1}^{j_{k-1}},\hs 1 \leq i < k \leq n,
\]
where $\beta_i = s_1s_2 \cdots s_{i-1}(\alpha_i)$, $\beta_k = s_1s_2 \cdots s_{k-1}(\alpha_k)$,
and $\xi_{j_{i+1}, \ldots, j_{k+1}} \in {\mathbb{Q}}[q, q^{-1}]$ for $(j_{i+1}, \ldots, j_{k+1}) \in {\mathbb{N}}^{k-i-1}$.
Thus the fact that the Poisson bracket
$\{\, , \, \}_\bfu$ is of the form
\[
\{z_i, \, z_k\}_\bfu = - \la \beta_i, \beta_k\ra z_iz_k + b_i(z_k), \hs 1 \leq i < k \leq n,
\]
with $b_i(z_k) \in \Cset[z_{i+1}, \ldots, z_{k-1}]$ is the Poisson analog of the Levendorskii-Soibelman straightening law.
However, while we have expressed all the polynomials $b_i(z_k)$ explicitly in terms of roots strings and structure
constants of the Lie algebra $\g$ in Theorem \ref{th-intro}, there are no similar formulas, as far as we know,
neither for
the subset of all indices $(j_{i+1}, \ldots, j_{k+1})$ for which $\xi_{j_{i+1}, \ldots, j_{k+1}}\neq 0$ nor 
for the elements $\xi_{j_{i+1}, \ldots, j_{k+1}} \in {\mathbb{Q}}[q, q^{-1}]$ themselves
(see, however, \cite[Appendix, (A4)-(A8)]{DP}
for some concrete formulas for the cases of $u$ being the longest elements in $W$ for the rank $2$ groups). 
It would thus be very interesting to
seek for a {\it quantization} of the formulas in Theorem \ref{th-intro} to obtain 
explicit expressions of the Levendorskii-Soibelman straightening law, and see in particular how the 
$q$-analogs of the binomial 
coefficients in Theorem \ref{th-intro} may appear such formulas. Partial results in this direction have been 
obtained in  \cite{Mi:thesis}.

\subsection{Symmetric Poisson CGL extensions through generalized Bruhat cells}\label{subsec-g-cell} 
With the notation as in $\S$\ref{intro-intro}, for any integer $n \geq 1$, let $B^n$ act on $G^n$ by 
\[
(g_1, g_2\ldots g_n) \cdot (b_1,b_2,\ldots, b_n)=(g_1b_1,\, b_1^{-1}g_2b_2,\, \ldots, \, b_{n-1}^{-1}g_nb_n), 
\hs g_i \in G, \, b_i \in B,
\]
and denote the quotient manifold by
\begin{equation}\label{eq-Fn}
F_n = G \times_B \cdots \times_B G/B.
\end{equation}
It is shown in \cite[$\S$7.1]{Lu-Mou:mixed} (see also $\S$\ref{piBS}) that the 
$n$-fold product Poisson structure $\pi_{\rm st}^n$ on $G^n$ projects to a well-defined Poisson structure on $F_n$, which
will also be denoted by $\pi_n$. Note that for any sequence $\bfu = (s_1, \ldots, s_n)$ of
simple reflections in $W$, the Bott-Samelson variety $\Zu$ 
is isomorphic to a closed submanifold of $F_n$ under the embedding $P_{s_1} \times \cdots \times P_{s_n} \subset G^n$.
As $P_{s_1} \times \cdots \times P_{s_n}$ is a Poisson submanifold of $G^n$   with respect to
$\pi_{\rm st}^n$, it follows from the definitions that $\Zu$, with the Poisson structure $\pi_n$ defined in 
$\S$\ref{intro-intro}, is a Poisson submanifold of
$(F_n, \pi_n)$. Note also that the $T$-action on the first factor of $G^n$ by left translation descends to 
a $T$-action on $F_n$ preserving $\pi_n$.

For an arbitrary sequence $\bfu = (u_1, \ldots, u_n)$ in the Weyl group $W$, where the $u_i$'s 
are not necessarily simple reflections,
the image of $B u_1B \times \cdots \times Bu_nB$ in $F_n$, denoted by
\[
B \bfu B/B = (B u_1 B)\times_B \cdots \times_B (B u_n B)/B \subset F_n,
\]
 is called 
a {\it generalized Bruhat cell} in \cite[$\S$1.3]{Lu-Mou:flags}. It follows from the Bruhat decomposition
$G = \bigsqcup_{u \in W} BuB$ of $G$ that one has the decomposition
\begin{equation}\label{eq-BuB}
F_n = \bigsqcup_{\bfu \in W^n} B\bfu B/B
\end{equation}
of $F_n$ into the disjoint union of generalized Bruhat cells.
As each $BuB$, where $u \in W$, is a Poisson submanifold of $G$ with respect to $\pist$, each generalized Bruhat cell
$B \bfu B/B$ is a $T$-invariant Poisson submanifold of $(F_n, \pi_n)$.

A generalized Bruhat cell of the form 
$B (s_1, \ldots, s_n) B/B \subset F_n$, where  
each $s_i$ is a simple reflection, is said to be of 
{\it Bott-Samelson type} \cite[$\S$1.3]{Lu-Mou:flags}. 
In the notation of the current paper, a generalized Bruhat cell $B (s_1, \ldots, s_n) B/B$ in $F_n$ of Bott-Samelson type
is nothing but the affine chart $\O^\bfu$ in the Bott-Samelson variety
$\Zu \subset F_n$, where $\bfu = (s_1, \ldots, s_n)$. See \eqref{eq-OU}.
Given an arbitrary $\bfu = (u_1, \ldots, u_n) \in W^n$, 
choose any reduced decomposition $u_i = s_{i, 1} s_{i, 2} \cdots s_{i, l(u_i)}$ for each $u_i$, where 
$l: W \to {\mathbb{N}}$ is the length function of 
$W$, and consider the sequence 
\[
\tilde{\bfu} = (s_{1, 1}, \; \ldots, \; s_{1, l(u_1)}, \; s_{2, 1}, \; \ldots, \; s_{2, l(u_2)}, \; \ldots,
s_{n, 1}, \; \ldots, \; s_{n, l(u_n)})
\]
of simple reflections of length $l(\bfu) =l(u_1) + \cdots + l(u_n)$. Then the
multiplication map on $G$ induces a $T$-equivariant Poisson isomorphism
\begin{equation}\label{eq-iso}
(Z_{\tilde{\bfu}}, \, \pi_{l(\bfu)}) \supset (\O^{\tilde{\bfu}}, \, \pi_{l(\bfu)}) =
(B \tilde{\bfu} B/B, \, \pi_{l(\bfu)}) \;\stackrel{\sim}{\lrw} \;
(B \bfu B/B, \,\pi_n) \subset (F_n, \, \pi_n)
\end{equation}
(see \cite[$\S$1.3]{Lu-Mou:flags}). Thus, as Poisson manifolds, every generalized Bruhat cell $B \bfu B/B$ 
is Poisson isomorphic to one of Bott-Samelson type. We will refer to 
the coordinates $(z_1, \ldots, z_{l(\bfu)})$ 
on $\O^{\tilde{\bfu}}=B \tilde{\bfu} B/B$ as defined in $\S$\ref{coordcharts} of the present paper as 
{\it Bott-Samelson coordinates} on $B \bfu B/B$ via the isomorphism in \eqref{eq-iso}.
Theorem \ref{thm-Ore} then immediately leads to the following conclusion on generalized Bruhat cells.

\begin{theorem}\label{th-g-cell}
For any generalized Bruhat cell $B \bfu B/B$, where $\bfu =(u_1, \ldots, u_n) \in W^n$,
the standard Poisson structure on $B \bfu B/B$ makes its coordinate ring into a symmetric Poisson
CGL extension in any Bott-Samelson coordinates $(z_1, \ldots, z_{l(\bfu)})$ defined through the isomorphism in
\eqref{eq-iso}; the corresponding
Poisson bracket on $\Cset[z_1, \ldots, z_{l(\bfu)}]$ is explicitly given in Theorem \ref{th-intro}.
\end{theorem}

We also point out that for an arbitrary generalized Bruhat cell $B\bfu B/B$, where $\bfu = (u_1, u_2, \ldots, u_n) \in W^n$,
the $T$-orbits of symplectic leaves of $\pi_n$ in $B \bfu B/B$, also called {\it $T$-leaves}, are described in
\cite[Theorem 1.1]{Lu-Mou:flags}. Namely, the $T$-leaves of $\pi_n$ in $B \bfu B/B$ are 
precisely all the submanifolds of $B \bfu B/B$ of the form 
\[
R^\bfu_w = \{[g_1, g_2, \ldots, g_n] \in B \bfu B/B: \; g_1g_2\cdots g_n \in B_-wB\},
\]
where $w \in W$, and $w \leq u_1 \ast u_2 \ast \cdots \ast u_n$, with $\ast$ being the monoidal product
on $W$. Here $B_-$ is the Borel subgroup of $G$ such
that $B \cap B_- = T$. Moreover, the dimension of every symplectic leaf of $\pi_n$ in $R^\bfu_w$ is shown in 
\cite[Theorem 1.1]{Lu-Mou:flags} to be 
equal to 
\[
l(\bfu) - l(w) -\dim {\rm ker}(1+u_1u_2 \cdots u_n w^{-1}),
\]
where $1+u_1u_2 \cdots u_n w^{-1}$ denotes
the linear operator on the Lie algebra $\h$ of $T$ given by $x \mapsto x + u_1u_2 \cdots u_n w^{-1}(x)$, $x \in \h$.
The {\it leaf-stabilizer subalgebra of $\h$} in $R^\bfu_w$ is explicitly described in \cite[Theorem 1.1]{Lu-Mou:flags}.
We refer to \cite[Theorem 1.1]{Lu-Mou:flags} for more detail. 

We regard Theorem \ref{th-g-cell} and the 
description of their $T$-leaves in \cite[Theorem 1.1]{Lu-Mou:flags} 
as two basic results on the standard Poisson structures on generalized Bruhat cells.
These two results set the foundation for 
applications of generalized Bruhat cells, some of which will be discussed in the remainder of the Introduction.

In the next $\S$\ref{subsec-BS-coor-Guv}, we describe an application from \cite{Lu-Mi:integrable} of Theorem \ref{th-g-cell} to
the Kogan-Zelevinsky integrable systems on double Bruhat cells. 
In $\S$\ref{subsec-building-blocks}, we explain how 
generalized Bruhat cells
are building blocks for many of the Poisson varieties related to the 
Poisson Lie group $(G, \pist)$, and how, as a result, important varieties in Lie theory may be regarded as being
{\it glued together by symmetric Poisson CGL extensions} constructed out of generalized Bruhat cells. In $\S$\ref{subsec-mutation},
 we discuss the implications
of such gluings in the context of mutations of iterated Poisson Ore extensions and cluster algebras.
To give evidence that symmetric Poisson CGL extensions constructed in this paper may be intimately related to
algebraic geometry, we explain in 
$\S$\ref{subsec-degeneration} a result from \cite{Lu-Peng:degeneration, Peng:thesis} that 
connects, for any Bott-Samelson variety $\Zu$, the symmetric Poisson CGL extension
$(\Cset[z_1, z_2, \ldots, z_n], \{\, , \, \}_\bfu)$
with toric degenerations of $\Zu$ through tropical geometry.

\subsection{Bott-Samelson coordinates on double Bruhat cells}\label{subsec-BS-coor-Guv}
Recall \cite{FZ:total}  that double Bruhat cell in $G$ are defined as 
\[
G^{u, v} = (B u B) \cap (B_-vB_-),
\]
where $u, v \in W$. Fomin and Zelevinsky introduced in \cite{FZ:total} certain regular functions on $G^{u, v}$, called
{\it twisted generalized minors}, which play crucial roles in the theory of total positivity and cluster
algebra structures \cite{BZ} on the coordinate ring of $G^{u, v}$. Moreover, for the case of $u = v$, Kogan and Zelevinsky
introduced in \cite{KZ} an integrable system on the Poisson manifold $(G^{u, u}, \pist)$ formed by some twisted generalized minors on $G^{u, u}$.
Generalizing results of M. Gekhtman and M. Yakimov \cite{Misha-Milen} for the case of $SL(n)$,
it is proved in \cite{Lu-Mi:integrable} that the Hamiltonian vector fields 
of
all the Fomin-Zelevinsky twisted generalized minors on every $G^{u, v}$ 
are {\it complete} in the sense
that all of their integral curves are defined on the whole of $\Cset$. Consequently 
all the Hamiltonian flows of the Kogan-Zelevinsky integrable system 
on each $G^{u, u}$ are defined on the whole of $\Cset$, much in the same spirit as the results of
B. Kostant and N. Wallach in  \cite{KW1, KW2} on the completeness of the Hamiltonian flows of the complex Gelfand-Tseitlin 
system on the space of all $m \times m$ complex matrices.

The main tools used in \cite{Lu-Mi:integrable} are the generalized Bruhat cell
$B(v^{-1}, u)B/B$ and a certain open Poisson embedding
\begin{equation}\label{eq-FZ-embedding}
F^{u, v}: \;\;\; (G^{u, v}, \; \pist) \longrightarrow (T \times (B(v^{-1}, u)B/B),\; \, 0 \bowtie \pi_2),
\end{equation}
called the 
{\it Fomin-Zelevinsky embedding}, where $0 \bowtie \pi_2$ is the sum of the
product Poisson structure $0 \times \pi_2$ and a {\it mixed term} defined using the
$T$-action on $B(v^{-1}, u)B/B$ by left translation. Defining {\it Boot-Samelson coordinates}
on $G^{u, v}$ to be the combination of any (algebraic) coordinates on $T$ and Bott-Samelson coordinates on
$B(v^{-1}, u)B/B$
(defined using reduced words for $u$ and $v$ as in $\S$\ref{coordcharts} of the present paper), it is shown in \cite{Lu-Mi:integrable} that 
all the Fomin-Zelevinsky twisted generalized minors on $G^{u, v}$ become certain distinguished
polynomials in the Bott-Samelson coordinates, and the fact that the Poisson structure $\pi_2$ on 
$B (v^{-1}, u)B/B$ is a symmetric iterated $T$-Poisson Ore extension is used in an essential way in 
\cite{Lu-Mi:integrable} to 
prove that the Hamiltonian flows of the distinguished polynomials are complete.
The Kogan-Zelevinsky integrable systems are also generalized 
in \cite{Lu-Mi:integrable} to the setting of
arbitrary generalized Bruhat cells.

\subsection{Generalized Bruhat cells as building blocks}\label{subsec-building-blocks}
It is a fundamental fact that the Poisson Lie group
$(G, \pist)$ is {\it generated by} the $3$-dimensional Poisson subgroups $SL_\alpha(2, \Cset)$, one for
each simple root $\alpha$ (see \cite[$\S$2.2]{KZ} for a precise statement). This salient feature of 
$(G, \pist)$ makes it possible to build, through actions of its Poisson subgroup $(B, \pist)$, 
coordinates on certain Poisson manifolds 
by a {\it one simple root each time} procedure, resulting in the Poisson structure being 
constructed {\it one 
coordinate each time} in the fashion of \eqref{eq-zzz-0}. 
Lemma \ref{le-Z-Pi-1}  of this paper
gives a precise statement to this effect
on building Poisson structures on ${\mathbb{P}}^1$-extensions, of which Bott-Samelson varieties
are prototypical examples.
Indeed, Lemma \ref{le-Z-Pi-1} is the key to the appearance of iterated Poisson Ore extensions associated to
the affine coordinate charts on Bott-Samelson varieties. 

More precisely, we say that an $n$-dimensional complex 
algebraic Poisson manifold $(P, \pi)$ is an
iterated Poisson Ore extension (of a point) (resp. a symmetric Poisson CGL extension (of a point))
if there exists an isomorphism $P \cong \Cset^n$ through which the Poisson coordinate ring of
$(P, \pi)$ becomes an iterated Poisson Ore extension (resp. a symmetric Poisson CGL extension) of $\Cset$.
A Poisson manifold $(X, \piX)$ is {\it paved (resp. covered) by iterated Poisson Ore extensions}
if it is the disjoint union of (resp. has an open cover by) iterated Poisson Ore extensions.
Theorem \ref{thm-Ore} in this paper, then, says that every Bott-Samelson variety $\Zu$ with the Poisson
structure $\pi_n$ is {\it covered} by iterated Poisson Ore extensions. The
decomposition in 
\eqref{eq-BuB} says that the Poisson manifold $(F_n, \pi_n)$ is {\it paved} by symmetric Poisson CGL extensions, namely
by generalized Bruhat cells $B\bfu B/B \subset F_n$.

Consider now the Poisson manifold $(F_1, \pi_1) = (G/B, \pi_1)$ and the finite open cover
\[
G/B = \bigcup_{u \in W} uB_-B/B.
\]
In their study in \cite{KWY} of singularities of Richardson varieties in $G/B$,
 A. Knutson, A. Woo, and A. Yong introduced, for each $u \in W$, an isomorphism
\[
K^u: \;\; uB_-B/B \longrightarrow (Bw_0uB/B) \times (BuB/B) \subset (G/B) \times (G/B),
\]
where $w_0$ is the longest element in $W_0$, 
which they then use to express singularities of Richardson varieties in $G/B$ in terms of that of
Schubert varieties in $(G/B) \times (G/B)$. Equip 
both $Bw_0uB/B$ and $BuB/B$ with the Poisson structure $\pi_1$, it is shown in \cite{Yu:thesis} that
$K^u$ is a Poisson isomorphism from the open Poisson submanifold $uB_-B/B$ of $(G/B, \pi_1)$ to 
the {\it mixed product Poisson manifold}
\[
((Bw_0uB/B) \times (BuB/B), \; \, \pi_1 \bowtie \pi_1),
\]
where $\pi_1 \bowtie \pi_1$ denotes, again,  the sum of the product Poisson structure $\pi_1 \times \pi_1$ with a 
mixed term defined using the $T$-actions on $Bw_0uB/B$ and $BuB/B$ by left translation. 
Consequently, $((Bw_0uB/B) \times (BuB/B), \, \pi_1 \bowtie \pi_1)$ is a symmetric Poisson CGL extension, built
as a {\it $T$-mixed product of two Bruhat cells}. The Poisson manifold $(G/B, \pi_1)$ is thus  
{\it covered} by symmetric Poisson CGL extensions. Similarly, for $G/T$ with the Poisson structure
$\pi_{\rm st}^\prime$ that is the projection of $\pist$ to $G/T$, it is shown in 
\cite{Yu:thesis} that $(G/T, \pi_{\rm st}^\prime)$ is covered by {\it $T$-mixed products of
generalized Bruhat cells} of 
the form 
\[
((Bw_0uB/B) \times (B(u, w_0)B/B), \; \, \pi_1 \bowtie \pi_2).
\]
We refer to \cite{Yu:thesis} for more details and for other examples of Poisson homogeneous spaces of
the Poisson Lie group $(G, \pist)$ that are covered by $T$-mixed products of generalized Bruhat cells.

\subsection{Mutations and gluings of iterated Poisson Ore extensions and cluster algebras}\label{subsec-mutation}
Consider an $n$-dimensional smooth Poisson variety
$(Z, \pi)$ with a (finite) cover by iterated Poisson Ore extensions, as is the case for Bott-Samelson
varieties and for $(G/B, \pi_1)$ and $(G/T, \pi_{\rm st}^\prime)$ discussed in $\S$\ref{subsec-building-blocks}.
We can then regard $(Z, \pi)$ as being {\it glued together by finitely many iterated Poisson Ore extensions}, called
{\it Poisson Ore charts}.
On the other hand, the changes of coordinates 
between these coordinate charts are, in general, very nontrivial birational maps from $\Cset^n$ to itself.
See Example \ref{ex-e-beta} for an example for Bott-Samelson varieties.
It seems a miracle that such highly complicated birational maps from $\Cset^n$ to $\Cset^n$
in fact 
transform one iterated Poisson Ore extension to another (see \cite[Appendix A]{Balazs:thesis} for 
some direct computations related to Example \ref{ex-e-beta}). One may thus ask: are the 
changes of coordinates between the Poisson Ore charts on $Z$ compositions of some simpler 
{\it one-step mutations of iterated Poisson Ore extensions}? Can one 
start with some particular {\it Poisson Ore seed} on $(Z, \pi)$ and obtain all the other Poisson Ore charts
on $Z$ through finitely many steps of mutations? 

Note first
that although similar in spirit, the changes of coordinates between the Poisson Ore
charts on $Z$ are not cluster transformations as
in the theory of cluster algebras. Indeed, a mutation in our sense produces a new  iterated Poisson Ore extension
from  an old one, while those in cluster algebra theory mutate between polynomial Poisson algebras with {\it log-canonical
Poisson brackets} \cite{GSV:book}. 

On the other hand, results in this paper on Bott-Samelson varieties provide testing ground for the questions above.
Indeed, for a Bott-Samelson variety $\Zu$, regarded as being glued together by the $2^n$ coordinate charts in the
atlas $\calA$ in \eqref{eq-calA}, it is not hard to see that one can connect any coordinate chart $\O^\gamma$ with the 
special chart $\O^\bfu$ by a chain of {\it adjacent} charts, and that the changes of coordinates between adjacent charts
are determined by the $T$-action and the vector fields $\sigma_i$ in \eqref{eq-Vi} on sub-Bott-Samelson varieties of $\Zu$.
On the other hand, by \eqref{eq-zizj-1} in Lemma \ref{le-zizj-1} and
\eqref{eq-zizj-th} in 
Theorem {\ref{th-zizj}}, the $T$-action and the vector fields $\sigma_i$ can be read off 
from the symmetric Poisson CGL extension 
$(\Cset[z_1, \ldots, z_n], \{\, , \, \}_\bfu)$. This 
suggests that 
{\it all the information needed to build up $\Zu$ as an algebraic variety by 
gluing $2^n$ copies of $\Cset^n$ is encoded in the symmetric Poisson CGL extension 
$(\Cset[z_1, \ldots, z_n], \{\, , \, \}_\bfu)$}. While such a statement is not precise, we will 
explain in $\S$\ref{subsec-degeneration} a precise statement from \cite{Lu-Peng:degeneration, Peng:thesis}
that relates degenerations of the Poisson algebra $(\Cset[z_1, \ldots, z_n], \{\, , \, \}_\bfu)$ with toric
degenerations of $\Zu$.

For Poisson varieties $(Z, \pi)$ that can be covered by symmetric Poisson CGL extensions, such as $(G/B, \pi_1)$ and
$(G/T, \pi_{\rm st}^\prime)$ discussed in $\S$\ref{subsec-building-blocks}, one may consider the
cluster algebra structure on each Poisson Ore chart defined by
the corresponding symmetric Poisson CGL extension using the
Goodearl-Yakimov theory \cite{GY:Poi} and
ask how they glue together to give some {\it global cluster structure} on $Z$.
This project will be taken up in the future. We point out for now that the
explcit result from this paper on the symmetric Poisson CGL extensions coming from generalized Bruhat cells
will be crucial for such a project.

\subsection{Poisson and toric degenerations of Bott-Samelson varieties}\label{subsec-degeneration}
Suppose that $Z$ is an $n$-dimensional smooth projective variety over $\Cset$ equipped
with a finite algebraic atlas of the form
\[
\calA = \{(\phi^\gamma: \Cset^n \longrightarrow \phi^\gamma(\Cset^n) \subset Z:\; \gamma \in \Upsilon\}.
\]
One can then {\it degenerate the pair $(Z, \calA)$}, i.e., degenerate the variety $Z$ by
rescaling the coordinates in the charts in $\calA$. Such a theory is developed in 
\cite{Lu-Peng:degeneration, Peng:thesis}, where the degenerations
are expressed in terms of the  {\it tropicalizations} and the {\it initial forms}
of the birational maps from $\Cset^n$ to $\Cset^n$ which are the 
changes of coordinates of the atlas $\calA$. In particular, one may require the degeneration to be toric, i.e.,
the central fiber be a toric variety. It is shown in \cite{Lu-Peng:degeneration, Peng:thesis}
that one can associate a cone $\calC_{\rm toric}(Z, \calA)$ in $\Rset^n$,
called the {\it toric degeneration cone of $(Z, \calA)$}, whose set of integral points parametrizes all the
toric degenerations of $(Z, \calA)$ up to equivalences. 

Turning to the Bott-Samelson variety $Z_\bfu$ with the atlas $\calA$ in \eqref{eq-calA}, the toric degeneration
cone $\calC_{\rm toric}(Z_\bfu, \calA)$, as well as all toric degenerations of $(Z_\bfu, \calA)$, are described explicitly in \cite{Lu-Peng:degeneration, Peng:thesis}. In particular, the zero fiber of
every toric degeneration is shown to be always isomorphic to the {\it Bott-Tower} 
$Z_\bfu^0$, which is a smooth projective toric variety introduced by M. Grossberg and Y. Karshon \cite{GK:BTower}
in their study of extended characters of representations of compact Lie groups. Closely related
 to this paper is the statement proved in \cite{Lu-Peng:degeneration, Peng:thesis} that
\begin{equation}\label{eq-C-C}
\calC_{\rm toric}(Z_\bfu, \calA) \cong \calC_{{\rm log-can}}(\Cset[z_1, z_2, \ldots, z_n], \{\, , \, \}_\bfu),
\end{equation}
by an element in $GL(n, \mathbb{Z})$, 
where $\calC_{{\rm log-can}}(\Cset[z_1, z_2, \ldots, z_n], \{\, , \, \}_\bfu)$ is the {\it log-canonical degeneration cone}
of the polynomial Poisson algebra $(\Cset[z_1, z_2, \ldots, z_n], \{\, , \, \}_\bfu)$.
Here the notion of the 
log-canonical degeneration cone for 
a polynomial Poisson algebra is introduced in \cite{Anton-I} by A. Alexseev and I. Davydenkova in their study of
degenerations of polynomial Poisson brackets to their log-canonical terms. The isomorphism in \eqref{eq-C-C}
essentially says that the ``directions" in which the Bott-Samelson
variety $\Zu$ can be degenerated to a toric variety via coordinate rescalings are 
the same as those in which the Poisson polynomial algebra $(\Cset[z_1, \ldots, z_n], \{\,, \, \}_\bfu)$
can be degenerated to its log-canonical term. The proof of \eqref{eq-C-C} in 
\cite{Lu-Peng:degeneration, Peng:thesis}
uses in
a very essential way the explicit formulas for the Poisson algebra $(\Cset[z_1, z_2, \ldots, z_n], \{\, , \, \}_\bfu)$
as given in Theorem \ref{th-intro} and Theorem \ref{th-zizj}.

\subsection{Acknowledgments} 
Parts of the paper, notably $\S$\ref{sec-coor-II}, $\S$\ref{subsec-Z}, and $\S$\ref{subsec-examples}, are from the first author's Mphil thesis. The authors would like to thank Allen Knutson and Victor Mouquin for helpful discussions.
This work was partially supported by a University of Hong Kong Post-graduate Studentship and by the Research Grants Council of the Hong Kong SAR, China (GRF HKU 704310 and 703712). 

\subsection{Notation}\label{notations}
Continuing with the notation from $\S$\ref{intro-intro}, let 
$\mathfrak{g}=\mathfrak{h} + \sum_{\alpha\in\Delta}\mathfrak{g}_\alpha$
be the root decomposition of $\mathfrak{g}$ with respect to $\h$. For $\alpha\in\Delta$, 
let $h_\alpha$ be the unique element in $[ \mathfrak{g}_\alpha ,\mathfrak{g}_{-\alpha} ]$ such that $\alpha(h_\alpha)=2$, and let
$\alpha^\vee: \Cset^\times \rightarrow T$ be the co-character of $T$ defined by $h_\alpha$. 
Let $\Delta_+ \subset \Delta$ be the set of positive roots
determined by $\b$, and let 
$\b_- = \h + \sum_{\alpha\in \Delta_+}\g_{-\alpha}$. The Borel subgroup of $G$ with Lie algebra
$\b_-$ is denoted by $B_-$.

Let $\alpha \in \Delta_+$.
If $e_\alpha \in \g_\alpha$ and $e_{-\alpha} \in \g_{-\alpha}$ are such that $[e_\alpha ,e_{-\alpha}]=h_\alpha$,
we call $\{h_\alpha, e_\alpha, e_{-\alpha}\}$ an $\sl(2, \Cset)$-triple for $\alpha$. Clearly, any non-zero
$e_\alpha \in \g_\alpha$
uniquely determines an $\sl(2, \Cset)$-triple $\{h_\alpha, e_\alpha, e_{-\alpha}\}$, and every other 
$\sl(2, \Cset)$-triple for $\alpha$ is of the form $\{h_\alpha, \lambda e_\alpha, \lambda^{-1}e_{-\alpha}\}$
for a unique $\lambda \in \Cset$. Given 
an $\sl(2, \Cset)$-triple $\{h_\alpha, e_\alpha, e_{-\alpha}\}$,
denote by
$\theta_\alpha: \mathfrak{sl}(2,\mathbb{C})\rightarrow \mathfrak{g}$ the Lie algebra homomorphism defined by
\[
\theta_\alpha: \;\;\; 
\begin{pmatrix}
1 & 0 \\
0 & -1
\end{pmatrix}\longmapsto h_\alpha,\qquad
\begin{pmatrix}
0 & 1 \\
0 & 0
\end{pmatrix}\longmapsto e_\alpha,\qquad
\begin{pmatrix}
0 & 0 \\
1 & 0
\end{pmatrix}\longmapsto e_{-\alpha},
\]
and denote also by  $\theta_\alpha: \mathrm{SL}(2,\mathbb{C})\rightarrow G$ the corresponding Lie group homomorphism, so that
\[
\alpha^\vee(t)=\theta_\alpha\left(
\begin{pmatrix}
t & 0 \\
0 & t^{-1}
\end{pmatrix}\right), \hs t\in\mathbb{C}^\times.
\]
An $\sl(2, \Cset)$-triple $\{h_\alpha, e_\alpha, e_{-\alpha}\}$ for $\alpha \in \Delta_+$ also gives rise to the one-parameter subgroups $u_{\pm \alpha}: \Cset \to G$ via 
\[
u_\alpha(z)=\theta_\alpha\left(
\begin{pmatrix}
1 & z \\
0 & 1
\end{pmatrix}\right)=\exp (ze_\alpha),\quad
u_{-\alpha}(z)=\theta_\alpha\left(
\begin{pmatrix}
1 & 0 \\
z & 1
\end{pmatrix}\right)=\exp (ze_{-\alpha}), \hs z \in \Cset.
\]
Let $W = N_G(T)/T$ be again the Weyl group of $(G, T)$.  
For $\alpha \in \Delta_+$, let $s_\alpha \in W$ be the reflection in $W$ determined by $\alpha$,
and if $\{h_\alpha, e_\alpha, e_{-\alpha}\}$ is an $\sl(2, \Cset)$-triple for $\alpha$, let
$\dot{s}_\alpha$
be the representative of $s_\alpha$ in $N_G(T)$ given by
\begin{equation}\label{eq-dots}
\dot{s}_\alpha=u_{\alpha}(-1)u_{-\alpha}\left(1\right)u_{\alpha}\left(-1\right) \in N_G(T).
\end{equation}

For a complex algebraic torus $\T$ with Lie algebra $\t$, 
we use the same notation for an element $\lambda \in {\rm Hom}(\T, \Cset^\times)$ and its
differential at the identity element of
$\T$, which is an element in $\t^*$. The values of $\lambda$ on $t \in \T$ and on $x \in \t$ are respectively denoted
as $t^\lambda \in \Cset^\times$ and $\lambda(x) \in \Cset$.

\section{Definition of the Poisson structure $\pi_n$ on $\Zu$}
\label{piZw}

\subsection{The standard semi-simple Poisson Lie group $(G, \pist)$}\label{subsec-pist}
Recall from \cite{CP, Etingof-Schiffmann} that if $L$ is a Lie group, a Poisson bivector field $\piL$ on $L$ is said to be multiplicative 
if the map 
\[
(L\times L, \; \piL\times \piL)\longrightarrow (L,\, \piL):\;\; (l_1,\, l_2)\longmapsto l_1l_2, \hs l_1, l_2 \in L,
\]
is Poisson, where $\piL \times \piL$ is the product Poisson structure on $L \times L$.
A Poisson Lie group is a pair $(L, \piL)$, where $L$ is a Lie group and $\piL$ is a 
multiplicative Poisson bivector field on $L$. A Poisson Lie subgroup of a Poisson Lie group $(L,\piL)$ is
a Lie subgroup $L_1$ of $L$ which is also a 
Poisson submanifold with respect to $\piL$, and in this case $(L_1, \piL|_{\sL_1})$, or simply denoted as
$(L_1, \piL)$, is a Poisson Lie group. 

Let $G$ be a connected complex semi-simple Lie group and let the 
notation be as in $\S$\ref{notations}. Fix, furthermore, a symmetric non-degenerate  invariant bilinear form $\lara$ on $\mathfrak{g}$, and denote also by $\lara$ the  induced bilinear form on $\mathfrak{h}^*$.
Define $\Lambda \in \wedge^2 \g$ by
\[
\Lambda=\sum_{\alpha \in\Delta_+}\frac{\langle \alpha,\alpha \rangle}{2}e_{-\alpha} \wedge e_{\alpha} \in \wedge^2\mathfrak{g},
\]
where for each $\alpha \in \Delta_+$, $\{h_\alpha, e_\alpha, e_{-\alpha}\}$ is an 
$\sl(2, \Cset)$-triple for $\alpha$.   Note that for any $\alpha \in \Delta_+$,  the element
$e_{-\alpha} \wedge e_\alpha \in \wedge^2 \g$ stays the same if the 
$\sl(2, \Cset)$-triple $\{h_\alpha, e_\alpha, e_{-\alpha}\}$ is changed to
$\{h_\alpha, \lambda e_\alpha, \frac{1}{\lambda}e_{-\alpha}\}$ for $\lambda \in \Cset^\times$.   Consequently, 
the element $\Lambda \in \wedge^2 \g$ depends on $\lara$ but {\it not} on the choices of the $\sl(2, \Cset)$-triples
for the positive roots. 
Let $\pist$ be the bivector field on $G$ given by
\[
\pist(g) =l_g (\Lambda)-r_g(\Lambda),\hs g \in G,
\] 
where for $g \in G$, $l_g$ and $r_g$ respectively denote the left and right translations on $G$ by $g$.
Then $(G,\pist)$ is a Poisson Lie group, called a {\it standard complex semi-simple Poisson Lie group}
\cite[$\S$4.4]{Etingof-Schiffmann}.
Moreover, the Poisson structure $\pist$ is invariant under the action of $T$ by left translation, and the
$T$-orbits of symplectic leaves, also called $T$-leaves, of $\pist$ are precisely the so-called double Bruhat cells $(BuB) \cap (B_-vB_-)$, where $u, v \in W$ (see \cite{Hodges-Levasseur, KZ}).
In particular, every $BuB$, where $u \in W$, is a Poisson submanifold of $(G, \pist)$, and
every parabolic subgroup $P$ of $G$ containing $B$, being a union of $(B, B)$-double cosets in $G$,
is a Poisson Lie subgroup of $(G, \pist)$. Similar statements hold if $B$ is replaced by $B_-$.

We state another important property of $(G, \pist)$: let $\alpha$ be a simple root and
consider the group homomorphism $\theta_\alpha: SL(2, \Cset) \to G$ defined in $\S$\ref{notations} corresponding to any choice
of an $sl(2, \Cset)$-triple $\{h_\alpha, e_\alpha, e_{-\alpha}\}$ for $\alpha$.
Equip $SL(2, \Cset)$ with the multiplicative Poisson structure
\begin{equation}\label{eq-pi-alpha}
\pi_{\scriptscriptstyle{SL(2, \Cset)}} (g) = l_g(\Lambda_0) - r_g(\Lambda_0), \hs g \in SL(2, \Cset),
\end{equation}
where $\displaystyle \Lambda_0 = \left(\begin{array}{cc} 0 & 0 \\ 1 & 0\end{array}\right) \wedge
\left(\begin{array}{cc} 0 & 1 \\ 0 & 0\end{array}\right) \in \wedge^2 \sl(2, \Cset).$
Then \cite{k-s:quantum}
\begin{equation}\label{theta-Poisson}
\theta_\alpha: \;\; \left(SL(2, \Cset), \; \frac{\la \alpha, \alpha \ra}{2}\pi_{\scriptscriptstyle{SL(2, \Cset)}}\right) \longrightarrow (G, \, \pist)
\end{equation}
is a Poisson map. 
It follows that $\theta_\alpha(SL(2, \Cset))$ is a Poisson Lie subgroup of $(G, \pist)$. 
Moreover,   let
$g = u_{-\alpha}(z)$ and $g' = u_{\alpha}(z)\dot{s}_\alpha$, where $z \in \Cset$.
Then 
\begin{align}\label{eq-lpi-1}
\pist(g) &= \frac{\la \alpha, \alpha \ra}{2} l_g(z h_\alpha \wedge e_{-\alpha}), \\
\label{eq-lpi}
\pist(g')& = \frac{\la \alpha, \alpha \ra}{2} l_{g'}(z h_\alpha \wedge e_{-\alpha} -2e_\alpha
\wedge e_{-\alpha}) = \frac{\la \alpha, \alpha \ra}{2} r_{g'}(z e_\alpha \wedge h_\alpha + 2 e_\alpha \wedge e_{-\alpha}).
\end{align}

\subsection{The definition of the Poisson structure $\pi_n$ on $\Zu$}
\label{piBS}
Recall that given a Poisson Lie group $(L,\piL)$ and a Poisson manifold $(Y, \piY)$, 
a left Lie group action $\sigma: L \times Y \to Y$ of $L$ on $Y$ is said to be a Poisson action if
$\sigma$ is a Poisson map from the product Poisson manifold $(L \times Y, \; \piL \times \piY)$ to $(Y, \piY)$.
Right Poisson actions of Poisson Lie groups are similarly defined.

Let $(Q, \piQ)$ be a Poisson Lie group, let $(X, \piX)$ be a Poisson manifold with a
right Poisson action by $(Q, \piQ)$, and let $(Y, \piY)$ a Poisson submanifold with a left Poisson action
by $(Q, \piQ)$. Define the right action of $Q$ on $X \times Y$ by 
\[
(x, \, y) \cdot q = (xq, \, q^{-1}y), \hs x \in X, \, y \in Y, \, q \in Q,
\]
and assume that the quotient space of $X \times Y$ by $Q$, denoted by $X \times_Q Y$, is a smooth manifold. 
Then (see \cite[$\S$7.1]{Lu-Mou:mixed} and \cite{STS2}) the direct product Poisson structure
$\piX \times \piY$ on $X \times Y$ projects to a well-defined Poisson structure on $X \times_Q Y$.

\begin{example}\label{ex-LQY}
{\em Let $(Q, \piQ)$ be a closed Poisson Lie subgroup of a Poisson Lie group $(L, \piL)$, and let
$(Y, \piY)$ be a Poisson manifold with a left Poisson action by $(Q, \piQ)$. Consider the quotient manifold
$Z = L \times_Q Y$,
where $Q$ acts on $L$ by right translation. Then $Z$  has the Poisson structure $\piZ$
that is the projection 
to $Z$ of 
the direct product Poisson structure $\piL \times \piY$ on $L \times Y$. Denoting the image in $Z$ of $(l, y) \in L \times Y$ by
$[l, y]$, it follows from the multiplicativity of $\piL$ that
the left action of $L$ on $Z$ given by
\begin{equation}\label{eq-g-dot}
l \cdot [l_1, \; y] = [ll_1, \; y], \hs l, \,l_1 \in L, \, y \in Y,
\end{equation}
is a Poisson action of the Poisson Lie group $(L, \piL)$ on the Poisson manifold $(Z, \, \piZ)$. Moreover, since $\piL(e) = 0$,
where $e$ is the identity element of $L$, the inclusion $Y \hookrightarrow L \times Y, y \mapsto (e, y)$, $y \in Y$,
is a Poisson embedding of $(Y, \piY)$ into $(L \times Y, \, \piL \times \piY)$. Consequently, 
\[
Y \hookrightarrow Z, \;\; y \longmapsto [e, \; y], \hs y \in Y,
\]
is a Poisson embedding of $(Y, \piY)$ into the Poisson manifold $(Z, \piZ)$.
\hfill $\diamond$}
\end{example}

Consider now the standard semisimple Poisson Lie group $(G, \pist)$ in $\S$\ref{subsec-pist}.
Let $\bfu=(s_1,\ldots, s_n)$ be any sequence of simple reflections in the Weyl group $W$. Then for each $1 \leq i \leq n$, 
the parabolic subgroup $P_{s_i} = B \cup Bs_iB$ is a Poisson Lie subgroup of $(G, \pist)$. By taking 
$(L, \piL) = (P_{s_i}, \pist)$ and $Q = B$ in
Example \ref{ex-LQY} and repeat the construction therein, one sees that the direct product Poisson structure
$\pi_{\rm st}^n$, regarded as a Poisson structure on the product manifold $P_{s_1} \times \cdots \times P_{s_n}$, projects to a well-defined 
Poisson structure, denoted by $\pi_n$, on the Bott-Samelson variety $\Zu$.
It also follows from Example \ref{ex-LQY} that the left action of $P_{s_1}$ on $\Zu$ given by 
\begin{equation}\label{eq-P-on-Zu}
p \cdot [p_1, p_2, \ldots, p_n] = [pp_1, \, p_2, \ldots, \, p_n], \hs p \in P_{s_1}, \, p_j \in P_{s_j}, \; 1 \leq j \leq n,
\end{equation}
is a Poisson action of the Poisson group $(P_{s_1}, \pist)$ on the Poisson manifold $(\Zu, \pi_n)$.
In particular, since $\pist(t) = 0$ for $t \in T$, the action of $T$ on $\Zu$ via (\ref{eq-P-on-Zu}) is by 
Poisson isomorphisms of $\pi_n$.

\subsection{${\mathbb P}^1$-extensions}\label{subsec-P-1}
To prepare for the calculation of the Poisson structure $\pi_n$ in coordinates, we first look at a special case of
Example \ref{ex-LQY}: let $(Y, \piY)$ be 
a Poisson manifold with a left Poisson action $\sigma$ by the Poisson Lie subgroup $(B, \pist)$ of $(G, \pist)$,
and let $\alpha$ be a simple root. One then has  the quotient manifold
$Z = P_{s_\alpha} \times_B Y$, which fibers over $P_{s_\alpha}/B \cong {\mathbb P}^1$ with fibers diffeomorphic to $Y$.
 Let $\piZ$ denote the projection to $Z$ of the
product Poisson structure $\pist \times \piY$ on $P_{s_\alpha} \times Y$. 
Choose any non-zero $e_\alpha \in \g_\alpha$, giving rise to the $\sl(2, \Cset)$-triple
$\{h_\alpha, e_\alpha, e_{-\alpha}\}$ for $\alpha$, and let
the notation be as in $\S$\ref{notations}. 
Consider the two open subsets 
\[
Z_- =\{[u_{-\alpha}(z), \, y]: \, z \in \Cset, \,y \in Y\} \hs \mbox{and} \hs
Z_+ =\{[u_{\alpha}(z) \ds_\alpha, \, \, y]: z \in \Cset, y \in Y\}
\]
of $Z$ with parametrizations 
\begin{align*}
&\psi_-: \;\; \Cset \times Y \longrightarrow Z_-,\;\; \psi_-(z, y) = [u_{-\alpha}(z), \; y], \\
&\psi_+: \;\; \Cset \times Y \longrightarrow Z_+,\;\;\psi_+(z, y) =[u_{\alpha}(z)\ds_\alpha, \; y].
\end{align*}
We will compute $\psi_-^{-1}(\piZ)$ and $\psi_+^{-1}(\piZ)$ as bi-vector fields on $\Cset \times Y$.
For $x \in \b$, let $\eta_x$ be the vector field on $Y$ given by 
$\eta_x(y) = \frac{d}{dt}|_{t = 0} \exp (tx) y$ for $y \in Y$.
In the statement of the following Lemma \ref{le-Z-Pi-1}, we use the obvious way of 
viewing vector fields on $\Cset$ and on $Y$ 
as that on $\Cset \times Y$.

\begin{lemma}\label{le-Z-Pi-1}
With the notation as above, one has
\begin{align}\label{eq-Pi-zy-1}
\psi_-^{-1}(\piZ)(z, y) &= -\frac{\la \alpha, \alpha \ra}{2} z \frac{d}{dz} 
\wedge \eta_{h_\alpha}(y) + \piY(y),\\
\label{eq-Pi-zy-2}
\psi_+^{-1}(\piZ)(z, y) &= 
\frac{\la \alpha, \alpha \ra}{2} \frac{d}{dz} \wedge \left(z \eta_{h_\alpha}(y) -2
\eta_{e_\alpha}(y)\right) + \piY(y).
\end{align}
\end{lemma}

\begin{proof} For $g \in P_{s_\alpha}$ and $y \in Y$, let
\begin{align*}
&\lambda_g:\;\;  Z \longrightarrow  Z: \;\; [p, \; y'] \longmapsto [gp, \; y'], \hs p \in P_{s_\alpha}, \, y' \in Y, \\
&\rho_y: \;\; P_{s_\alpha} \longrightarrow Z:\;\;  p \longmapsto [p, \; y], \hs p \in P_{s_\alpha}.
\end{align*}
Fix $z \in \Cset$ and $y \in Y$, and let $g = u_{-\alpha}(z) \in P_{s_\alpha}$ and 
$q =[u_{-\alpha}(z), \,y] = \lambda_g([e, \, y]) \in Z$. 
By Example \ref{ex-LQY},   
$\piZ(q) = \lambda_g (\piZ([e, y])) + \rho_y (\pist(g))$. 
Using  (\ref{eq-lpi-1}), one has
\[
\piZ(q)  = \lambda_g (\piZ([e, y])) + \frac{\la \alpha, \alpha \ra}{2}(\rho_y l_g) (z h_\alpha \wedge e_{-\alpha}) =
\lambda_g (\piZ([e, y])) + \frac{\la \alpha, \alpha \ra}{2}(\lambda_g \rho_y) (z h_\alpha \wedge e_{-\alpha})
\]
and thus
\[
(\psi_-^{-1}(\piZ))(z, y)=\psi_-^{-1}(\piZ(q)) =
(\psi_-^{-1} \circ \lambda_g) (\piZ([e, y])) + \frac{\la \alpha, \alpha \ra}{2} (\psi_-^{-1} \lambda_g \rho_y) (z h_\alpha \wedge e_{-\alpha}).
\]
Since the inclusion $(Y, \piY) \hookrightarrow (Z, \piZ): y' \mapsto [e, y']$ is Poisson, 
$(\psi_-^{-1} \circ \lambda_g) (\piZ([e, y]))=\piY(y)$.
Direct calculations give 
\[
(\psi_-^{-1} \lambda_g \rho_y) (h_\alpha) = \eta_{h_\alpha}(y) \hs \mbox{and} \hs
(\psi_-^{-1} \lambda_g \rho_y) (e_{-\alpha})=\frac{d}{dz}.
\]
 One thus has
\eqref{eq-Pi-zy-1}. 
Similarly, for $z \in \Cset$ and $y \in Y$, letting $g' = u_{\alpha}(z)\ds_\alpha $ and using (\ref{eq-lpi}), one has
\[
\psi_+^{-1}(\piZ)(z, y)  = \piY(y) + \frac{\la \alpha, \alpha \ra}{2}(\psi_+^{-1} \lambda_{g'} \rho_y) 
((zh_\alpha - 2e_\alpha)\wedge e_{-\alpha}).
\]
Since $(\psi_+^{-1} \lambda_{g'} \rho_y)(h_\alpha) =\eta_{h_\alpha}, \, $ 
$(\psi_+^{-1} \lambda_{g'} \rho_y)(e_\alpha) =\eta_{e_\alpha}$, and
$(\psi_+^{-1} \lambda_{g'} \rho_y)(e_{-\alpha}) = -\frac{d}{dz}$, one has \eqref{eq-Pi-zy-2}.
\end{proof}

\section{The Poisson structure $\pi_n$ in affine coordinate charts, I}\label{sec-coors}
%\subsection{Notation}\label{nota-fix-simple}
Throughout $\S$\ref{sec-coors}, we fix a sequence $\bfu =(s_1,\ldots, s_n)$ of simple reflections in $W$, and let $\Zu$ be the Bott-Samelson variety associated to $\bfu$. Recall that $\Gamma$ denotes the set of all simple roots. 
For $1 \leq j \leq n$, let $\alpha_j\in\Gamma$ be such that $s_j=s_{\alpha_j}$.
To define local coordinates on $\Zu$, we also fix a root vector $e_\alpha$ for each $\alpha \in \Gamma$ and let $e_{-\alpha} 
\in \g_{-\alpha}$ be the unique element such that $[e_\alpha, e_{-\alpha}] = h_\alpha$. 
One then (see $\S$\ref{notations}) has the one-parameter subgroups $u_{\pm \alpha}: \Cset \to G$ for each $\alpha \in \Gamma$ and 
the representative $\dot{s}_\alpha \in N_G(T)$ for the simple reflection $s_\alpha \in W$. 

\subsection{Affine coordinate charts on $\Zu$}\label{coordcharts}
Let
\[
\Upu=\{e,s_1\}\times\{e,s_2\}\times\cdots \{e,s_n\},
\]
where $e$ denotes the identity element of $W$. Elements in $\Upu$ will be 
called {\it subexpressions} of $\bfu$. When $\gamma = \bfu$, we say that $\gamma$ is the {\it full subexpression}
of $\bfu$. For
$\gamma = (\gamma_1, \gamma_2, \ldots, \gamma_n) \in \Upu$,
let $\gamma^0 = e$ and 
$\gamma^i = \gamma_1\gamma_2 \cdots \gamma_i \in W$ for $1 \leq i \leq n$.

As a subgroup of $P_{s_1}$, the maximal torus $T$ of $G$ acts on $\Zu$ via (\ref{eq-P-on-Zu}),
with the fixed point set $(\Zu)^T =\{[\dot{\gamma}_1, \dot{\gamma}_2, \ldots, \dot{\gamma}_n]: 
(\gamma_1,\gamma_2, \ldots, \gamma_n)\in\Upu\}$, where $\dot{e}=e$. 
For each $\ga = (\ga_1, \ga_2, \ldots, \ga_n) \in \Upu$,
let $\mathcal{O}^\gamma \subset \Zu$ be the image of the embedding
$\Phi^\gamma:\mathbb{C}^n\rightarrow \Zu$ given by
\begin{equation}\label{eq-Phi-gamma}
\Phi^\gamma(z_1,\ldots, z_n)= [u_{-\gamma_1(\alpha_1)}(z_1)\dot{\gamma}_1,\, u_{-\gamma_2(\alpha_2)}(z_2)\dot{\gamma}_2,\,
\ldots, \,u_{-\gamma_n(\alpha_n)}(z_n)\dot{\gamma}_n].
\end{equation}
The parametrization $\Phi^\gamma$ of $\mathcal{O}^\gamma$ by $\mathbb{C}^n$ depends on the choice of the root vectors 
$\{e_\alpha: \alpha \in \Gamma\}$ for the simple
roots, but different choices of such root vectors only result in re-scalings of the
coordinate functions. In particular, the open subset $\mathcal{O}^\gamma$ of $\Zu$ is canonically defined.
It is also easy to see that each $\mathcal{O}^\gamma$ is $T$-invariant with
\begin{equation}\label{hvectorfield}
t \cdot \Phi^\gamma(z_1, z_2, \ldots, z_n) = \Phi^\gamma\left(t^{-\gamma^1(\alpha_1)}z_1, \;
t^{-\gamma^2(\alpha_2)}z_2, \;\ldots, \;t^{-\gamma^n(\alpha_n)}z_n\right), 
\end{equation}
where  $t \in T$ and $(z_1, z_2, \ldots, z_n) \in {\mathbb C}^n.$
Note also that $\bigcup_{\gamma\in\Upsilon_{\textbf{w}}}\mathcal{O}^\gamma=\Zu$, i.e.,
$\Zu$ is covered by the $2^n$ $T$-invariant affine coordinate charts
$\{(\Phi^\gamma: \Cset^n \to \mathcal{O}^\gamma): \gamma \in \Upu\}$, which we will also abbreviated as
the affine charts $\{\O^\gamma: \gamma \in \Upu\}$.
.

\subsection{The Poisson structure $\pi_n$ in coordinates, I}\label{subsec-coor-I}
For each $\gamma \in \Upu$, we now give 
our first formula for the Poisson structure $\pi_n$ on $\Zu$ in the coordinates $(z_1, z_2, \ldots, z_n)$ on 
$\mathcal{O}^\gamma$ given in \eqref{eq-Phi-gamma}. A more detailed formula, expressing each Poisson bracket $\{z_i, z_k\}$, 
where  $1 \leq i < k \leq n$, 
as a polynomial with coefficients explicitly given in terms of the structure constants of the Lie algebra
$\g$, will be given in $\S$\ref{sec-coor-II}. 

For $1\leq i \leq n-1$, let $\sigma_i$ be the holomorphic vector field on the Bott-Samelson variety $Z_{(s_{i+1},\ldots, s_n)}$ given by
\begin{equation}\label{eq-Vi}
\sigma_i(p)=\frac{d}{dt}|_{t=0} ((\exp (t e_{\alpha_i}))\cdot p),\hs \hs p\in Z_{(s_{i+1},\ldots, s_n)},
\end{equation}
where $\cdot$ denotes the left action of $B \subset P_{s_i}$ on $Z_{(s_{i+1},\ldots, s_n)}$ by left translation
(see (\ref{eq-P-on-Zu})).

\begin{lemma}\label{le-zizj-1}
Let $\gamma \in \Upu$. In the coordinates $(z_1,\ldots z_n)$ on
the affine chart ${\mathcal{O}^\gamma}$ given in \eqref{eq-Phi-gamma}, the Poisson structure $\pi_n$  on $\Zu$ is given by,
\begin{equation}\label{eq-zizj-1}
\{z_i , z_k \} = \begin{cases}
\langle \gamma^i(\alpha_i), \, \gamma^k(\alpha_k) \rangle z_iz_k,  &\mbox{if }\gamma_i=e \\
-\langle \gamma^i(\alpha_i), \, \gamma^k(\alpha_k) \rangle z_iz_k -\langle \alpha_i, \alpha_i \rangle \sigma_i(z_k) &\mbox{if }\gamma_i=s_i
\end{cases},
\quad 1\leq i < k \leq n,
\end{equation}
where $\sigma_i(z_k)$ denotes the action of the vector field $\sigma_i$ on $z_k$ as a local function
on $Z_{(s_{i+1},\ldots, s_n)}$ via the parametrization 
\[
\mathbb{C}^{n-i} \ni (z_{i+1}, \, \ldots, \, z_n) \longmapsto [u_{-\gamma_{i+1}(\alpha_{i+1})}(z_{i+1}) \dot{\gamma}_{i+1}, \, \ldots, u_{-\gamma_n(\alpha_n)}(z_n) \dot{\gamma}_n].
\]
\end{lemma}

\begin{proof}
Identify $\O^\ga \cong \Cset \times \O^{\ga'}$, where
$\ga' = (\ga_2, \ldots, \ga_n) \in \Upsilon_{\bfu'}$ and $\bfu'=(s_2, \ldots, s_n)$. 
Equip $\O^{\ga'}$ with the Poisson structure $\pi_{n-1}$ on $Z_{(s_2, \ldots, s_n)}$. One has, by Lemma \ref{le-Z-Pi-1},
\begin{equation}\label{eq-Pi-Pi}
\pi_n = \begin{cases} -\frac{\la \alpha_1, \alpha_1\ra}{2} z_1 \frac{d}{dz_1} \wedge \eta_1 + \pi_{n-1}, & \;\; \mbox{if}\;\;
\ga_1 = e,\vspace{.1in}\\
\frac{\la \alpha_1, \alpha_1\ra }{2} \frac{d}{dz_1} \wedge (z_1\eta_1 -
2\sigma_1) + \pi_{n-1}, & \;\; \mbox{if} \;\;\ga_1 = s_1,\end{cases}
\end{equation}
where $\eta_1$ is the holomorphic vector field on $Z_{(s_2, \ldots, s_n)}$ given by
\[
\eta_1(q) = \frac{d}{dt}|_{t=1}(\check{\alpha}_1(t) \cdot q), \hs q \in Z_{(s_2, \ldots, s_n)}.
\]
By (\ref{hvectorfield}),  the vector field $\eta_1$ is given in 
the coordinates $(z_2, \ldots, z_n)$ on $\O^{\gamma'}$ by 
\[
\eta_1= \sum_{k=2}^n(-\ga_2 \cdots \ga_k(\alpha_k))(h_{\alpha_1}) z_k \frac{\pa}{\pa z_k} = 
-\sum_{k=2}^n \frac{2\la \ga^1(\alpha_1), \ga^k(\alpha_k)\ra}{\la \alpha_1, \alpha_1 \ra} z_k \frac{\pa}{\pa z_k}.
\]
Lemma \ref{le-zizj-1}) now follows  by repeatedly using \eqref{eq-Pi-Pi}.
\end{proof}

\begin{example}\label{ex-1}
{\em 
Consider $G = SL(3, \Cset)$ with the standard choices of $B$ and $B_-$ consisting respectively of
upper triangular
and lower triangular matrices in $SL(3, \Cset)$, and let the bilinear form $\lara$ on $\sl(3, \Cset)$ be given by
$\la X, Y\ra = {\rm tr}(XY)$ for $X, Y \in \sl(3, \Cset)$. Denote the two simple roots by $\alpha_1$ and $\alpha_2$
choose root vectors $e_{\alpha_1} = E_{12}$ and $e_{\alpha_2} = E_{23}$, where $E_{ij}$ has  $1$ at the $(i,j)$-entry and $0$ everywhere else. Let $\bfu = (s_{\alpha_1}, s_{\alpha_2}, s_{\alpha_1})$. Using Lemma \ref{le-zizj-1}, one can compute
directly the Poisson structure $\pi_3$ on $\Zu$ in any of the eight affine coordinate charts with coordinates $(z_1, z_2, z_3)$. For example, 
for $\gamma = \bfu$,  one has
\begin{equation}\label{eq-ex-0}
\{z_{1},z_{2}\} = -z_{1}z_{2}, \hs  
\{z_{1},z_{3}\} = z_{1}z_{3}- 2, \hs 
\{z_{2},z_{3}\} = -z_{2}z_{3}, 
\end{equation}
and for $\gamma =
(s_{\alpha_1}, e, e) \in \Upu$, one has 
\begin{equation}\label{eq-ex-1}
\{z_{1},z_{2}\} = z_{1}z_{2}, \hs  
\{z_{1},z_{3}\} = -2z_{1}z_{3}+2z_{3}^2, \hs 
\{z_{2},z_{3}\} = -z_{2}z_{3}. 
\end{equation}
\hfill $\diamond$
}
\end{example}

\subsection{Some log-canonical charts for $\pi_n$}\label{subsec-log} Let $\ga \in \Upu$. We say that the affine coordinate chart $\O^\ga$ 
of $\Zu$ is
{\it log-canonical} for the Poisson structure $\pi_n$, or that 
the Poisson structure $\pi_n$ is {\it log-canonical} in the affine coordinate chart
$\O^\ga$, if
the Poisson brackets between the coordinate functions $(z_1, z_2, \ldots, z_n)$ on
$\O^\ga$ have the form $\{z_i, z_k\} = \lambda_{ik} z_iz_k$ for some $\lambda_{ik} \in \Cset$ for each pair $1 \leq i < k \leq n$.
By Lemma \ref{le-zizj-1}, $\pi_n$ is log-canonical in $\O^\gamma$ if and only if 
\[
\{z_i, z_k\} = \epsilon_i \la \ga^i(\alpha_i), \ga^k(\alpha_k)\ra z_iz_k, \hs 1 \leq i < k \leq n,
\]
where $\epsilon_i = 1$ if $\ga_i = e$ and $\epsilon_i = -1$ if $\ga_i = s_i$.
The following Lemma \ref{le-zizj-eee}, which follows trivially from Lemma \ref{le-zizj-1}, says that $\pi_n$ is log-canonical in the affine chart $\O^{(e, e, \ldots, e)}$.

\begin{lemma}\label{le-zizj-eee} In the coordinates $(z_1, z_2, \ldots, z_n)$ on $\O^{(e, e, \ldots, e)}$, one has
\[
\{z_i, \, z_k \} =\la \alpha_i, \, \alpha_k \ra z_iz_k, \hs \forall \; 1 \leq i < k \leq n.
\]
\end{lemma}

To exhibit other log-canonical affine coordinate charts for $\pi_n$, we make the following observation
on the functions $\sigma_i(z_k)$, $1 \leq i < k \leq n$, in Lemma \ref{le-zizj-1}.

\begin{lemma}\label{le-zizj-2} Let $\ga = (\ga_1, \ldots, \ga_n) \in \Upu$, and let $1 \leq i \leq n$.
If $\ga_i = s_i$ and if $k > i$ is such that $s_{j} \neq s_i$ for all $i+1 \leq j \leq k$, then
$\sigma_i(z_k) = 0$.
\end{lemma}

\begin{proof} For
$i+1 \leq j \leq n$, let $z_j \in \Cset$ and 
$p_{j} = u_{-\ga_{j}(\alpha_j)}(z_j)\dot{\gamma}_k$. For  $t \in \Cset$, consider  
\[
[u_{\alpha_i}(t)p_{i+1}, \; p_{i+2}, \;\ldots, \;p_n] \in Z_{(s_{i+1}, \ldots, s_n)}.
\] 
For each $i+1 \leq j \leq k$, since $p_{i+1} p_{i+2} \cdots p_j$ lies in the Levi subgroup of the parabolic subgroup of $G$ determined by the 
set of simple roots in $\{\alpha_{i+1}, \ldots, \alpha_j\}$ which does not contain $\alpha_i$, one has
\[
(p_{i+1} p_{i+2} \cdots p_j)^{-1} u_{\alpha_i}(t)p_{i+1} p_{i+2} \cdots p_j \in N,
\]
where $N$ is the unipotent subgroup of $G$ with Lie algebra $\n = \sum_{\alpha \in \Delta_+} \g_\alpha$.
It thus follows from the definition of the vector field $\sigma_i$ that $\sigma_{i}(z_k) = 0$, where
$z_k$ is now regarded as a local function on $Z_{(s_{i+1},\ldots, s_n)}$.
\end{proof}  

The next 
Lemma \ref{le-zizj-3}, which follows directly from Lemma
\ref{le-zizj-1} and Lemma \ref{le-zizj-2}, exhibits a log-canonical affine chart for $\pi_n$ associated to each 
$s \in \{s_1, \ldots, s_n\}$.

\begin{lemma}\label{le-zizj-3}
Let $s \in \{s_1, s_2, \ldots, s_n\}$ and let $i_0 = {\rm max}\{i: 1 \leq i \leq n, \, s_i = s\}$.
Let $\ga = (\ga_1, \ga_2, \ldots, \ga_n)$ be such that
$\ga_{i_0} = s$ and $\ga_{i} = e$ for all $i \neq i_0$. Then in the coordinates $(z_1, z_2, \ldots, z_n)$ on $\O^\ga$ and for all
$1 \leq i < k \leq n$, one has
\begin{equation}\label{eq-zizj-i0}
\{z_i, z_k\} = \begin{cases} \la \alpha_i, \alpha_k\ra z_iz_k, & \;\; 1 \leq i < k < i_0 \;\; \mbox{or} \;\; 
i_0 < i < k \leq n,\\
\la \alpha_i, s(\alpha_k)\ra z_iz_k, & \;\; 1 \leq i \leq i_0 \leq k \leq n, \; i \neq  k.\end{cases}
\end{equation}
\end{lemma}

The following Corollary \ref{co-zizj} also follows directly from Lemma \ref{le-zizj-1} and Lemma \ref{le-zizj-2}.

\begin{corollary}\label{co-zizj}
If $\bfu = (s_1, s_2, \ldots, s_n)$ is such that $s_i \neq s_j$ for all $i \neq j$, then the Poisson 
structure $\pi_n$ on $\Zu$ is log-canonical in every one of the $2^n$ affine coordinate charts
$\{\O^\ga: \ga \in \Upu\}$.
\end{corollary}

\section{The Poisson structure $\pi_n$ in affine coordinates charts, II}\label{sec-coor-II}
Throughout $\S$\ref{sec-coor-II}, fix a sequence $\bfu = (s_1, \ldots, s_n)$ of simple reflections, and let $\Zu$ be the
corresponding  Bott-Samelson variety. 
To better understand the Poisson structure $\pi_n$ in the coordinates
$(z_1, z_2, \ldots, z_n)$ on the affine chart ${\mathcal O}^\gamma$ defined in $\S$\ref{coordcharts}, where $\gamma \in \Upu$, one needs to compute more explicitly 
the vector field $\sigma_i$ in Lemma \ref{le-zizj-1} on the Bott-Samelson variety
$Z_{(s_{i+1},\ldots, s_n)}$ for $1 \leq i \leq n-1$. 
For $x \in \b$, define the
vector field $\sigma_x$ on $\Zu$ by 
\begin{equation}\label{eq-Vx}
\sigma_x(p)=\frac{d}{dt}|_{t=0} ((\exp t x)\cdot p),
\qquad p \in \Zu,
\end{equation}
where $\cdot$ denotes the left action of $B \subset P_{s_1}$ on $\Zu$ given in (\ref{eq-P-on-Zu}).
Using some facts on root strings of the root system of $\g$ reviewed in $\S$\ref{subsec-root-strings}, 
for any $\beta \in \Delta_+$ and $e_\beta \in \g_\beta$, we give in
$\S$\ref{subsec-V-beta}
an explicit formula for $\sigma_{e_\beta}$ in 
the coordinates $(z_1, z_2, \ldots, z_n)$ one each  affine chart ${\mathcal O}^\gamma$ of $\Zu$. 
The formula for $\sigma_{e_\beta}$, given in Theorem \ref{epsilonbeta}, is expressed explicitly
in terms of the root strings and the structure constants of $\g$. As a consequence
(see Theorem \ref{th-zizj}), the Poisson structure $\pi_n$
can also be expressed  in each affine coordinate 
chart $\O^\gamma$ in terms of
root strings and the structure constants of $\g$. We believe that our formula for the vector fields
$\sigma_{e_\beta}$ is of interest irrespective of the Poisson structure $\pi_n$.

\subsection{Some lemmas on root strings}\label{subsec-root-strings}
In $\S$\ref{subsec-root-strings}, let
\begin{equation}\label{eq-basis}
\{h_\alpha\}_{\alpha \in \Gamma} \cup \{e_{\pm\alpha} \in \g_{\pm\alpha}\}_{\alpha \in \Delta_+}
\end{equation}
be any basis of $\g$ such that $[e_\alpha, e_{-\alpha}] = h_\alpha$ for each $\alpha \in \Delta_+$. One then has the Lie group homomorphism $\theta_\alpha: SL(2, \Cset) \to G$ for each $\alpha \in \Delta_+$. Let the notation be
as in $\S$\ref{notations}.
For $\alpha$, $\beta\in\Delta$ such that $\alpha+\beta\in\Delta$, let $N_{\alpha,\beta}\neq 0$ be such that $[e_\alpha,e_\beta]=N_{\alpha,\beta}e_{\alpha+\beta}$.

\begin{lemma}
\label{ddtcommute}
For $\alpha\in \Delta_+$, one has
\begin{align}
\label{eb1}
u_\alpha(t)u_{\alpha}(z)\dot{s}_\alpha&=u_{\alpha}(t+z)\dot{s}_\alpha, &&t,z\in\mathbb{C}, \\
\label{eb2}
u_\alpha(t)u_{-\alpha}(z)&=u_{-\alpha}\left( \frac{z}{1+tz} \right)u_\alpha(t(1+tz))\alpha^\vee(1+tz), && t,z\in\mathbb{C}, 1+tz\neq 0, \\
u_{-\alpha}(t)&=u_\alpha\left( \frac{1}{t} \right)\dot{s}_\alpha u_\alpha(t)\alpha^\vee(t), &&t\in\mathbb{C}^\times. \label{change}
\end{align}
For $\alpha, \beta\in\Gamma$ and $\alpha \neq \beta$, one has . 
\begin{equation}\label{ezcommutation}
u_\beta(t)\beta^\vee(t)u_{-\alpha}(z)=
u_{-\alpha}\left( t^{\frac{-2\langle \alpha,\beta\rangle}{\langle \beta,\beta \rangle}}z \right) u_\beta(t)\beta^\vee(t),
\;\;\;\;\;\;\;t\in\mathbb{C}^\times, \, z \in {\mathbb{C}}.
\end{equation}
\end{lemma}
\begin{proof}
Identities (\ref{eb2}) and (\ref{change}) follow from computations in $\mathrm{SL}(2,\mathbb{C})$, and (\ref{ezcommutation})
follows from the fact that the two root subgroups corresponding to $-\alpha$ and $\beta$ commute.
\end{proof}

Let $\alpha$ and $\beta$ be two linearly independent roots, $\alpha \in \Delta_+$, and let
 $\{\beta+j\alpha: -p \leq j \leq q\}$, where $p$ and $q$ are non-negative integers, be the $\alpha$-string through $\beta$. 
Then the subspace 
\[
L=\sum_{j=-p}^q \mathfrak{g}_{\beta+j\alpha}
\] 
of $\mathfrak{g}$ becomes an $SL(2, \mathbb{C})$-module via the 
group homomorphism $\theta_\alpha: SL(2, \mathbb{C}) \to G$ and the adjoint representation of $G$ on
$\mathfrak{g}$. On the other hand, let $L^{p+q}$ be the vector space of
homogeneous polynomials in $(x, y)$ of degree $p+q$ with the action of
$\mathrm{SL}(2,\mathbb{C})$ by 
\[
\left(\begin{pmatrix}
a & b \\
c & d 
\end{pmatrix} \cdot f\right)(x,y)=f\left(  (x , y)  \begin{pmatrix} 
a & b \\
c & d 
\end{pmatrix} \right) = f \left(ax+cy, \, bx+dy\right), \;\; \begin{pmatrix}
a & b \\
c & d 
\end{pmatrix} \in SL(2, \Cset).
\]
Let $\{u_0,\ldots, u_{p+q}\}$ be the basis of $L^{p+q}$ given by
\begin{equation}\label{eq-ui}
u_i=\varepsilon_0\varepsilon_1\cdots \varepsilon_{i-1}\binom{p+q}{i}x^iy^{p+q-i},\qquad 0 \leq i \leq p+q,
\end{equation}
where  for 
$0 \leq j \leq p+q-1$, $\varepsilon_j \in \Cset$ is defined by
\begin{equation}\label{epsilon-i}
\varepsilon_j=\frac{j+1}{N_{\alpha,\beta-(p-j)\alpha}},
\end{equation}
and it is understood that
$\varepsilon_0\varepsilon_1\cdots \varepsilon_{i-1}=1$ when $i = 0$ in (\ref{eq-ui}).

\begin{lemma}
\label{sl2lemma}
With the notation as above, the linear map
\begin{equation}\label{chi}
\chi:\; L\longrightarrow L^{p+q}:\;\; \chi(e_{\beta+j\alpha})=u_{p+j},\quad -p\leq j \leq q,
\end{equation}
is an $\mathrm{SL}(2,\mathbb{C})$-equivariant isomorphism.
\end{lemma}
\begin{proof} The two irreducible representations of $SL(2, \mathbb{C})$ on $L$ and on
$L^{p+q}$, being of the same dimension, must be isomorphic, and by Schur's lemma, there
is a unique $SL(2, \mathbb{C})$-equivariant isomorphism $\chi: L \to L^{p+q}$ 
such that $\chi(e_{\beta-p\alpha}) = u_0$. Straightforward calculations show that 
$\chi$ must be given as in (\ref{chi}).
See also \cite[Lemma 6.2.2]{Car}.
\end{proof}

The following Lemma \ref{useful} is the key to the proof of Theorem \ref{epsilonbeta} in $\S$\ref{subsec-V-beta}.

\begin{lemma}
\label{useful}
Let $\alpha \in \Delta_+$ and $\beta \in \Delta$ be linearly independent, and let
$\{\beta+j\alpha: -p \leq j \leq q\}$ be the $\alpha$-string through $\beta$. Then for any $t\in\mathbb{C}$, one has
\begin{align}
\label{+comm}
\mathrm{Ad}_{(u_\alpha (t) \dot{s}_{\alpha})^{-1}}(e_\beta) &= \sum_{j=0}^q(-1)^p\frac{\varepsilon_0\varepsilon_1\cdots \varepsilon_{p-1}}{\varepsilon_0\varepsilon_1\cdots\varepsilon_{q-j-1}}\binom{p+j}{j}t^j e_{s_\alpha(\beta)-j\alpha}, \\
\label{-comm}
\mathrm{Ad}_{(u_{-\alpha}(t))^{-1}}(e_\beta)&=\sum_{j=0}^p(-1)^j \varepsilon_{p-j}\varepsilon_{p-j+1}\cdots\varepsilon_{p-1}\binom{q+j}{j}t^je_{\beta-j\alpha}.
\end{align}
\end{lemma}

\begin{proof} By Lemma \ref{sl2lemma}, one has
\begin{align*}
\chi\left( \mathrm{Ad}_{(u_\alpha (t) \dot{s}_{\alpha})^{-1}}(e_\beta) \right) &= \begin{pmatrix}
0 & 1 \\
-1 & t
\end{pmatrix}\cdot u_p \\
&=\varepsilon_0 \varepsilon_1\cdots \varepsilon_{p-1}\binom{p+q}{p}(-y)^p(x+ty)^q \\
&=\varepsilon_0 \varepsilon_1\cdots \varepsilon_{p-1}\binom{p+q}{p}(-y)^p\left(\sum_{j=0}^q\binom{q}{j}t^jy^jx^{q-j} \right) \\
&=\sum_{j=0}^q(-1)^p\frac{\varepsilon_0 \varepsilon_1\cdots \varepsilon_{p-1}}{\varepsilon_0\varepsilon_1\cdots\varepsilon_{q-j-1}}\binom{p+j}{j}t^ju_{q-j}.
\end{align*}
It follows that
\[
\mathrm{Ad}_{(u_\alpha (t) \dot{s}_{\alpha})^{-1}}(e_\beta) = \sum_{j=0}^q(-1)^p\frac{\varepsilon_0\varepsilon_1\cdots \varepsilon_{p-1}}{\varepsilon_0\varepsilon_1\cdots\varepsilon_{q-j-1}}\binom{p+j}{j}t^j e_{\beta+(q-p-j)\alpha}.
\]
As (see for example, \cite[Proposition 25.1]{Humphreys}) $\displaystyle \frac{2\la \beta, \alpha\ra}{\la \alpha, \alpha\ra}
= p-q$, one has, for any $j \in \Zset$,
\[
s_\alpha (\beta) - j\alpha = \beta - \frac{2\la \beta, \alpha\ra}{\la \alpha, \alpha\ra} \alpha - j \alpha =
\beta +(q-p-j)\alpha,
\]
from which  (\ref{+comm}) follows. One proves (\ref{-comm}) similarly (see also Lemma 6.2.1 in \cite{Car}). 
\end{proof}

To unify the two formulas in (\ref{+comm}) and (\ref{-comm}), for $\alpha \in \Delta_+$, $\kappa \in \{s_\alpha, e\}$,
and $t \in \Cset$, let
\begin{equation}\label{eq-p-alpha}
p_{\kappa, \alpha}(t) = u_{-\kappa(\alpha)}(t)\dot{\kappa} \in P_{s_\alpha},
\end{equation}
and for $\beta \in \Delta$, $\beta \neq \pm \alpha$, as in Lemma \ref{useful}, let 
\begin{align}\label{constants-1}
c_{\alpha,\beta}^{\kappa,j} &=(-1)^p\frac{\varepsilon_0\varepsilon_1\cdots \varepsilon_{p-1}}{\varepsilon_0\varepsilon_1\cdots\varepsilon_{q-j-1}}\binom{p+j}{j}, &j=0,\ldots,  q \textrm{ and }\kappa=s_\alpha, \\
\label{constants-2}
c_{\alpha,\beta}^{\kappa,j} &=(-1)^j\varepsilon_{p-j}\varepsilon_{p-j+1}\cdots\varepsilon_{p-1}\binom{q+j}{j}, &j=0,\ldots,  p \textrm{ and }\kappa=e.
\end{align}
Lemma \ref{useful} can now be reformulated as follows.

\begin{lemma}\label{useful-1}
Let $\alpha \in \Delta_+$ and $\beta \in \Delta$ be linearly independent. Then for $\kappa \in \{s_\alpha, e\}$ and $t \in \Cset$,
\begin{equation}\label{eq-combine}
{\rm Ad}_{(p_{\kappa, \alpha}(t))^{-1}}(e_\beta)
= \sum_{\stackrel{j \geq 0,}{\kappa(\beta) - j\alpha \in \Delta}} c_{\alpha, \beta}^{\kappa, j}\,
t^j \,e_{\kappa(\beta)-j\alpha}.
\end{equation}
\end{lemma}

\begin{proof} Let $j \in \Zset$ and $j \geq 0$. 
When $\kappa = e$, $\kappa(\beta)-j\alpha \in \Delta$ if and only if $\beta - j\alpha \in \Delta$, which is the same as
$0 \leq j \leq p$. When $\kappa = s_\alpha$, $\kappa(\beta)-j\alpha \in \Delta$ if and only if
$s_\alpha(\beta+j\alpha) \in \Delta$, which is the same as $\beta + j\alpha \in \Delta$, which, in turn, is the same as
$0 \leq j \leq q$.
\end{proof}

\begin{remark}
\label{chevbase}
{\em Recall that a basis $\{h_\alpha\}_{\alpha \in \Gamma} \cup \{e_{\alpha} \in \g_{\alpha}\}_{\alpha \in \Delta}$
of $\g$ is said to be a Chevalley basis if 
$[e_\alpha, e_{-\alpha}] = h_\alpha$ for all $\alpha \in \Delta$, and if 
for all $\alpha,\beta\in\Delta$ such that $\alpha+\beta\in\Delta$, one has $N_{\alpha,\beta}=-N_{-\alpha,-\beta}$.
If $\{h_\alpha\}_{\alpha \in \Gamma} \cup \{e_{\alpha} \in \g_{\alpha}\}_{\alpha \in \Delta}$ is a Chevalley basis of $\mathfrak{g}$,
by \cite[Theorem 4.1.2]{Car} and \cite[Theorem 25.2]{Humphreys}, $N_{\alpha, \beta} = \pm (p+1)$ for any
roots $\alpha$ and $\beta$ such that $\alpha + \beta \in \Delta$, where $p$ is the largest 
non-negative integer such that $\beta -p \alpha \in \Delta$. Thus, for 
$\alpha$ and $\beta$ as in Lemma \ref{useful} and for every $0 \leq j \leq p+q-1$,
one has $\varepsilon_j = \pm 1$, and consequently all the coefficients $c_{\alpha, \beta}^{\kappa, j}$'s appearing in 
\eqref{eq-combine} are integers. \hfill $\diamond$
}
\end{remark}

\subsection{The vector field $\sigma_{e_\beta}$ in coordinates}\label{subsec-V-beta}
Fix again $\bfu = (s_1, \ldots, s_n) = (s_{\alpha_1}, \ldots, s_{\alpha_n})$ be a sequence of simple reflections, and let $\Zu$ be the
corresponding Bott-Samelson variety. Let $\{e_\alpha \in \g_\alpha: \alpha \in \Gamma\}$ be a set of root vectors for the simple roots, and extend it
to a basis $\{h_\alpha\}_{\alpha \in \Gamma} \cup \{e_{\alpha} \in \g_{\alpha}\}_{\alpha \in \Delta}$ of $\g$
such that $[e_\alpha, e_{-\alpha}] = h_\alpha$ for all $\alpha \in \Delta$.
Recall from (\ref{eq-Vx}) that for any $x \in \b$, $\sigma_x$ is the vector field on $\Zu$ generating the action of $B$ on $\Zu$ in the direction of $x$. 
For $\beta \in \Delta_+$, we then have the vector field $\sigma_{e_\beta}$ on $\Zu$ given by
\begin{equation}\label{e-beta}
\sigma_{e_\beta}(p) = \frac{d}{dt}|_{t=0} ((\exp te_\beta) \cdot p), \hs p \in \Zu.
\end{equation}
On the other hand, the choice $\{e_\alpha: \alpha \in \Gamma\}$ gives rise to coordinates $(z_1, \ldots, z_n)$ on
the affine chart $\O^\gamma$ for each $\gamma \in \Upu$.
In this section, for every $\gamma \in \Upu$, we use the results in 
$\S$\ref{subsec-root-strings} to compute the vector fields $\sigma_{e_\beta}$, $\beta \in \Delta_+$,
in the coordinates $(z_1, \ldots, z_n)$
on $\mathcal{O}^\gamma$ in terms of root strings and structure constants of $\g$ in the basis  
$\{h_\alpha\}_{\alpha \in \Gamma} \cup \{e_{\alpha} \in \g_{\alpha}\}_{\alpha \in \Delta}$ of $\g$. 

For $x \in \b$ and $1 \leq k \leq n$, consider also the vector field $\sigma_x^{(k)}$ on the Bott-Samelson variety 
$Z_{(s_k, \ldots, s_n)}$ defined by
\begin{equation}\label{eq-sigma-k}
\sigma_x^{(k)}(p)=\frac{d}{dt}|_{t=0}  ((\exp t x) \cdot p), \hs p \in Z_{(s_k, \ldots, s_n)},
\end{equation}
where again $\cdot$ denotes the left action of $B$ on $Z_{(s_k, \ldots, s_n)}$ (see (\ref{eq-P-on-Zu})).
Note that $\sigma_x = \sigma_x^{(1)}$ for $x \in \b$.

Fix $\gamma = (\gamma_1, \ldots, \gamma_n)\in\Upu$ and let $(z_1, \ldots, z_n)$ be the coordinates on $\O^\gamma \subset \Zu$. 
For $1 \leq k \leq n$, we also regard $(z_k, \ldots, z_n)$ as coordinates on the affine chart
$\O^{(\gamma_k, \ldots, \gamma_n)}$ of $Z_{(s_k,\ldots, s_n)}$, so for $x \in \b$ and $k \leq j \leq n$, $\sigma_x^{(k)}(z_j)$
is the action of $\sigma_x^{(k)}$ on $z_j$ as a function on $\O^{(\gamma_k, \ldots, \gamma_n)}\subset Z_{(s_k,\ldots, s_n)}$.

The following Lemma \ref{recursive} gives a recursive formula for $\sigma_{e_\beta}$, regarded as a vector field on $\O^\gamma$.

\begin{lemma}
\label{recursive}
Let $\beta\in\Delta_+$ and $\gamma = (\gamma_1, \ldots, \gamma_n)\in \Upu$.

1) $\beta=\alpha_1$ and $\gamma_1=s_1$. In this case, $\sigma_{e_\beta}(z_1)=1$ and $\sigma_{e_\beta}(z_k)=0$ for all $k\geq 2$;

2) $\beta=\alpha_1$ and $\gamma_1=e$. In this case, $\sigma_{e_\beta}(z_1)=-z_1^2$ and for $k\geq 2$,
\[
\sigma_{e_\beta}(z_k)=\sigma_{e_\beta}^{(2)}(z_k)+z_1 \sigma_{h_{\alpha_1}}^{(2)}(z_k);
\]

3) $\beta\neq\alpha_1$. In this case, $\sigma_{e_\beta}(z_1)=0$ and for $k\geq 2$,
\[
\sigma_{e_\beta}(z_k)=\sum_{\stackrel{j\geq 0,}{ \gamma_1(\beta)-j\alpha_1\in\Delta_+}} \, c_{\alpha_1,\beta}^{\gamma_1,j}\,z_1^j\, \sigma_{e_{\gamma_1(\beta)-j\alpha_1}}^{(2)}(z_k).
\]
\end{lemma}

\begin{proof}
Cases 1) and 2) follow from (\ref{eb1}) and (\ref{eb2}) respectively. Case 3) follows from Lemma \ref{useful-1}
and the fact that, as $\beta \in \Delta_+$ and $\beta \neq \alpha_1$, all the roots in the $\alpha_1$-string through $\gamma_1(\beta)$ are positive.
\end{proof}

To combine the cases in Lemma \ref{recursive}, we note that when $\beta = \alpha_1$, 
\[
\{j_1 \geq 0: \; \gamma_1(\beta)-j_1\alpha_1 \in \Delta_+\} = \begin{cases} \emptyset, & \;\; \mbox{if} \;\; \gamma_1 = s_1,\\
\{0\}, &\;\; \mbox{if} \;\; \gamma_1 = e.\end{cases}
\]
For $\alpha \in \Gamma$, also set
\begin{equation}\label{eq-c0}
c_{\alpha,\alpha}^{e,0}=1.
\end{equation}
We can now reformulate Lemma \ref{recursive} as follows.

\begin{lemma}\label{recursive-1}
Let $\beta\in\Delta_+$ and $\gamma = (\gamma_1, \ldots, \gamma_n)\in \Upu$. Then
\begin{equation}\label{eq-z1}
\sigma_{e_\beta}(z_1) = \begin{cases} 1, & \;\; \mbox{if} \;\; \beta = \alpha_1 \; \mbox{and} \; \gamma_1 = s_1,\\
-z_1^2, & \;\; \mbox{if} \;\; \beta = \alpha_1 \; \mbox{and} \; \gamma_1 = e,\\
0, &  \;\; \mbox{if} \;\; \beta \neq \alpha_1,\end{cases}
\end{equation}
and for $2 \leq k \leq n$,
\begin{align}
\label{repeat-0}
\sigma_{e_\beta}(z_k)=\!\!\!\sum_{\stackrel{j_1 \geq 0,}{\gamma_1(\beta)-j_1\alpha_1 \in \Delta_+}}
c_{\alpha_1,\beta}^{\gamma_1,j_1}\, z_1^{j_1}\,\sigma_{e_{\beta_{\gamma_1(\beta)-j_1\alpha_1}}}^{(2)}(z_k)+
\begin{cases}
z_1\sigma_{h_{\alpha_1}}^{(2)}(z_k), &\textrm{if}\beta=\alpha_1\textrm{ and }\gamma_1=e, \\
0, &\mbox{otherwise}.
\end{cases}
\end{align}
\end{lemma}

To obtain a closed formula for the vector field $\sigma_{e_\beta}$ on $\Zu$, we introduce more notation. 
Let $\mathbb{N}$ denote the set of non-negative integers.
\begin{notation}\label{nota-phi-psi}
{\em  For $\beta \in \Delta_+$ and $(j_1,\ldots, j_n)\in\mathbb{N}^n$, let $\beta_{(j_1)} = \gamma_1(\beta) - j_1 \alpha_1 \in \h^*$, and for $2 \leq
k \leq n$, let
\begin{align*}
&\beta_{(j_1,\ldots, j_k)} =\gamma_k (\beta_{(j_1, \ldots, j_{k-1})}) - j_k \alpha_k \\
& \hs \hs \;\;\;\;= \gamma_k\gamma_{k-1}\cdots \gamma_2\gamma_1(\beta)-j_1\gamma_k\gamma_{k-1}\cdots\gamma_2(\alpha_1)-\ldots -j_{k-1}\gamma_k(\alpha_{k-1})-j_k\alpha_k\in\mathfrak{h}^*,\\
&J_{k} =\left\{(j_1,\ldots,j_{k-1})\in {\mathbb{N}}^{k-1}: \beta_{(j_1,\ldots, j_l)} \in\Delta_+, \,\forall\,
1\leq l \leq k-1, \; \mbox{and} \;\, \beta_{(j_1,\ldots, j_{k-1})} =\alpha_{k}\right\}.
\end{align*}
For $2 \leq k \leq n$ and for $(j_1,\ldots, j_{k-1})\in J_k$, let
\begin{align}
\label{const}
c_{j_1,\ldots, j_{k-1}}^\gamma=c_{\alpha_1,\beta}^{\gamma_1,j_1}\cdots
c_{\alpha_{k-1},\beta_{(j_1,\ldots, j_{k-2})}}^{\gamma_{k-1},j_{k-1}} \neq 0.
\end{align}
Here it is understood that $\beta_{(j_1,\ldots, j_{k-2})} = \beta$ if $k = 2$. Also note that for $k \geq 2$ and 
$1 \leq i \leq k-1$, $c_{\alpha_{i},\beta_{(j_1,\ldots, j_{i-1})}}^{\gamma_{i},j_{i}}$ is defined in 
(\ref{constants-1}) and (\ref{constants-2}) when $\beta_{(j_1,\ldots, j_{i-1})} \neq \alpha_i$, and if 
$\beta_{(j_1,\ldots, j_{i-1})} = \alpha_i$, then $\gamma_{i} (\beta_{(j_1,\ldots, j_{i-1})}) -j_i \alpha_i \in \Delta_+$ only 
if $\gamma_i = e$ and $j_i = 0$, and in this case  $c_{\alpha_{i},\beta_{(j_1,\ldots, j_{i-1})}}^{\gamma_{i},j_{i}}
=1$ as defined in (\ref{eq-c0}).

For each $1\leq k \leq n$, introduce two functions $\phi_\beta^\gamma(z_1,\ldots, z_{k-1})$ and $\psi_\beta^\gamma(z_1,\ldots, z_{k-1})$ as follows: for $k=1$, let
\begin{equation}\label{eq-phi-psi-1}
\phi_\beta^\gamma(z_1,\ldots, z_{k-1})=
\begin{cases}
1 &\textrm{ if }\beta=\alpha_1, \\
0 &\textrm{ if }\beta\neq\alpha_1,
\end{cases}
\quad
\textrm{ and }
\quad
\psi_\beta^\gamma(z_1,\ldots, z_{k-1})=0,
\end{equation}
and for $2\leq k \leq n$, let
\begin{align}
\label{phibeta}
\phi_\beta^\gamma(z_1,\ldots, z_{k-1})&=\sum_{(j_1,\ldots, j_{k-1})\in J_k}c_{j_1,\ldots, j_{k-1}}^\gamma
\, z_1^{j_1}z_2^{j_2}\cdots z_{k-1}^{j_{k-1}}, \\
\label{psibeta}
\psi_\beta^\gamma(z_1,\ldots, z_{k-1})&=-\sum_{1\leq i \leq k-1, \, \gamma_i = e}\frac{2\langle \gamma^i(\alpha_i),\gamma^k(\alpha_k) \rangle}{\langle \gamma^i(\alpha_i),\gamma^i(\alpha_i) \rangle}z_i\phi_\beta^\gamma(z_1,\ldots z_{i-1}),
\end{align}
where recall that $\gamma^i=\gamma_1\gamma_2\cdots\gamma_i$ for $1\leq i \leq n$, and the function
$\phi^\gamma_\beta(z_1, \ldots, z_{k-1})$ (resp. $\psi^\gamma_\beta(z_1, \ldots, z_{k-1})$) is defined to be $0$ 
if the index set for the summation on the right hand side of (\ref{phibeta}) (resp. (\ref{psibeta})) is empty.
}
\end{notation}

\begin{remark}\label{power}
{\em Since a root string can have length at most $4$, it follows
from \eqref{phibeta} and \eqref{psibeta} that the powers of any coordinate $z_i$ in the polynomials
$\phi_\beta^\gamma(z_1,\ldots, z_{k-1})$ and $\psi_\beta^\gamma(z_1,\ldots, z_{k-1})$ can be at most $3$ (and $1$ when
$\g$ is simply-laced).
\hfill $\diamond$
}
\end{remark}

The following Theorem \ref{epsilonbeta} gives a purely combinatorial formula for the vector field $\sigma_{e_\beta}$.

\begin{theorem}
\label{epsilonbeta}
Let $\beta \in \Delta_+$ and  let $\gamma = (\gamma_1, \ldots, \gamma_n)\in \Upu$. The vector field $\sigma_{e_\beta}$ acts on the coordinate
functions $(z_1, \ldots, z_n)$ on the affine chart $\mathcal{O}^\gamma$ as follows: for $1 \leq k \leq n$,
\begin{equation}\label{V-e-beta}
\sigma_{e_\beta}(z_k)=
\begin{cases}
\phi_\beta^\gamma(z_1,\ldots, z_{k-1})+\psi_\beta^\gamma(z_1,\ldots,z_{k-1})z_k, &\textrm{ if }\;\gamma_k=s_k, \\
-\phi_\beta^\gamma(z_1,\ldots, z_{k-1})z_k^2+\psi_\beta^\gamma(z_1,\ldots, z_{k-1})z_k, &\textrm{ if }\;\gamma_k=e.
\end{cases}
\end{equation}
\end{theorem}

\begin{proof}
When $k=1$, Theorem \ref{epsilonbeta} holds by (\ref{eq-phi-psi-1}) and by Lemma \ref{recursive-1}. 
Let $k \geq 2$. Let
\[
J_k^\prime=\left\{(j_1,\ldots,j_{k-1})\in\mathbb{N}^{k-1}: \beta_{(j_1,\ldots, j_l)} \in\Delta_+, \,\forall\,
1\leq l \leq k-1\right\},
\]
and define $c_{j_1,\ldots, j_{k-1}}^\gamma \in \Cset^\times$ for $(j_1, \ldots, j_{k-1}) \in J_k^\prime$ as in
(\ref{const}). 
Then by Lemma \ref{recursive-1}, 
\begin{align}
\label{repeat}
\sigma_{e_\beta}(z_k)=\sum_{j_1\in J_2^\prime}c_{\alpha_1,\beta}^{\gamma_1,j_1}z_1^{j_1}\sigma_{e_{\beta_{(j_1)}}}^{(2)}(z_k)+
\begin{cases}
z_1\sigma_{h_{\alpha_1}}^{(2)}(z_k), &\textrm{ if }\beta=\alpha_1\textrm{ and }\gamma_1=e, \\
0, &\textrm{ otherwise.}
\end{cases}
\end{align}
By repeatedly using (\ref{repeat}), one has
\begin{align*}
\sigma_{e_\beta}(z_k) &=\sum_{(j_1,\ldots, j_{k-1})\in J_k^\prime}c_{j_1,\ldots, j_{k-1}}^\gamma z_1^{j_1}\cdots z_{k-1}^{j_{k-1}}\sigma_{e_{\beta_{(j_1,\ldots, j_{k-1})}}}^{(k)}(z_k)\\
& \hs +\sum_{1\leq i \leq k-1, \, \gamma_i = e}\phi_\beta^\gamma(z_1,\ldots, z_{i-1})z_i
\sigma_{h_{\alpha_i}}^{(i+1)}(z_k).
\end{align*}
Let $z_k^\prime=1$ if $\gamma_k=s_k$ and $z_k^\prime=-z_k^2$ if $\gamma_k=e$. By Lemma \ref{recursive}, for $(j_1,\ldots, j_{k-1})\in J_k^\prime$, one has $\sigma_{e_{\beta_{(j_1,\ldots, j_{k-1})}}}^{(k)}(z_k)=0$ unless $\beta_{(j_1,\ldots, j_{k-1})}=\alpha_k$, in which case $\sigma_{e_{\beta_{(j_1,\ldots, j_{k-1})}}}^{(k)}(z_k)=z_k^\prime$. Thus
\begin{align*}
\sigma_{e_\beta}(z_k)&=\sum_{(j_1,\ldots, j_{k-1})\in J_k}c_{j_1,\ldots, j_{k-1}}^\gamma z_1^{j_1}\cdots z_{k-1}^{j_{k-1}}z_k^\prime+\sum_{1\leq i \leq k-1, \, \gamma_i=e}\phi_\beta^\gamma(z_1,\ldots, z_{i-1})z_i\sigma_{h_{\alpha_i}}^{(i+1)}(z_k) \\
&=\phi_\beta^\gamma(z_1,\ldots, z_{k-1})z_k^\prime + \sum_{1\leq i \leq k-1, \,\gamma_i=e}\phi_\beta^\gamma(z_1,\ldots, z_{i-1})z_i\sigma_{h_{\alpha_i}}^{(i+1)}(z_k)
\end{align*}
On the other hand, for each $1\leq i \leq k-1$ with $\gamma_i=e$,
\[
\sigma_{h_{\alpha_i}}^{(i+1)}(z_k)=-\frac{2\langle \alpha_i,\gamma_{i+1}\cdots \gamma_k(\alpha_k) \rangle}{\langle \alpha_i,\alpha_i \rangle}z_k=-\frac{2\langle \gamma^i(\alpha_i),\gamma^k(\alpha_k) \rangle}{\langle \gamma^i(\alpha_i),\gamma^i(\alpha_i) \rangle}z_k.
\]
It follows that
\[
\sigma_{e_\beta}(z_k)=\phi_\beta^\gamma(z_1,\ldots, z_{k-1})z_k^\prime+\psi_\beta^\gamma(z_1,\ldots, z_{k-1})z_k.
\]
\end{proof}

\begin{remark}\label{re-V}
{\em In the context of Theorem \ref{epsilonbeta}, for a given $\gamma = (\gamma_1, \ldots, \gamma_n) \in \Upu$ and
$1 \leq k \leq n$, let
$\gamma^\prime = (\gamma_1, \ldots, \gamma_{k-1}, \gamma_k s_k, \gamma_{k+1}^\prime, \ldots, \gamma_n^\prime) \in \Upu$,
where $\gamma_j^\prime \in \{e, s_j\}$ are arbitrary for $k+1 \leq j \leq n$, and let $(z_1^\prime, \ldots, z_n^\prime)$
be the coordinates on $\O^{\gamma^\prime}$. Then $z_j = z_j^\prime$ for $1 \leq j \leq k-1$, and $z_k^\prime = 1/z_k$.
By (\ref{phibeta}) and (\ref{psibeta}), 
\[
\phi^\gamma_{\beta}(z_1, \ldots, z_{k-1}) = \phi^{\gamma^\prime}_\beta(z_1, \ldots, z_{k-1}) \hs\mbox{and} \hs
\psi^\gamma_{\beta}(z_1, \ldots, z_{k-1}) = -\psi^{\gamma^\prime}_\beta(z_1, \ldots, z_{k-1}).
\]
One can thus derive one case of the formula (\ref{V-e-beta}) from the other case using the change of coordinates
$z_k^\prime = 1/z_k$.
\hfill $\diamond$}
\end{remark}

\begin{example}\label{ex-e-beta}
{\em
Let $\beta$ be a simple root and let $\gamma = (e, e, \ldots, e) \in \Upu$. 
Then in the affine chart $\mathcal{O}^{(e,e,\ldots,e)}$ with coordinates $(z_1, \ldots, z_n)$ given in \eqref{eq-Phi-gamma},
the vector field $\sigma_{e_\beta}$  is given  by
\begin{equation}\label{eq-eee}
\sigma_{e_\beta}(z_k)=-\frac{2\langle \beta,\alpha_k \rangle}{\langle \beta,\beta \rangle} \left(\sum_{1\leq i\leq k-1,\,\alpha_i=\beta}z_{i} \right)
z_k+
\begin{cases}
0, &\textrm{if }\alpha_k\neq\beta, \\
-z_k^2, &\textrm{if }\alpha_k=\beta,
\end{cases}\qquad 1 \leq k \leq n.
\end{equation}
Indeed, let $1 \leq k \leq n$. By Theorem \ref{epsilonbeta}, one has,
\[
\sigma_{e_{\beta}}(z_k)=
-\phi_{\beta}^\gamma(z_1,\ldots,z_{k-1})z_k^2+\psi_{\beta}^\gamma(z_1,\ldots,z_{k-1})z_k.
\]
As $\beta$ is a simple root, one sees from the definition of $\phi_\beta^\gamma$ 
that $\phi_\beta^\gamma(z_1, \ldots, z_{k-1}) = 1$ if $\alpha_k = \beta$ and $\phi_\beta^\gamma(z_1, \ldots, z_{k-1}) = 0$ if
$\alpha_k \neq \beta$. It follows from the definition of $\psi_\beta^\gamma$ that 
\[
\psi_\beta^\gamma(z_1, \ldots, z_{k-1}) = -\frac{2\langle \beta,\alpha_k \rangle}{\langle \beta,\beta \rangle} \left(\sum_{1\leq i\leq k-1,\,\alpha_i=\beta}z_{i} \right).
\]
This proves \eqref{eq-eee}.
Applying Lemma \ref{le-zizj-1} and \eqref{eq-eee}, one sees that in the affine chart $\mathcal{O}^{(s_1, e, \ldots, e)}$, the Poisson structure $\Pi$ is given by
\begin{align*}
\{z_i,z_k\}&= \langle\alpha_i,\alpha_k\rangle z_iz_k, \;\;\;\;\;\textrm{if}\,\; 2\leq i <k \leq n, \\
\{z_1,z_k\}
&=\begin{cases}
-\langle \alpha_1,\alpha_k \rangle \left( z_1 -2\sum_{2\leq i\leq k-1,\,\alpha_i=\alpha_1}z_{i}\right)z_k,&\textrm{if } 2 \leq k \leq n \; \mbox{and} \; \alpha_k\neq\alpha_1,  \\
-\langle \alpha_1,\alpha_1 \rangle \left( z_1 -z_k -2\sum_{2 \leq i \leq k-1,\, \alpha_i=\alpha_1}z_{i}\right)z_k, &\textrm{if }
2 \leq k \leq n \; \mbox{and} \;\alpha_k=\alpha_1.
\end{cases}
\end{align*}
On the other hand, by Lemma \ref{le-zizj-eee}, in the coordinates $(\xi_1, \ldots, \xi_n)$ on 
$\mathcal{O}^{(e,e,\ldots ,e)}$ given by
\[
(\xi_1,\xi_2,\ldots ,\xi_n) \longmapsto [u_{-\alpha_1}(\xi_1),\, u_{-\alpha_2}(\xi_2),\,\ldots, \,u_{-\alpha_n}(\xi_n)],
\] 
the Poisson structure $\pi_n$ is given by 
$\{\xi_i, \xi_k\} = \langle \alpha_i, \alpha_k \rangle \xi_i\xi_k$ for all $1 \leq i < k \leq n$. 
It is easy to see that on the intersection $\mathcal{O}^{(e,e,\ldots ,e)} \cap \mathcal{O}^{(s_1,e,\ldots ,e)}$,
the changes between the 
coordinates $(\xi_1, \xi_2, \ldots, \xi_n)$ on $\O^{(e, e, \ldots, e)}$ and the coordinates $(z_1, z_2, \ldots, z_n)$
on $\O^{(s_1, e, \ldots, e)}$ are given by  $z_1 = 1/\xi_1$, and for $2 \leq k \leq n$, 
\[
z_k=\begin{cases}
\xi_k\left( \sum_{\stackrel{\alpha_i=\alpha_1}{1\leq i\leq k-1}}\xi_i \right)^{\frac{-2\langle \alpha_1,\alpha_k \rangle}{\langle \alpha_1,\alpha_1 \rangle}} &\mbox{if}\;\; \alpha_k\neq \alpha_1, \\
\frac{\xi_k}{\left( \sum_{\stackrel{\alpha_i=\alpha_1}{1 \leq i\leq k-1}}\xi_i \right)\left( \sum_{\stackrel{\alpha_i=\alpha_1}{1\leq i\leq k}}\xi_i \right)} & \mbox{if} \;\; \alpha_k=\alpha_1.\end{cases}
\]
It is remarkable (see \cite{Balazs:thesis} for some  details of the calculations) that these 
changes of coordinates  indeed change the quadratic Poisson structure expressed in the coordinates $(z_1, \ldots, z_n)$ to
the log-canonical one in the coordinates $(\xi_1, \ldots, \xi_n)$. 
\hfill $\diamond$
}
\end{example}

\subsection{The Poisson structure $\pi_n$ in coordinates, II}\label{subsec-coor-II}
Let again $\{e_\alpha \in \g_\alpha: \alpha \in \Gamma\}$ be a set of root vectors for the simple roots, which gives
rise to the coordinates $(z_1, \ldots, z_n)$ on each affine chart $\O^\gamma$ via \eqref{eq-Phi-gamma}.
Recall from
Lemma \ref{le-zizj-1} that the Poisson structure $\pi_n$ can be expressed
in the coordinates $(z_1, \ldots, z_n)$ on $\O^\gamma$ in terms of the vector fields $\sigma_i$, $1 \leq i \leq n-1$
on the Bott-Samelson variety $Z_{(s_{i+1}, \ldots, s_n)}$, given in 
(\ref{eq-Vi}). We now apply Theorem \ref{epsilonbeta} to the vector fields $\sigma_i$.

To this end, extend the set $\{e_\alpha \in \g_\alpha: \alpha \in \Gamma\}$
to a basis $\{h_\alpha\}_{\alpha \in \Gamma} \cup \{e_{\alpha} \in \g_{\alpha}\}_{\alpha \in \Delta}$ of $\g$
such that $[e_\alpha, e_{-\alpha}] = h_\alpha$ for all $\alpha \in \Delta$. Fix $\gamma =(\gamma_1, \ldots, \gamma_n) \in
\Upu$. For $1 \leq i < k \leq n$,
define two polynomials in the variables $(z_{i+1}, \ldots, z_{k-1})$ by
\begin{align}\label{eq-phi-ik}
\phi_{i,k}^\gamma(z_{i+1}, \ldots, z_{k-1}) & \,\stackrel{{\rm def}}{=}\, 
\phi_{\alpha_i}^{(\gamma_{i+1}, \ldots, \gamma_n)}(z_{i+1}, \ldots, z_{k-1}), \\
\label{eq-psi-ik}
\psi_{i,k}^\gamma(z_{i+1}, \ldots, z_{k-1}) & \,\stackrel{{\rm def}}{\,=\,}\,  
\psi_{\alpha_i}^{(\gamma_{i+1}, \ldots, \gamma_n)}(z_{i+1}, \ldots, z_{k-1})
\end{align}
by taking $\beta = \alpha_i$ and replacing $\bfu$ by $(s_{i+1}, \ldots, s_n)$  and
$\gamma$ by $(\gamma_{i+1}, \ldots, \gamma_n)$ in (\ref{phibeta}) and (\ref{psibeta}).
Here recall that when $k = i+1$, it is understood that $\Cset[z_{i+1}, \ldots, z_{k-1}] = \Cset$.
Let $1 \leq i \leq n-1$. By Theorem \ref{epsilonbeta}, the vector field $\sigma_i$ is given in the coordinates
$(z_{i+1}, \ldots, z_n)$ on the affine chart $\O^{(\gamma_{i+1}, \ldots, \gamma_n)}$ of $Z_{(s_{i+1}, \ldots, s_n)}$ by
\begin{equation}\label{eq-Vi-zk}
\sigma_i(z_k) =\begin{cases}
\phi_{i,k}^\gamma(z_{i+1},\ldots, z_{k-1})+\psi_{i,k}^\gamma(z_{i+1},\ldots,z_{k-1})z_k, &\textrm{ if }\;\gamma_k=s_k, \\
-\phi_{i,k}^\gamma(z_{i+1},\ldots, z_{k-1})z_k^2+\psi_{i,k}^\gamma(z_{i+1},\ldots, z_{k-1})z_k, &\textrm{ if }\;\gamma_k=e,
\end{cases} \hs i < k \leq n.
\end{equation}

\begin{lemma}\label{le-indep} The polynomials $\phi_{i,k}^\gamma(z_{i+1}, \ldots, z_{k-1})$
and $\psi_{i,k}^\gamma(z_{i+1}, \ldots, z_{k-1})$, where $\gamma \in \Upu$ and $1 \leq i < k \leq n$,
are independent of the extension of $\{e_\alpha: \alpha \in \Gamma\}$ to the 
basis $\{h_\alpha\}_{\alpha \in \Gamma} \cup \{e_{\alpha} \in \g_{\alpha}\}_{\alpha \in \Delta}$ of $\g$.
\end{lemma}

\begin{proof}
The coordinates $(z_1, \ldots, z_n)$ on $\O^\gamma$ and
the definition of the vector fields $\sigma_i$, $1 \leq i \leq n-1$,
on $Z_{(s_{i+1}, \ldots, s_n)}$ depend only on the choice of $\{e_\alpha: \alpha \in \Gamma\}$ and not
on its extension to the basis $\{h_\alpha\}_{\alpha \in \Gamma} \cup \{e_{\alpha} \in \g_{\alpha}\}_{\alpha \in \Delta}$ of $\g$. 
\end{proof}

The following Theorem \ref{th-zizj}, which expresses more explicitly the formula for the Poisson structure $\pi_n$ on
$\Zu$ in the affine coordinates given in Lemma \ref{le-zizj-1}, is a combination of Lemma \ref{le-zizj-1} and
Theorem \ref{epsilonbeta}.

\begin{theorem}\label{th-zizj} Let $\{e_\alpha: \alpha \in \Gamma\}$ be any choice of a set of 
root vectors for the simple roots and 
let $\gamma \in \Upu$. Then in the coordinates 
$(z_1, \ldots, z_n)$ on the affine chart $\O^\gamma$ of $\Zu$ determined by 
$\{e_\alpha: \alpha \in \Gamma\}$,
the Poisson structure $\pi_n$ is given  by
\begin{equation}\label{eq-zizj-th}
\{z_i , z_k \} = \begin{cases}
\langle \gamma^i(\alpha_i), \, \gamma^k(\alpha_k) \rangle z_iz_k,  &\mbox{if }\gamma_i=e \\
-\langle \gamma^i(\alpha_i), \, \gamma^k(\alpha_k) \rangle z_iz_k -\langle \alpha_i,
\alpha_i \rangle \sigma_i(z_k) &\mbox{if }\gamma_i=s_i
\end{cases},
\quad 1\leq i < k \leq n,
\end{equation}
where for $1 \leq i < k \leq n$, $\sigma_i(z_k) \in \Cset[z_{i+1}, \ldots, z_k]$ is given in (\ref{eq-Vi-zk}). In particular,
when $\gamma = \bfu$ is the full subexpression, $\sigma_i(z_k) \in \Cset[z_{i+1}, \ldots, z_{k-1}]$ for all
$1 \leq i < k \leq n$.
\end{theorem}

\section{The polynomial Poisson algebras $(\Cset[z_1, \ldots, z_n], \{\, , \, \}_\gamma)$}\label{sec-pi-gamma}
%\subsection{Definition of $(\Cset[z_1, \ldots, z_n], \{\, , \, \}_\gamma)$}
Throughout $\S$\ref{sec-pi-gamma}, fix a Bott-Samelson variety $\Zu$ with  $\bfu = (s_1, \ldots, s_n) = (s_{\alpha_1}, \ldots, s_{\alpha_n})$ and 
$\alpha_i \in \Gamma$ for $1 \leq i \leq n$,

\begin{definition}\label{de-pi-gamma}
{\em Given a set 
$\{e_\alpha: \alpha \in \Gamma\}$ 
of root vectors for the simple roots, for each $\gamma \in \Upu$, 
let $\{\, , \, \}_\gamma$ denote the
Poisson structure on the polynomial algebra $\Cset[z_1, \ldots, z_n]$ given by (\ref{eq-zizj-th})
in Theorem \ref{th-zizj}.
}
\end{definition}

The coordinates $(z_1, \ldots, z_n)$ on the affine charts $\O^\gamma$ of $\Zu$ depend on the choice of the
set $\{e_\alpha: \alpha \in \Gamma\}$ of root vectors for the simple roots. A different choice of such a set
gives rise to re-scalings of the coordinates and thus may result in a different Poisson bracket on the 
polynomial algebra of the coordinate functions.
We show in $\S$\ref{subsec-pi-gamma} that this is not the case.

\subsection{Re-scaling of coordinates}\label{subsec-pi-gamma}
Let $\{e_\alpha: \alpha \in \Gamma\}$ and $\{e_\alpha^\prime: \alpha \in \Gamma\}$ be two sets of choices
of root vectors for the
simple roots. 
For $\alpha \in \Gamma$, let $u_{\pm \alpha}, u_{\pm \alpha}^\prime: \Cset \to G$ be the one-parameter subgroups of $G$ respectively determined by the $\sl(2)$-triples $\{e_\alpha, e_{-\alpha}, h_\alpha\}$ and 
$\{e_\alpha^\prime, e_{-\alpha}^\prime, h_\alpha\}$  (see $\S$\ref{notations}), and let
\[
\dot{s}_\alpha=u_{\alpha}(-1)u_{-\alpha}\left(1\right)u_{\alpha}\left(-1\right)\in N_G(T) \hs \mbox{and} \hs
\dot{s}_\alpha^\prime=u_{\alpha}^\prime(-1)u_{-\alpha}^\prime\left(1\right)u_{\alpha}^\prime\left(-1\right)\in N_G(T). 
\]
For $z \in \Cset$, and $\kappa \in \{e, s_\alpha\}$, let 
\[
p_{\kappa, \alpha}(z) = u_{-\kappa(\alpha)}(z) \dot{\kappa} \in P_{s_\alpha} \hs \mbox{and} \hs
p_{\kappa, \alpha}^\prime(z) = u_{-\kappa(\alpha)}^\prime(z) \dot{\kappa}^\prime \in P_{s_\alpha},
\]
where recall that $\dot{e} = \dot{e}^\prime = e \in G$. 
For each $\gamma = (\gamma_1, \ldots, \gamma_n) \in \Upu$, one then has two sets of coordinates 
 $(z_1, \ldots, z_n)$ and $(z_1^\prime, \ldots, z_n^\prime)$ on $\O^\gamma$,
respectively by
\begin{align}\label{eq-z-0}
&\Cset^n \ni (z_1, \ldots, z_n) \longmapsto [p_{\gamma_1, \alpha_1}(z_1), \; \ldots, \; p_{\gamma_n, \alpha_n}(z_n)],\\
\label{eq-z-prime-0}
&\Cset^n \ni (z_1^\prime, \ldots, z_n^\prime) \longmapsto [p_{\gamma_1, \alpha_1}^\prime(z_1^\prime), \; \ldots, \; 
p_{\gamma_n, \alpha_n}^\prime(z_n^\prime)].
\end{align}
The main result of $\S$\ref{subsec-pi-gamma} is the following Proposition \ref{pr-indep}.

\begin{proposition}\label{pr-indep}
Let $\gamma = (\gamma_1, \ldots, \gamma_n) \in \Upu$ and let the two sets of coordinates 
$(z_1, \ldots, z_n)$ and $(z_1^\prime, \ldots, z_n^\prime)$ on $\O^\gamma$ be given as
in \eqref{eq-z-0} and \eqref{eq-z-prime-0}. For $1 \leq i < k \leq n$, let $\{z_i, z_k\} = f_{i,k} (z_1, \ldots, z_n)
\in \Cset[z_1, \ldots, z_n]$. Then 
\[
\{z_i^\prime, \; z_k^\prime\} = f_{i,k}(z_1^\prime, \, \ldots, \, z_n^\prime), \hs 1 \leq i < k \leq n.
\]
\end{proposition}

\begin{remark}\label{re-indep}
{\em It is easy to see that the two sets of coordinates are related by re-scalings, i.e., there exist $\delta_1, \ldots,
\delta_n \in \Cset^\times$ such that $z_i^\prime = \delta_i z_i$ for each $1 \leq i \leq n$. One thus has 
\[
\{z_i^\prime, \; z_k^\prime\} = \delta_i \delta_k \{z_i, \, z_k\} 
=\delta_i\delta_k f_{i,k}(z_1, \, \ldots, \, z_n)= \delta_i\delta_k f_{i,k}(\delta_1^{-1}z_1^\prime, \, \ldots, \, \delta_n^{-1}z_n^\prime),
\]
for all $1 \leq i < k \leq n$. 
Proposition \ref{pr-indep} states that the polynomials $f_{i,k}$  satisfy
\[
\delta_i\delta_k f_{i,k}(\delta_1^{-1}z_1^\prime, \, \ldots, \, \delta_n^{-1}z_n^\prime)
=f_{i,k}(z_1^\prime, \, \ldots, \, z_n^\prime), \hs  1 \leq i < k \leq n.
\]
We will show in Lemma \ref{le-tp} that the re-scaling of the coordinates comes from the action of 
an element $t \in T$, from which Proposition \ref{pr-indep} will follow. 
\hfill $\diamond$
}
\end{remark}

\begin{lemma}\label{le-pp}
Let $\alpha \in \Gamma$ and let $\lambda_\alpha \in \Cset^\times$ be such that $e_\alpha^\prime = 
\lambda_\alpha e_\alpha$. Then for $\kappa \in \{e, s_\alpha\}$ and $z \in \Cset$, one has
\begin{equation}\label{eq-alpha-prime}
p_{\kappa, \alpha}^\prime(z)  = \begin{cases} p_{\kappa, \alpha} (\lal z) \alpha^\vee (1/\lal), & \;\; \kappa = s_\alpha,\\
p_{\kappa, \alpha}(z/\lal), & \;\; \kappa = e.\end{cases}
\end{equation}
\end{lemma}

\begin{proof}
Let $\theta_\alpha,  \theta_{\alpha}^\prime:  SL(2, {\mathbb C}) \rightarrow G$ be the Lie group homomorphisms respectively
determined by the $\sl(2)$-triples $\{e_\alpha, e_{-\alpha}, h_\alpha\}$ and 
$\{e_\alpha^\prime, e_{-\alpha}^\prime, h_\alpha\}$ (see $\S$\ref{notations}). 
Then 
\[
\theta_\alpha^\prime = {\rm Ad}_{\alpha^\vee(\sal)} \circ \theta_\alpha,
\]
where ${\rm Ad}_{\alpha^\vee(\sal)}: G \to G$ denotes conjugation by $\alpha^\vee(\sal) \in T$.
It follows that 
\begin{equation}\label{eq-dot-prime}
\dot{s}_\alpha^\prime = {\rm Ad}_{\alpha^\vee(\sal)}(\dot{s}_\alpha) = \dot{s}_\alpha \,\alpha^\vee(1/\lal),
\end{equation}
and thus
\[
p_{\kappa, \alpha}^\prime(z) = {\rm Ad}_{\alpha^\vee(\sal)} (p_{\kappa, \alpha}(z)) = \begin{cases} p_{\kappa, \alpha} (\lal z) \alpha^\vee (1/\lal), & \;\; \kappa = s_\alpha,\\
p_{\kappa, \alpha}(z/\lal), & \;\; \kappa = e.\end{cases}
\]
\end{proof}

Let $\alpha \in \Gamma$ and $\lambda_\alpha \in \Cset^\times$ be as in Lemma \ref{le-pp}. 
Choose either one of the two square roots of $\lal$ in $\Cset^\times$
and denote it by $\sqrt{\lambda_\alpha}$. Note that $e_{-\alpha}^\prime =\lal^{-1} 
e_{-\alpha}$ for each $\alpha \in \Gamma$. 
Choose any $t \in T$ such that
\begin{equation}\label{eq-t-alpha}
t^\alpha = \lal, \hs \forall \; \alpha \in \Gamma.
\end{equation}
Such an element indeed exists, as it can be taken to be any of the preimages in $T \subset G$ of
the unique such element in the maximal torus $T/Z(G)$ of $G_{\rm ad} \stackrel{{\rm def}}{=} G/Z(G)$, where $Z(G)$ is the center of $G$.
Recall from (\ref{eq-P-on-Zu}) that $\cdot$ denotes the left action of $B$ on $\Zu$.

\begin{lemma}\label{le-tp}
For any $t \in T$ satisfying (\ref{eq-t-alpha}) and for any
$\gamma = (\gamma_1, \ldots, \gamma_n) \in \!\Upu$,  one has
\begin{equation}\label{eq-tpp}
t \cdot [p_{\gamma_1, \alpha_1}(z_1), \; \ldots, \; p_{\gamma_n, \alpha_n}(z_n)] =
[p_{\gamma_1, \alpha_1}^\prime(z_1), \; \ldots, \; 
p_{\gamma_n, \alpha_n}^\prime(z_n)], \hs  (z_1, \ldots, z_n) \in \!\Cset^n.
\end{equation}
\end{lemma}

\begin{proof} We prove Lemma \ref{le-tp} by induction on $n$.
When $n = 1$, $t^{-\gamma_1(\alpha_1)} = t^{\alpha_1} = \lambda_{\alpha_1}$ if $\gamma_1 = s_1$ and
$t^{-\gamma_1(\alpha_1)} = t^{-\alpha_1} = 1/\lambda_{\alpha_1}$ if $\gamma_1 = e$, so by Lemma \ref{le-pp},
\[
t \cdot [p_{\gamma_1, \alpha_1}(z_1)] = [p_{\gamma_1, \alpha_1}(t^{-\gamma_1(\alpha_1)}z)] =[p_{\gamma_1, \alpha_1}^\prime(z)].
\] 
Let $n \geq 2$ and 
assume that Lemma \ref{le-tp} holds for $n-1$. Then 
\begin{align*}
t \cdot [p_{\gamma_1, \alpha_1}(z_1), &\; \ldots, \; p_{\gamma_n, \alpha_n}(z_n)] \\
&  =
[p_{\gamma_1, \alpha_1}(t^{-\gamma_1(\alpha_1)}z_1), \; \gamma_1(t)p_{\gamma_2, \alpha_2}(z_2), \;
p_{\gamma_3, \alpha_3}(z_3), \;\ldots, \; p_{\gamma_n, \alpha_n}(z_n)].
\end{align*}
If $\gamma_1 = e$, then $p_{\gamma_1, \alpha_1}(t^{-\gamma_1(\alpha_1)}z_1) =p_{\gamma_1, \alpha_1}(z_1/\lambda_{\alpha_1})
=p_{\gamma_1, \alpha_1}^\prime(z_1)$, so (\ref{eq-tpp}) holds by the induction assumption.
Assume that $\gamma_1 = s_1$. Then by Lemma \ref{le-pp},
\begin{align*}
t \cdot [p_{\gamma_1, \alpha_1}(z_1), &\; \ldots, \; p_{\gamma_n, \alpha_n}(z_n)] \\
& =
[p_{\gamma_1, \alpha_1}^\prime(z_1), \; \alpha_1^\vee(\lambda_{\alpha_1})s_1(t)p_{\gamma_2, \alpha_2}(z_2), \;
p_{\gamma_3, \alpha_3}(z_3), \;\ldots, \; p_{\gamma_n, \alpha_n}(z_n)].
\end{align*}
Consider now  the element $\alpha_1^\vee(\lambda_{\alpha_1})s_1(t) \in T$. For every $\alpha \in \Gamma$, one has
\[
(\alpha_1^\vee(\lambda_{\alpha_1})s_1(t))^\alpha =\lambda_{\alpha_1}^{\frac{2\la \alpha, \alpha_1\ra}{\la \alpha_1, \alpha_1\ra}}
t^{s_1(\alpha)} =  t^{\frac{2\la \alpha, \alpha_1\ra}{\la \alpha_1, \alpha_1\ra}\alpha_1 + s_1(\alpha)} = t^\alpha = \lal.
\]
By the induction assumption, one then has
\begin{align*}
\alpha_1^\vee(\lambda_{\alpha_1})s_1(t)\cdot [p_{\gamma_2, \alpha_2}(z_2), &\;
p_{\gamma_3, \alpha_3}(z_3), \;\ldots, \; p_{\gamma_n, \alpha_n}(z_n)]\\
&  = 
[p_{\gamma_2, \alpha_2}^\prime(z_2), \;\ldots, \; p_{\gamma_n, \alpha_n}^\prime(z_n)] \in Z_{(s_1, \ldots, s_n)},
\end{align*}
and hence (\ref{eq-tpp}) holds.
\end{proof}

\bigskip
\noindent
{\it Proof of Proposition \ref{pr-indep}}: 
Let $t$ be any element in $T$ satisfying \eqref{eq-t-alpha}. 
By setting
\[
[p_{\gamma_1, \alpha_1}(z_1), \; \ldots, \; p_{\gamma_n, \alpha_n}(z_n)] = 
[p_{\gamma_1, \alpha_1}^\prime(z_1^\prime), \; \ldots, \; 
p_{\gamma_n, \alpha_n}^\prime(z_n^\prime)] \in \O^\gamma,
\]
and by Lemma \ref{le-tp}, one has
\[
[p_{\gamma_1, \alpha_1}(z_1^\prime), \; \ldots, \; 
p_{\gamma_n, \alpha_n}(z_n^\prime)] = t^{-1} \cdot [p_{\gamma_1, \alpha_1}(z_1), \; \ldots, \; p_{\gamma_n, \alpha_n}(z_n)].
\]
It follows from \eqref{hvectorfield} that 
\[
z_i^\prime  = (t^{-1})^* z_i = t^{\gamma^i(\alpha_i)} z_i, \hs 1 \leq i \leq n,
\]
 where $(t^{-1})^*: {\rm Reg}(\O^\gamma) \to {\rm Reg}(\O^\gamma)$
is given by $((t^{-1})^*f)(q) = f(t^{-1}\cdot q)$ for $f \in {\rm Reg}(\O^\gamma)$ and $q \in \O^\gamma$, and ${\rm Reg}(\O^\gamma)$ is
the algebra of regular functions on $\O^\gamma$.
As the action of $T$ on $(\Zu, \pi_n)$ is by Poisson isomorphisms (see $\S$\ref{piBS}), one has, for any $1 \leq i , k \leq n$,
\begin{align*}
\{z_i^\prime, \; z_k^\prime\} &= \{(t^{-1})^* z_i, \, (t^{-1})^* z_k\} = (t^{-1})^*\{z_i, \, z_k\}
= ((t^{-1})^* f_{i,k})(z_1, \ldots, z_n)\\
&= f_{i,k}(z_1^\prime, \ldots, z_n^\prime).
\end{align*}
This finishes the proof of Proposition \ref{pr-indep}.

\subsection{The Poisson algebra $({\mathbb{C}}[z_1, \ldots, z_n], \{\, , \, \}_\gamma)$ as an iterated 
$T$-Poisson Ore extension of $\Cset$}\label{subsec-Poi-Ore}
Recall \cite{GL, L-L:prime, Oh} that a Poisson polynomial algebra 
\[
A =(\Cset[z_1, \ldots, z_n], \;\{\, , \, \})
\]
is said to be an {\it iterated Poisson Ore extension} (of $\mathbb{C}$) if the Poisson bracket
$\{\, , \,\}$ satisfies
\[
\{z_i, \Cset[z_{i+1},\ldots, z_n] \} \subset z_i\Cset[z_{i+1},\ldots, z_n] + \Cset[z_{i+1},\ldots, z_n], \hs \hs
1 \leq i \leq n-1.
\]
In such a case, define the derivations $a_i$ and $b_i$ on $\Cset[z_{i+1}, \ldots, z_n]$ by
\begin{equation}\label{eq-zi-f}
\{z_i, f\} = z_i a_i(f) + b_i(f), \hs  1 \leq i \leq n-1, \;\; f \in \Cset[z_{i+1},\ldots, z_n].
\end{equation}
Then \cite{Oh} for each $1 \leq i \leq n-1$, 
$a_i$ is a Poisson derivation,
and $b_i$ an $a_i$-Poisson derivation, of the Poisson subalgebra $\Cset[z_{i+1}, \ldots, z_n]$  of the Poisson algebra
$A$,  i.e.,
\begin{align}\label{eq-a-i}
a_i\{f, \, g\} &= \{a_i(f), \; g\} + \{f, \, a_i(g)\},\\
\label{eq-b-i}
b_i\{f, \, g\} &= \{b_i(f), \; g\} + \{f, \, b_i(g)\} + a_i(f) b_i(g) - b_i(f) a_i(g)
\end{align}
for $f, g \in \Cset[z_{i+1}, \ldots, z_n]$. In this case, the Poisson algebra $A$ is also
denoted as
\begin{equation}\label{eq-A}
A = \Cset[z_n]\,[z_{n-1}; \;a_{n-1},\; b_{n-1}]\,\cdots\, [z_2;\; a_2, \;b_2]\,[z_1; \;a_1,\; b_1].
\end{equation}
An iterated Poisson Ore extension as in \eqref{eq-A} is said to be {\it nilpotent} \cite[Definition 4]{Goodearl-Yakimov:PNAS}
if $b_i$ is a locally nilpotent derivation of $\Cset[z_{i+1}, \ldots, z_n]$ for each
$1 \leq i \leq n-1$.
The following Definition \ref{de-Ore} follows \cite[Definition 4]{Goodearl-Yakimov:PNAS} but emphasizes on the 
torus actions.

\begin{definition}\label{de-Ore}
{\em 
Let $A = (\Cset[z_1, \ldots, z_n], \{\,, \, \})$ be a polynomial Poisson algebra and 
$\T$ a complex algebraic torus with Lie algebra $\t$ acting on $A$ rationally \cite{GL} by Poisson algebra automorphisms.
$A$ is said to be an {\it iterated $\T$-Poisson Ore extension (of $\Cset$)} (with respect to the given $\T$-action)
if each $z_i$, $1 \leq i \leq n$, is a weight vector for the $\T$-action with
weight $\lambda_i \in {\rm Hom}(\T, \Cset^\times)$, and if 
\[
A = \Cset[z_n]\,[z_{n-1}; \;a_{n-1},\; b_{n-1}]\,\cdots\, [z_2;\; a_2, \;b_2]\,[z_1; \;a_1,\; b_1]
\]
is an iterated Poisson Ore extension
such that there exist $h_1, \ldots, h_{n-1} \in \t$ satisfying $\lambda_i(h_i) \neq 0$ and
$a_i = h_i|_{\Cset[z_{i+1}, \ldots, z_n]}$ for each $1 \leq i \leq n-1$.
Such an iterated $\T$-Poisson Ore extension  is said to be {\it symmetric}
if 
\[
b_i(z_k) \in \Cset[z_{i+1}, \ldots, z_{k-1}], \hs 1 \leq i < k \leq n,
\]
 and if, there exist
$h_2^\prime, \ldots, h_n^\prime \in \t$ such that $\lambda_i(h_i^\prime) \neq 0$ for $2 \leq i \leq n$ and 
\begin{equation}\label{eq-lam-sym}
\lambda_i(h_k^\prime) = \lambda_k(h_i), \hs \hs 1 \leq i <  k \leq n.
\end{equation} 
Following \cite{GY:Poi} (see Remark \ref{re-nil}), a polynomial Poisson algebra which is a 
symmetric iterated $\T$-Poisson Ore extension for some torus $\T$ is  called a {\it symmetric Poisson CGL extension (of $\Cset$)}.
}
\end{definition}

\begin{remark}\label{re-a-b}
{\em For an iterated $\T$-Poisson Ore extension as in 
Definition \ref{de-Ore}, one has  
\[
\{z_i, \, z_k\} = a_i(z_k)z_i + b_i(z_k) = \lambda_k(h_i) z_iz_k + b_i(z_k) 
\in  \lambda_k(h_i)z_iz_k + \Cset[z_{i+1}, \ldots, z_n]
\]
for all $1 \leq i < k \leq n$, a property referred to as {\it semi-quadratic} in \cite[Definition 4]{Goodearl-Yakimov:PNAS}.
\hfill $\diamond$
}
\end{remark}

\begin{remark}\label{re-h-bi}
{\em 
Let $A$ be an iterated $\T$-Poisson Ore extension as in 
Definition \ref{de-Ore}. Then 
\begin{equation}\label{eq-h-b}
[h|_{\Cset[z_{i+1}, \ldots, z_n]}, \; b_i] = \lambda_i(h) b_i, \hs \hs 1 \leq i \leq n-1, \; \; h \in \t,
\end{equation}
where the left hand side denotes the commutator bracket between the two derivations
$h|_{\Cset[z_{i+1}, \ldots, z_n]}$ and $b_i$ of $\Cset[z_{i+1}, \ldots, z_n]$. 
In fact,  \eqref{eq-h-b} is equivalent to 
\[
[h|_{\Cset[z_{i+1}, \ldots, z_n]}, \; b_i](z_k) = \lambda_i(h) b_i(z_k), \hs 
1\leq i < k \leq n, \;  h \in \t,
\]
 which, by the fact that $z_j$ is a
$\T$-weight vector with weight $\lambda_j$ for each $1 \leq j \leq n$, is in turn equivalent to 
\[
h(\{z_i, z_k\}) = \{h(z_i), z_k\} + \{z_i, h(z_k)\}, \hs h \in \t, \; 1 \leq i < k \leq n,
\]
 which follows from the 
assumption that 
$\T$ acts on $A$ by Poisson automorphisms. In particular, one has 
\[
[a_i, \, b_i] = \lambda_i(h_i) b_i, \hs \hs 1 \leq i \leq n-1.
\]
Let $1 \leq i \leq n-1$ and consider the $2$-dimensional Lie bialgebra $\b_2= \Cset x + \Cset y$ with Lie bracket
$[x, y] = 2y$ and Lie co-bracket $\delta: \b_2 \to \wedge^2 \b_2$ given by $\delta(x) = 0$
and $\delta(y) = -\frac{\lambda_i(h_i)}{2}x \wedge y$.  Consider the Poisson subalgebra $A_{i+1} = \Cset[z_{i+1}, \ldots, z_n]$ 
of $A$ and let ${\rm Der}_{\Cset}(A_{i+1})$ be the Lie algebra of derivations (for the commutative algebra structure)
of $A_{i+1}$. Define the Lie algebra anti-homomorphism $\sigma: \b_2 \to {\rm Der}_{\Cset}(A_{i+1})$ by
\[
\sigma(x) = -\frac{2}{\lambda_i(h_i)}a_i, \hs \sigma(y) = \frac{1}{\lambda_i(h_i)}b_i. 
\]
Then \eqref{eq-a-i} and \eqref{eq-b-i} are equivalent to 
$\sigma$ being a {\it left Poisson action of the Lie bialgebra} $(\b_2, \delta)$ on the Poisson algebra
$A_{i+1}$ (see \cite[$\S$2]{Lu-Mou:mixed}). Let $\b_2^*$ be the dual vector space of $\b_2$ with basis $(x^*, y^*)$ dual  
to the basis $(x, y)$ of $\b_2$. Then the dual Lie bialgebra of $(\b_2, \delta)$ is $\b_2^*$ with  
Lie bracket $[x^*, y^*] = -\frac{\lambda_i(h_i)}{2}y^*$ and Lie co-bracket 
$x^* \mapsto 0$ and $y^* \mapsto 2 x^* \wedge y^*$. 
Let $\rho: \b_2^* \to {\rm Der}_{\Cset} \Cset[z_i]$ be the Lie algebra homomorphism given by 
\[
\rho(x^*) = \frac{\lambda_i(h_i)}{2}z_i \partial /\partial z_i, \hs \rho(y^*) = -\lambda_i(h_i)\partial /\partial z_i.
\]
 Then $\rho$ is a 
{\it right Poisson action of the Lie bialgebra}
$\b_2^*$ on $\Cset[z_i]$ with the trivial Poisson bracket. The Poisson Ore extension $A_i :=
\Cset[z_i, z_{i+1}, \ldots, z_n]$ of $A_{i+1}$ with the Poisson bracket given in \eqref{eq-zi-f} can now be interpreted as
the {\it mixed product Poisson structure} on $A_i = \Cset[z_i] \otimes A_{i+1}$ defined by the pair $(\rho, \sigma)$ of
Poisson actions of Lie bialgebras introduced in \cite{Lu-Mou:mixed}. 
\hfill $\diamond$
}
\end{remark}

\begin{remark}\label{re-nil}
{\em A symmetric iterated $\T$-Poisson Ore extension  is automatically nilpotent. Indeed, let $1 \leq i \leq n-1$
and let the notation be as in Definition \ref{de-Ore}.
To show that $b_i$ is locally nilpotent as a derivation of $\Cset[z_{i+1}, \ldots, z_n]$,
observe first that for integers $m, N \geq 1$ and $f_1, f_2, \ldots, f_m
\in \Cset[z_{i+1}, \ldots, z_n]$, $b_i^N(f_1f_2 \cdots f_m)$ is a linear combination of terms 
of the form $b_i^{N_1}(f_1) b_i^{N_2}(f_1) \cdots b_i^{N_m}(f_m)$ with $N_1 + N_2 + \cdots + N_m = N$. 
Thus $b_i$ is locally nilpotent if for each $i < k \leq n$, $b_i^{N_k} (z_k) = 0$ 
for some integer $N_k \geq 1$. As $b_i(z_{i+1}) \in \Cset$, one has $b_i^2(z_{i+1}) = 0$. Assume that 
there exist $N_j \geq 1$ such that $b_i^{N_j}(z_j) = 0$ for $i+1 \leq j \leq k-1$. As
$b_i(z_k) \in \Cset[z_{i+1}, \ldots, z_{k-1}]$, the above observation shows that
there is an integer $N_k \geq 1$
such that $b_i^{N_k}(z_k) = 0$. Induction on $k$ now shows that $b_i$ is locally nilpotent.
Observe also that if $A$ is a symmetric iterated $\T$-Poisson Ore extension, then for $1 \leq i < k \leq n$,
\begin{equation}\label{eq-zizk-0}
\{z_i, \, z_k\} = \lambda_k(h_i) z_iz_k + b_i(z_k) \in \lambda_k(h_i)z_iz_k + \Cset[z_{i+1}, \ldots, z_{k-1}]
\subset \Cset[z_i, \ldots, z_k].
\end{equation}
Consequently, $\Cset[z_i, \ldots, z_k]$ is a Poisson subalgebra of $A$ for all $1 \leq i < k \leq n$.
\hfill $\diamond$
}
\end{remark}

\begin{lemma}\label{le-reverse}\cite{GY:Poi}
If $A = (\Cset[z_1, \ldots, z_n], \{\,, \, \})$ is a symmetric iterated
$\T$-Poisson Ore extension, then, with respect to the same $\T$-action,
$A$ is a $\T$-Poisson Ore extension in the reversed order of the variables. More precisely,
in the notation of Definition \ref{de-Ore}, for each $2 \leq k \leq n$,
$\Cset[z_1, \ldots, z_{k-1}]$ is a 
Poisson subalgebra of $A$, and  
\begin{equation}\label{eq-f-z}
\{f, \, z_k\} = a_k^\prime(f)z_k + b_k^\prime(f), \hs f \in \Cset[z_1, \ldots, z_{k-1}],
\end{equation}
where $a_k^\prime = h_k^\prime|_{\Cset[z_1, \ldots, z_{k-1}]}$
as a derivation of $\Cset[z_1, \ldots, z_{k-1}]$ and $b_k^\prime$ is the unique derivation of
$\Cset[z_1, \ldots, z_{k-1}]$ such that $b_k^\prime(z_i) = b_i(z_k) \in \Cset[z_{i+1}, \ldots, z_{k-1}]$ for
$1 \leq i \leq k-1$. Moreover, 
for any $h \in \t$, $[h|_{\Cset[z_1, \ldots, z_{k-1}]}, \, b_k^\prime] = \lambda_k(h) b_k^\prime$
as derivations of $\Cset[z_1, \ldots, z_{k-1}]$.
\end{lemma}

\begin{proof}
It follows from \eqref{eq-zizk-0} that $\Cset[z_1, \ldots, z_{k-1}]$ is a 
Poisson subalgebra of $A$ for every $2 \leq k \leq n$. The assumption that $\lambda_i(h_k^\prime) = \lambda_k(h_i)$
for all $1 \leq i < k \leq n$ and the definition of the $b_k^\prime$'s imply that \eqref{eq-f-z} holds
for $f = z_i$ for each $i < k$, so it holds for all $f \in \Cset[z_1, \ldots, z_{k-1}]$. 
Let $h \in \t$ and $2 \leq k \leq n$. Then for each $1 \leq i \leq k-1$, using 
\eqref{eq-h-b}, one has $h(b_i(z_k)) - b_i(h(z_k)) = \lambda_i(h) b_i(z_k)$, from which one has
\[
h(b_i(z_k)) - \lambda_i(h) b_i(z_k) = b_i(h(z_k)) = \lambda_k(h) b_i(z_k),
\]
 and it follows that
\[
h(b_k^\prime(z_i)) - b_k^\prime(h(z_i)) = 
h(b_i(z_k))-\lambda_i(h) b_i(z_k) = \lambda_k(h) b_i(z_k) = \lambda_k(h) b_k^\prime(z_i).
\]
This proves that 
$[h|_{\Cset[z_1, \ldots, z_{k-1}]}, \, b_k^\prime] = \lambda_k(h) b_k^\prime$
as derivations of $\Cset[z_1, \ldots, z_{k-1}]$.
\end{proof}

\begin{notation}\label{note-reverse}
{\em 
In the context of Lemma \ref{le-reverse}, we also write
\begin{equation}\label{eq-reverse}
A = \Cset[z_1]\,[z_2; \;a_2^\prime, \;b_2^\prime]\,\cdots \,[z_{n-1}; \;a_{n-1}^\prime, \;b_{n-1}^\prime]\,
[z_n; \;a_n^\prime, \;b_n^\prime].
\end{equation}
}
\end{notation}

\bigskip
Returning to the Bott-Samelson variety $\Zu$ with the Poisson structure $\pi_n$, where $\bfu = (s_1, \ldots, s_n) = (s_{\alpha_1}, \ldots, s_{\alpha_n})$,
choose again any set $\{e_\alpha: \alpha \in \Gamma\}$ of root vectors for the simple roots, so that one has 
coordinates $(z_1, \ldots, z_n)$ on $\O^\gamma$ for each $\gamma \in \Upu$. Fix $\gamma \in \Upu$ and consider
the Poisson polynomial algebra
$(\Cset[z_1, \ldots, z_n], \{\, , \, \}_\gamma)$. Recall again that the maximal torus $T$ acts on  
$\O^\gamma$ by \eqref{hvectorfield}, which gives rise to 
a rational action of $T$ on $(\Cset[z_1, \ldots, z_n], \{\, , \, \}_\gamma)$ by Poisson automorphisms. More precisely,
\begin{equation}\label{eq-T-z}
t \cdot z_i = t^{-\gamma^i(\alpha_i)} z_i, \hs 1 \leq i \leq n.
\end{equation}
For $h \in \h = {\rm Lie}(T)$, denote by $\partial_h$ the Poisson derivation of
$(\Cset[z_1, \ldots, z_n], \{\, , \, \}_\gamma)$ generating the $T$-action in the direction of $h$, i.e,
\begin{equation}\label{eq-x-z}
\partial_h(z_i) = -\gamma^i(\alpha_i)(h)z_i, \hs 1 \leq i \leq n, \; h \in \h.
\end{equation}
Note that both the $T$-action and the derivations $\partial_h$, $h \in \h$, on $\Cset[z_1, \ldots, z_n]$
depend on $\gamma$,  but for notational simplicity we do not include the dependence on $\gamma$ in the notation.
For $1 \leq i \leq n-1$, recall also  the vector field $\sigma_i$ on
the Bott-Samelson variety $Z_{(s_{i+1}, \ldots, s_n)}$ defined in \eqref{eq-Vi}, and recall that the induced derivation on
$\Cset[z_{i+1}, \ldots, z_n]$, identified with the algebra of regular functions on $\O^{(s_{i+1}, \ldots, s_n)}
\subset  Z_{(s_{i+1}, \ldots, s_n)}$ is also denoted by $\sigma_i$.

\begin{theorem}\label{thm-Ore}
For each $\gamma \in \Upu$, $(\Cset[z_1, \ldots, z_n], \{\, , \, \}_\gamma)$
is an iterated $T$-Poisson Ore extension of $\Cset$ with respect to the $T$-action on 
given in \eqref{eq-T-z}. More explicitly,
\begin{equation}\label{eq-pi-gamma-Ore}
(\Cset[z_1, \ldots, z_n], \{\, , \, \}_\gamma) = 
\Cset[z_n]\,[z_{n-1}; \;a_{n-1},\; b_{n-1}]\,\cdots \,[z_2;\; a_2, \;b_2]\,[z_1; \;a_1,\; b_1],
\end{equation}
where for $1 \leq i \leq n-1$, 
\begin{equation}\label{eq-pi-gamma-Ore-1}
a_i = -\frac{\la \alpha_i, \alpha_i\ra}{2} \partial_{\gamma^{i-1}(h_{\alpha_i})}|_{\Cset[z_{i+1}, \ldots, z_n]},
\hs
b_i =  \begin{cases}0, & \;\;\; \mbox{if} \;\; \gamma_i = e,\\
- \la \alpha_i, \alpha_i\ra\sigma_i, 
& \;\;\; \mbox{if} \;\; \gamma_i = s_i.
\end{cases}
\end{equation}
When $\gamma = \bfu$, the extension is symmetric. More explicitly, for $\gamma = \bfu$, one also has
\begin{equation}\label{eq-A-u}
A = \Cset[z_1]\,[z_2; \;a_2^\prime, \;b_2^\prime]\,\cdots \,[z_{n-1}; \;a_{n-1}^\prime, \;b_{n-1}^\prime]\,
[z_n; \;a_n^\prime, \;b_n^\prime],
\end{equation}
where for $2 \leq k \leq n$, $a_k^\prime = -\frac{\la \alpha_k, \alpha_k\ra}{2} \partial_{\gamma^{k-1}(h_{\alpha_k})}|_{\Cset[z_{1}, \ldots, z_{k-1}]}$, and $b_k^\prime$ is the unique derivation
of $\Cset[z_{1}, \ldots, z_{k-1}]$ such that $b_k^\prime(z_i) = -\la \alpha_i, \alpha_i\ra \sigma_i(z_k)$ for 
$1 \leq i \leq k-1$.
\end{theorem}

\begin{proof} Let $\gamma = (\gamma_1, \ldots, \gamma_n) \in \Upu$ and let
$\lambda_i = -\gamma^i(\alpha_i)$ for $1 \leq i \leq n$.
By \eqref{eq-T-z}, $z_i$ is a weight vector for the $T$-action on $\Cset[z_1, \ldots, z_n]$ 
with weight $\lambda_i$.
For $1 \leq i \leq n$, define $h_i \in \h = {\rm Lie}(T)$ by
\begin{equation}\label{eq-hi}
h_i = -\frac{\la \alpha_i, \alpha_i\ra}{2} \gamma^{i-1}(h_{\alpha_i}) =\displaystyle 
\begin{cases}-\frac{\la \alpha_i, \alpha_i\ra}{2}\gamma^{i}(h_{\alpha_i}), & \;\;\; \mbox{if} \;\;\; \gamma_i =e,
\vspace{.1in}
\\
\frac{\la \alpha_i, \alpha_i\ra}{2}\gamma^{i}(h_{\alpha_i}), & \;\;\; \mbox{if} \;\;\; \gamma_i =s_i.\end{cases}
\end{equation}
Then for $1 \leq i < k \leq n$, 
\[
\partial_{h_i}(z_k) =\lambda_k(h_i) z_k = -\gamma^k(\alpha_k)(h_i) z_k = \la \gamma^{i-1}(\alpha_i), \, \gamma^k(\alpha_k)\ra z_k.
\]
It now follows from Theorem \ref{th-zizj} that \eqref{eq-pi-gamma-Ore} holds with the $a_i$'s and $b_i$'s given by
\eqref{eq-pi-gamma-Ore-1}. Moreover, for each $1 \leq i \leq n$, $\lambda_i(h_i) \neq 0$, as 
\begin{equation}\label{eq-lambda-h-i}
\lambda_i(h_i) = \la \gamma^{i-1}(\alpha_i), \, \gamma^{i}(\alpha_i)\ra = \la \alpha_i, \, \gamma_i(\alpha_i)\ra
=\begin{cases} \la \alpha_i, \, \alpha_i\ra, & \;\; \gamma_i = e,\\
-\la \alpha_i, \, \alpha_i\ra, & \;\; \gamma_i = s_i.\end{cases}
\end{equation}
Thus $(\Cset[z_1, \ldots, z_n], \{\, , \, \}_\gamma)$
is an iterated $T$-Poisson Ore extension of $\Cset$.

Assume now that $\gamma = \bfu$ is the full subexpression of $\bfu$. In this case, let
\[
h_i = -\frac{\la \alpha_i, \, \alpha_i\ra}{2} \gamma^{i-1}(h_{\alpha_i}) = -\frac{\la \alpha_i, \, \alpha_i\ra}{2}
s_1s_2 \cdots s_{i-1}(h_{\alpha_i}) \in \h, \hs 1 \leq i \leq n,
\]
and let $h_k^\prime = h_k$ for $2 \leq k \leq n$.
With $\lambda_i = s_1s_2 \cdots s_{i-1}(\alpha_i)$, one has, for $1 \leq i < k \leq n$,
\[
\lambda_i(h_k^\prime) =-\la \gamma^i(\alpha_i), \; \gamma^k(\alpha_k)\ra =
-\la s_1s_2 \cdots s_{i-1}(\alpha_i), \; s_1s_2 \cdots s_{k-1}(\alpha_k)\ra = \lambda_k(h_i).
\]
By Theorem \ref{th-zizj}, one also has $b_i(z_k) \in \Cset[z_{i+1}, \ldots, z_{k-1}]$ for $1 \leq i < k \leq n$. 
This shows that $(\Cset[z_1, \ldots, z_n], \{\, , \, \}_\bfu)$, as an
iterated $T$-Poisson Ore extension of $\Cset$ with respect to the $T$-action
given in \eqref{eq-T-z}, is symmetric.  By Lemma \ref{le-reverse}, \eqref{eq-A-u} holds.
\end{proof}

\begin{remark}\label{remark-h-b}
{\em We already know from Remark \ref{re-h-bi} that
for $h \in \t$ and $1 \leq i \leq n-1$,  the 
two derivations $a_h :=\partial_h|_{\Cset[z_{i+1}, \ldots, z_n]}$ and $b_i$
on $\Cset[z_{i+1}, \ldots, z_n]$ in Theorem \ref{thm-Ore} satisfy $[a_h, b_i] = \lambda_i(h)b_i$.
This can also be checked directly: it clearly holds when $\gamma_i = e$.  Assume that $\gamma_i = s_i$.
In the notation of \eqref{eq-sigma-k} and by Lemma \ref{le-Z-Pi-1}, one has
$a_h = \sigma^{(i+1)}_{(\gamma^i)^{-1}(h)}$ and $b_i = -\la \alpha_i, \alpha_i\ra \sigma^{(i+1)}_{e_{\alpha_i}}$. 
Thus
\[
[a_h, \, b_i] = -\la \alpha_i, \alpha_i\ra 
\left[\sigma^{(i+1)}_{(\gamma^i)^{-1}(h)}, \, \sigma^{(i+1)}_{e_{\alpha_i}}\right]
=  \la \alpha_i, \alpha_i\ra\sigma^{(i+1)}_{[(\gamma^i)^{-1}(h), e_{\alpha_i}]}
=\lambda_i(h)b_i.
\]
\hfill $\diamond$
}
\end{remark}

\begin{remark}\label{not-nil}
{\em 
For an arbitrary $\gamma \in \Upu$, $(\Cset[z_1, \ldots, z_n], \{\, ,\, \}_\gamma)$ 
expressed as an iterated $T$-Poisson Ore extension as in \eqref{eq-pi-gamma-Ore}
is not necessarily a Poisson CGL extension in the sense of \cite{GY:Poi}, as the definition in \cite{GY:Poi} 
requires the derivations $b_i$ be locally nilpotent. In Example \ref{ex-1} for $\gamma = (s_{\alpha_1}, e, e)$,
the derivation $b_1$ on 
$\Cset[z_2, z_3]$ is given by $b_1(z_2) = 0$ and $b_1(z_3) = 2z_3^2$ which is not locally nilpotent. 
\hfill $\diamond$
}
\end{remark}

\subsection{The Poisson structure $\pi_n$ in $\O^\bfu$}\label{subsec-u}
We now look in more detail at the Poisson polynomial algebra 
$(\Cset[z_1, \ldots, z_n], \{\, , \, \}_\bfu)$.
In this case, $T$ acts on 
$\Cset[z_1, \ldots, z_n]$  by
\begin{equation}\label{eq-T-z-u}
t \cdot z_i = t^{s_1s_2 \cdots s_{i-1}(\alpha_i)} z_i, \hs t \in T, \; 1 \leq i \leq n,
\end{equation}
and the Poisson structure $\{\, , \, \}_\bfu$ on $\Cset[z_1, \ldots, z_n]$ is given by 
\begin{equation}\label{eq-zizk-u}
\{z_i, z_k\}_\bfu = c_{i,k}z_iz_k -\la \alpha_i, \alpha_i\ra \sigma_i(z_k) =c_{i,k}z_iz_k + b_k^\prime(z_i) , \hs 1 \leq i < k \leq n,
\end{equation}
where for $1 \leq i, k \leq n$, 
\begin{equation}\label{eq-cik}
c_{i,k} = -\la \gamma^i(\alpha_i), \, \gamma^k(\alpha_k)\ra = -\la s_1s_2 \cdots s_{i-1}(\alpha_i), \; s_1s_2 \ldots s_{k-1}(\alpha_k)\ra,
\end{equation}
$\sigma_i$ is the derivation on $\Cset[z_{i+1}, \ldots,z_{k-1}]$ corresponding to the vector field
on the Bott-Samelson variety $Z_{(s_{i+1}, \ldots, s_n)}$ generating the $B$-action 
in the direction of $e_{\alpha_i}$ (see \eqref{eq-Vi}),
and $b_k^\prime$ is the unique derivation on $\Cset[z_1, \ldots, z_{k-1}]$ such that
$\b_k^\prime (z_i) = -\la \alpha_i, \alpha_i\ra \sigma_i(z_k)$.

We now give the geometric meaning of the derivation $b_k^\prime$ on $\Cset[z_1, \ldots, z_{k-1}]$.
To this end, consider the quotient manifold
\[
F_{-n}^\prime = B_-\backslash G \times_{B_-} G \times \cdots \times_{B_-} G
\]
of $G^n$ by $(B_-)^n$, where 
$(B_-)^n$ acts on  $G^n$ from the left by
\begin{equation}\label{eq-BP-left}
(b_1, b_2, \ldots, b_n) \cdot (g_1, g_2, \ldots, g_n)  =(b_1g_1b_2^{-1}, \, b_2g_2b_3^{-1}, \ldots, b_ng_n), \hs b_j \in B_-, \, g_j \in G. 
\end{equation}
Let $\rho_-: G^n \to F_{-n}^\prime$
be the natural projection.
Similar to the case of the quotient manifold $F_n$ in \eqref{eq-Fn}, the product Poisson structure $\pi_{\rm st}^n$ on $G^n$
projects by $\rho_-$ to a well-defined Poisson structure on $F_{-n}^\prime$, which will be denoted by $\pi_{-n}^\prime$.
Let $P_{-s_i} = B_- \cup B_-s_i B_-$ for $1 \leq i \leq n$. 
As each $P_{-s_i}$ is a Poisson submanifold of $(G, \pist)$, 
the closed submanifold
\[
Z_{-\bfu}^\prime = B_-\backslash P_{-s_1} \times_{B_-} P_{-s_2} \times \cdots \times_{B_-} P_{-s_n}
\]
of $F_{-n}^\prime$ is a Poisson submanifold with respect to $\pi_{-n}^\prime$. 
We will also call $Z_{-\bfu}^\prime$ a Bott-Samelson variety.
Note that for each $1 \leq i \leq n$, one has
\[
u_{\alpha_i}(z) \dot{s}_i = \dot{s}_i u_{-\alpha_i}(-z), \hs z \in \Cset.
\]
Setting $\rho_-(g_1, g_2, \ldots, g_n) = [g_1, g_2, \ldots, g_n]_- \in F_{-n}^\prime$ for $(g_1, g_2, \ldots, g_n) \in G^n$,
it follows that one has the open affine chart 
\[
\O^{\prime, \bfu}_-:= B_-\backslash (B_-s_1 B_-) \times_{B_-} (B_-s_2 B_-) \times \cdots \times_{B_-} (B_-s_n B_-)
\]
of $Z_{-\bfu}^\prime$, with the parametrization by $\Cset^n$ via
\begin{equation}\label{eq-coor-prime}
\Cset^n \ni (z_1, z_2, \ldots, z_n) \longmapsto 
[u_{\alpha_1} (z_1)\dot{s}_{\alpha_1}, \; u_{\alpha_2} (z_2)\dot{s}_{\alpha_2}, \;\ldots, \; 
u_{\alpha_n} (z_n)\dot{s}_{\alpha_n}]_- \in \O^{\prime, \bfu}_-.
\end{equation}
The restriction of the Poisson structure $\pi_{-n}^\prime$ to $\O^{\prime, \bfu}_-$ will also be denoted by
$\pi_{-n}^\prime$.

\begin{proposition}\label{pr-dual-pair}
The map $I: (\O^{\bfu}, \, \pi_n) \to (\O^{\prime, \bfu}_-, \; \pi_{-n}^\prime)$ given by
\begin{align*}
[u_{\alpha_1} (z_1)\dot{s}_{\alpha_1}, \; u_{\alpha_2} (z_2)\dot{s}_{\alpha_2}, \;\ldots, &\; 
u_{\alpha_n} (z_n)\dot{s}_{\alpha_n}] \\
&\hs \hs\longmapsto 
[u_{\alpha_1} (z_1)\dot{s}_{\alpha_1}, \; u_{\alpha_2} (z_2)\dot{s}_{\alpha_2}, \;\ldots, \; 
u_{\alpha_n} (z_n)\dot{s}_{\alpha_n}]_-, 
\end{align*}
where $(z_1, z_2, \ldots, z_n) \in \Cset^n$, 
is a Poisson anti-isomorphism.
\end{proposition}

\begin{proof}
Let $\rho: G^n \to F_n$ be the natural projection, so that $\pi_n = \rho(\pi_{\rm st}^n)$.
It is proved in \cite[$\S$8]{Lu-Mou:mixed} that the pair 
\[
\rho: \;\; (G^n, \, \pi_{\rm st}^n) \lrw (F_n, \, \pi_n) \hs
\mbox{and} \hs \rho_-: \;\; (G^n, \, \pi_{\rm st}^n) \lrw \left(F_{-n}^\prime, \;\, \pi_{-n}^\prime\right)
\]
of Poisson submersions 
is a Poisson pair (see $\S$\ref{sec-Poisson-pairs} the Appendix), i.e., the map
\[
(\rho, \; \rho_-): \;\; (G^n, \, \pi_{\rm st}^n) \lrw (F_n \times F_{-n}^\prime, \;  \pi_n \times \pi_{-n}^\prime),
\;\; (g, \; g^\prime) \longmapsto (\rho(g), \; \rho_-(g')), \hs g, g' \in G^n,
\]
is Poisson.
For $\alpha \in \Gamma$, let $\Sigma_\alpha$ be the symplectic leaf of $\pist$ in $G$ through the point $\dot{s}_\alpha 
\in G$. To describe the two-dimensional symplectic manifold $(\Sigma_\alpha, \pist|_{\Sigma_\alpha})$,
consider the surface
\[
\Sigma = \{(p, q, t) \in \Cset^3: \; t^2(1-pq) = 1\}
\]
in $\Cset^3$ and equip $\Sigma$ with the Poisson structure $\pi$ given by
\begin{equation}\label{eq-pi}
\{p, q\} = 2(1-pq), \hs \{p, t\} = pt, \hs \{q, t\} = -qt.
\end{equation}
A calculation in $SL(2, \Cset)$ shows that 
the embedding
\[
J: \;\; \Sigma \longrightarrow SL(2, \Cset), \;\;\; (p, q, t) \longmapsto 
\left(\begin{array}{cc}pt & -t \\ t & -qt\end{array}\right), \hs (p, q, t) \in \Sigma,
\]
identifies $(\Sigma, \pi)$ as the symplectic leaf through 
$\displaystyle \left(\begin{array}{cc}0 & -1 \\ 1 & 0\end{array}\right) \in SL(2, \Cset)$ of
the Poisson structure $\pi_{\scriptscriptstyle{SL(2, \Cset)}}$ on $SL(2, \Cset)$ in \eqref{eq-pi-alpha}. 
Using the Poisson homomorphism $\theta_\alpha$ in \eqref{theta-Poisson}, one sees \cite{KZ} that
\[
\Sigma_\alpha = \{g_\alpha(p, q, t): \; (p, q, t) \in \Sigma\},
\]
and $\pist|_{\Sigma_\alpha} = \frac{\la \alpha, \alpha\ra}{2}(\theta_\alpha \circ J)(\pi)$, 
where for $(p, q, t) \in \Sigma$, 
\begin{equation}\label{eq-g-alpha}
g_\alpha(p, q, t) = \theta_\alpha\left(\begin{array}{cc}pt & -t \\ t & -qt\end{array}\right)= 
u_\alpha(p) \ds_\alpha \check{\alpha}(t) u_\alpha(-q)
=u_{-\alpha}(q) \check{\alpha}(t) \ds_\alpha u_{-\alpha}(-p).
\end{equation}
Consider now the product manifold $\Sigma_{\bfu} = \Sigma_{\alpha_1} \times \Sigma_{\alpha_2} \times \cdots \times \Sigma_{\alpha_n}$
and denote the restriction of the product Poisson structure $\pi_{\rm st}^n$ to
$\Sigma_{\bfu}$ still by $\pi_{\rm st}^n$. 
It follows from \eqref{eq-g-alpha} that 
\[
\rho(\Sigma_{\bfu}) = \O^\bfu
\hs \mbox{and} \hs
\rho_-(\Sigma_{\bfu}) = \O^{\prime,\bfu}_{-},
\]
and, denoting again by $\rho$ (resp. $\rho_-$) the induced map from  
$\Sigma_{\bfu}$ to $\O^\bfu$ (resp. to $\O^{\prime, \bfu}_-$), 
\begin{equation}\label{eq-rho-pm}
\rho: \; (\Sigma_{\bfu}, \, \pi_{\rm st}^n) \lrw 
(\O^\bfu, \; \pi_n)
\hs \mbox{and} \hs
\rho_-: \; (\Sigma_{\bfu}, \; \pi_{\rm st}^n) \lrw (\O^{\prime,\bfu}_{-}, \; \pi_{-n}^\prime)
\end{equation}
are Poisson submersions and form a Poisson pair. Moreover,
the submanifold 
\[
L:=\{ (u_{\alpha_1} (z_1)\dot{s}_{\alpha_1}, \; u_{\alpha_2} (z_2)\dot{s}_{\alpha_2}, \;\ldots, \; 
u_{\alpha_n} (z_n)\dot{s}_{\alpha_n}): \;  (z_1, z_2, \ldots, z_n) \in \Cset^n\}
\]
of $\Sigma_{\bfu}$ is Lagrangian with respect to  $\pi_{\rm st}^n$, and it is clear that 
$\rho|_L: L \to \O^\bfu$ is a diffeomorphism. It now follows from Lemma 
\ref{le-graph-coisotropic} in the Appendix that  $I= \rho_- \circ (\rho|_L)^{-1}: (\O^{\bfu}, \, \pi_n) \to (\O^{\prime, \bfu}_-, \; \pi_{-n}^\prime)$ 
is a Poisson anti-isomorphism.
\end{proof}

We now prove a fact similar to that in Lemma \ref{le-Z-Pi-1}: let $(X, \piX)$ be a Poisson manifold with 
a right Poisson action by the Poisson Lie group $(B_-, \pist)$, 
let $\alpha$ be a simple root, and consider the quotient manifold
$Z = X \times_{B_-} P_{-s_\alpha}$ (see notation in $\S$\ref{piBS}) 
equipped with Poisson structure $\piZ$ which is the projection to $Z$ of the product Poisson structure $\piX \times \pist$ on $X \times
P_{-s_\alpha}$. Denote by $[x, p]$ the image of $(x, p) \in X \times P_{-s_\alpha}$ in $Z$. Fix any 
$\sl(2, \Cset)$-triple $\{e_\alpha, e_{-\alpha}, h_\alpha\}$ and consider
\[
\phi: \;\; X \times \Cset \lrw Z_0, \;\; (x, z) \longmapsto [x, \; u_\alpha(z)\dot{s}_\alpha], \hs x \in X, \, z \in \Cset.
\]
Then $\phi$ is an embedding, and we regard $\phi$ as a 
diffeomorphism from $X \times \Cset$ to $Z_0 = \phi(X \times \Cset)$.
For $\xi \in \b$, let $\sigma_\xi^\prime$ be the vector field on $X$ defined by
\[
\sigma_\xi^\prime(x) = \frac{d}{dt}|_{t=0} (x \exp(t\xi)), \hs x \in X.
\]
Using the second part of \eqref{eq-lpi}, the proof of the following Lemma \ref{le-Z-Pi-2} is similar to that of Lemma \ref{le-Z-Pi-1} and is omitted.

\begin{lemma}\label{le-Z-Pi-2}
With the notation as above, one has
\[
\phi^{-1}(\piZ)(x, z) = \piX(x) + 
\frac{\la \alpha, \alpha \ra}{2} \frac{d}{dz} \wedge \left(z\sigma_{h_\alpha}^\prime(x)+2\sigma_{e_{-\alpha}}^\prime(x)\right).
\]
\end{lemma}

Returning now to the Bott-Samelson variety $Z_{-\bfu}^\prime$ for $\bfu = (s_1, \ldots, s_n) = (s_{\alpha_1}, \ldots,
s_{\alpha_n})$,
let $2 \leq k \leq n$, and consider 
\[
Z_{-(s_1, \ldots, s_{k-1})}^{\prime} = B_-\backslash P_{-s_1} \times_{B_-} P_{-s_2} \times \cdots \times_{B_-} P_{-s_{k-1}}.
\]
Denote again by $[p_1, \ldots, p_{k-1}]_-$ the image of
$(p_1, \ldots, p_{k-1}) \in P_{-s_1} \times \cdots \times P_{-s_{k-1}}$ in $Z_{-(s_1, \ldots, s_{k-1})}^{\prime}$, and let
$B_-$ act on $Z_{-(s_1, \ldots, s_{k-1})}^{\prime}$ from the right by 
\[
[p_1, \, \ldots, \, p_{k-2}, \, p_{k-1}]_- \cdot b_- = [p_1, \, \ldots, \, p_{k-2}, \, p_{k-1}b_-],
\hs b_- \in B_-, \, p_i \in P_{-s_i}, \, 1 \leq i \leq k-1.
\]
For $\xi \in \b_-$, denote by $\sigma_{\xi}^{\prime, (k-1)}$
the vector field on $Z_{-(s_1, \ldots, s_{k-1})}^{\prime}$ given by 
\begin{equation}\label{eq-xi-k}
\sigma_{\xi}^{\prime, (k-1)}([p_1, \ldots, p_{k-2}, p_{k-1}]) = \frac{d}{dt}|_{t=0} [p_1, \ldots, p_{k-2}, p_{k-1}\exp(t\xi)]_-,
\end{equation}
where $p_i \in P_{-s_i}$ for $1 \leq i \leq k-1$,
so $\sigma_{\xi}^{\prime, (k-1)}$ generates the action of $B_-$ on $Z_{-(s_1, \ldots, s_{k-1})}^{\prime}$ in the direction of 
$\xi$.
Let 
\begin{equation}\label{eq-e-k-prime}
\sigma_k^\prime = \sigma_{e_{-\alpha}}^{\prime, (k-1)}.
\end{equation}
Consider the coordinates $(z_1, z_2, \ldots, z_n)$ on $\O_-^{\prime, \bfu}$ given in 
\eqref{eq-coor-prime}. Then $(z_1, \ldots, z_{k-1})$ can be considered as coordinates on 
the open submanifold 
\begin{align*}
\O^{\prime, (s_1, \ldots, s_{k-1})}_- &= B_-\backslash (B_-s_1 B_-) \times_{B_-} (B_-s_2 B_-) \times \cdots \times_{B_-}
(B_- s_{k-1} B_-)\\
&=\{[u_{\alpha_1} (z_1)\dot{s}_{\alpha_1},  \;\ldots, \; 
u_{\alpha_{k-1}} (z_{k-1})\dot{s}_{\alpha_{k-1}}]_-: \; (z_1, \ldots, z_{k-1}) \in \Cset^{k-1}\}
\end{align*}
of $Z_{-(s_1, \ldots, s_{k-1})}^{\prime}$, and $\sigma_k^\prime$ can be regarded as a derivation on
$\Cset[z_1, \ldots, z_{k-1}]$.

\begin{lemma}\label{le-zizj-prime}
In the coordinates $(z_1, z_2, \ldots, z_n)$ on $\O_-^{\prime, \bfu}$ given in 
\eqref{eq-coor-prime}, the Poisson structure $\pi_{-n}^\prime$ is given by 
\begin{equation}\label{eq-zizj-u-prime}
\{z_i, \, z_k\} = -c_{i,k} z_i z_k -\la \alpha_k, \, \alpha_k\ra \sigma_k^\prime(z_i),\hs 1 \leq i < k \leq n,
\end{equation}
where for $1 \leq i, k \leq n$, $c_{i,k}$ is given in \eqref{eq-cik}.
\end{lemma}

\begin{proof}
By repeatedly applying Lemma \ref{le-Z-Pi-2} to the Poisson manifold $(\O_-^{\prime, \bfu}, \pi_{-n}^\prime)$, one sees that
$\pi_{-n}^\prime$ is given in the coordinates $(z_1, z_2, \ldots, z_n)$ on $\O_-^{\prime, \bfu}$ by (see notation in
\eqref{eq-xi-k})
\[
\{z_i, \, z_k\} = -\frac{\la \alpha_k, \, \alpha_k\ra}{2} z_k \sigma_{h_{\alpha_k}}^{\prime, (k-1)}(z_i)
-\la \alpha_k, \, \alpha_k\ra \sigma_{k}^{\prime}(z_i),\hs 1 \leq i < k \leq n,
\]
For $h \in \t$, one checks directly from the definition of the vector field $\sigma_h^{\prime, (k-1)}$ that
\begin{equation}\label{eq-h-k}
\sigma_h^{\prime, (k-1)}(z_i) = (s_{k-1}s_{k-2} \cdots s_{i+1}(\alpha_i)(h)) z_i, \hs \hs 1 \leq i \leq k-1.
\end{equation}
\eqref{eq-zizj-u-prime} now follows from 
\begin{align*}
\frac{\la \alpha_k, \, \alpha_k\ra}{2} \sigma_{h_{\alpha_k}}^{\prime, (k-1)}(z_i)& = 
\la s_{k-1}s_{k-2}\cdots s_{i+1}(\alpha_i), \; \alpha_k\ra z_i\\
&= -\la s_1s_2 \cdots s_{i-1}(\alpha_i), \;
s_1s_2 \cdots s_{k-1}(\alpha_k)\ra z_i\\
&=c_{i,k}z_i.
\end{align*}
\end{proof}

\begin{corollary}\label{co-b-k-prime}
In the notation in Theorem \ref{thm-Ore} for the case of $\gamma = \bfu$, one has 
\[
b_k^\prime = \la \alpha_k, \alpha_k\ra \sigma_k^\prime, \hs \hs 
2 \leq k \leq n.
\]
\end{corollary}

\begin{proof} By Proposition \ref{pr-dual-pair} and Lemma \ref{le-zizj-prime}, the Poisson structure $\pi_n$ is given 
in the coordinates $(z_1, \ldots, z_n)$ on the affine chart
$\O^\bfu$ by 
\[
\{z_i, \, z_k\} = c_{i,k} z_i z_k +\la \alpha_k, \, \alpha_k\ra \sigma_k^\prime(z_i),\hs 1 \leq i < k \leq n.
\]
It follows from the definition of $b_k^\prime$ that $b_k^\prime = \la \alpha_k, \alpha_k\ra \sigma_k^\prime$ for
$2 \leq k \leq n$.
\end{proof}

\begin{remark}\label{re-h-b}
{\em
We already know from Lemma \ref{le-reverse} that for any $h \in \t$ and $2 \leq k \leq n$,
$[a_h^\prime, \, b_k^\prime] =\lambda_k(h)b_k^\prime$, as derivations of $\Cset[z_1, \ldots, z_{k-1}]$, where 
$a_h^\prime = \partial_h|_{\Cset[z_1, \ldots, z_{k-1}]}$ and $\lambda_k= s_1s_2 \cdots s_{k-1}(\alpha_k)$.
This fact can also be checked directly from Corollary \ref{co-b-k-prime}. Indeed, 
that in the notation of 
\eqref{eq-xi-k}, it follows from \eqref{eq-h-k} that $a_h^\prime = -\sigma^{\prime, (k-1)}_{s_{k-1} \cdots s_2s_1(h)}$ and
$b_k^\prime = \la \alpha_k, \alpha_k\ra \sigma_{e_{-\alpha_k}}^{\prime, (k-1)}$, so 
\[
[a_h^\prime, \, b_k^\prime] = -\la \alpha_k, \alpha_k\ra \left[\sigma^{\prime, (k-1)}_{s_{k-1} \cdots s_2s_1(h)}, \;
\sigma_{e_{-\alpha_k}}^{\prime, (k-1)}\right] = \lambda_k(h) b_k^\prime.
\]
\hfill $\diamond$
}
\end{remark}

\subsection{The polynomial rings $(\Zset[z_1, \ldots, z_n], \{\, , \, \}_\gamma)$}\label{subsec-Z}
Recall from $\S$\ref{piZw} that once the Borel subgroup $B$ and the maximal torus $T \subset B$ of $G$ are fixed, 
the definition of the Poisson structure $\pi_n$ on $\Zu$ depends only on the choice
of a symmetric non-degenerate invariant bilinear form $\lara$ on $\g$ and not on the choices of root
vectors $e_\alpha$ for $\alpha \in \Delta$. Although a choice of the set $\{e_\alpha: \alpha \in \Gamma\}$ of
root vectors for the simple roots is needed to define the coordinates $(z_1, \ldots, z_n)$ on 
$\O^\gamma$ for $\gamma \in \Upu$, we proved in Proposition \ref{pr-indep} that the polynomials $f_{i,k} :=\{z_i, z_k\} \in \Cset[z_1, \ldots, z_n]$
for $1 \leq i, k \leq n$
are independent on the choices of the root vectors for the simple roots. For each $\gamma \in \Upu$, one thus has a
well-defined Poisson polynomial algebra $(\Cset[z_1, \ldots, z_n], \{\, , \, \}_\gamma)$.

\begin{theorem}\label{th-Z}
Suppose that the symmetric non-degenerate invariant bilinear form $\lara$ on $\g$ is chosen such that 
$\frac{1}{2} \la \alpha, \, \alpha\ra \in \Zset$ for each $\alpha \in \Delta$. Then for any 
$\gamma \in \Upu$, the Poisson structure $\{\, , \, \}_\gamma$ on $\Cset[z_1, \ldots, z_n]$ has the property that
$\{z_i, z_k\} \in \Zset[z_i, \ldots, z_k] \subset \Zset[z_1, \ldots, z_n]$ for all $1 \leq i < k \leq n$.
\end{theorem}

\begin{proof}
Choose any set $\{e_\alpha: \alpha \in \Gamma\}$ of
root vectors for the simple roots and extend it to a Chevalley basis of $\g$. 
Theorem \ref{th-Z} now follows from Remark \ref{chevbase} and the fact that for any $\alpha, \beta \in \Delta$, 
\[
\la \alpha, \, \beta\ra = \frac{2\la\alpha, \beta\ra}{\la\alpha, \alpha\ra}\frac{\la \alpha, \alpha\ra}{2} \in \Zset.
\]
\end{proof}

Note that a canonical choice of the bilinear form $\lara$ on $\g$ is such that $\la \alpha, \alpha\ra = 2$ for
the short roots for each of the simple factors of $\g$. 

\begin{remark}\label{re-mod-2}
{\em
By Theorem \ref{th-Z}, each $\gamma \in \Upu$ gives rise to a Poisson algebra 
\[
({\bf k}[z_1, \ldots, z_n], \;\{\, ,  \}_\gamma)
\]
over any field ${\bf k}$ of arbitrary characteristic. In particular, it follows from \eqref{eq-zizj-th} in 
Theorem \ref{th-zizj} that the Poisson structure $\{\, , \, \}_\gamma$ on 
${\bf k}[z_1, \ldots, z_n]$ is log-canonical for every $\gamma \in \Upu$ if ${\rm char}({\bf k}) = 2$.

Choose the bilinear form $\lara$ on $\g$ such that $\la \alpha, \alpha \ra = 2$ for all the short roots. Then 
$\la \alpha, \alpha \ra  \in \{2, 4, 6\}$ for all $\alpha \in \Gamma$. It follows from \eqref{eq-lambda-h-i} that
$({\bf k}[z_1, \ldots, z_n], \{\,, \, \}_{\bfu})$ is a symmetric Poisson CGL extension of any field $\bk$ with
${\rm char}(\bk) \neq 2, 3$.
\hfill $\diamond$
}
\end{remark}

\subsection{Examples}\label{subsec-examples}
Assume that $\g$ is simple and let $\lara$ be such that $\la \alpha, \alpha\ra = 2$ for the short roots of $\g$.
Based on Theorem \ref{th-zizj}, the first author has written  a 
computer program in the GAP  language \cite{GAP4} which allows one to compute
the Poisson 
bracket $\{\, , \, \}_\gamma$ on $\Zset[z_1, \ldots, z_n]$ for any $\bfu = (s_1, \ldots, s_n)$ and any
$\gamma \in \Upu$. We given some examples.

\begin{example}\label{ex-G2-1}
{\em Consider $G_2$ with the two simple roots $\alpha_1$ and $\alpha_2$ satisfying 
\[
\la \alpha_2, \alpha_2\ra = 3\la \alpha_1, \alpha_1\ra = 6.
\]
 Let ${\bf u} = (s_{\alpha_1}, s_{\alpha_2}, s_{\alpha_1}, s_{\alpha_2},
s_{\alpha_1}, s_{\alpha_2})$ and note that $s_{\alpha_1}s_{\alpha_2}s_{\alpha_1}s_{\alpha_2}s_{\alpha_1}s_{\alpha_2}$
is the longest element in the Weyl group of $G_2$. For $\gamma = \bfu$, one has 
\begin{align*}
\{z_{1},z_{2}\} &= -3z_{1}z_{2}, \hs 
\{z_{1},z_{3}\} = -z_{1}z_{3}-2z_{2}, \hs 
\{z_{1},z_{4}\} = -6z_{3}^2, \\
\{z_{1},z_{5}\} &= z_{1}z_{5}-4z_{3}, \hs  
\{z_{1},z_{6}\} = 3z_{1}z_{6}-6z_{5}, \hs
\{z_{2},z_{3}\} = -3z_{2}z_{3}\\ 
\{z_{2},z_{4}\} &= -6z_{3}^3-3z_{2}z_{4}, \hs 
\{z_{2},z_{5}\} = -6z_{3}^2, \hs 
\{z_{2},z_{6}\} = 3z_{2}z_{6}-18z_{3}z_{5}+6z_{4}\\ 
\{z_{3},z_{4}\} &= -3z_{3}z_{4}, \hs 
\{z_{3},z_{5}\} = -z_{3}z_{5}-2z_{4}, \hs 
\{z_{3},z_{6}\} = -6z_{5}^2\\ 
\{z_{4},z_{5}\} &= -3z_{4}z_{5}, \hs 
\{z_{4},z_{6}\} = -6z_{5}^3-3z_{4}z_{6}, \hs 
\{z_{5},z_{6}\} = -3z_{5}z_{6}.
\end{align*}
For the same $\bfu$ but $\gamma =(s_{\alpha_1}, s_{\alpha_2}, e, e, s_{\alpha_1}, e)$, one has
\begin{align*}
\{z_{1},z_{2}\} &= -3z_{1}z_{2}, \hs  
\{z_{1},z_{3}\} = 2z_{2}z_{3}^2+z_{1}z_{3}, \hs
\{z_{1},z_{4}\} = -6z_{2}z_{3}z_{4}+6z_{3}z_{4}^2-3z_{1}z_{4},\\
& \{z_{1},z_{5}\} = -4z_{2}z_{3}z_{5}+6z_{3}z_{4}z_{5}-z_{1}z_{5}-2z_{2}+2z_{4},\\ 
\{z_{1},z_{6}\} &= 6z_{3}z_{5}^3z_{6}^2+6z_{5}^2z_{6}^2+6z_{2}z_{3}z_{6}-6z_{3}z_{4}z_{6}, \hs  
\{z_{2},z_{3}\} = 3z_{2}z_{3},\\ 
\{z_{2},z_{4}\} &= -6z_{2}z_{4}+6z_{4}^2, \hs  
\{z_{2},z_{5}\} = -3z_{2}z_{5}+6z_{4}z_{5},\\ 
\{z_{2},z_{6}\} &= 6z_{5}^3z_{6}^2+3z_{2}z_{6}-6z_{4}z_{6}, \hs 
\{z_{3},z_{4}\} = -3z_{3}z_{4},\hs 
\{z_{3},z_{5}\} = -2z_{3}z_{5}, \\  
\{z_{3},z_{6}\} &= 3z_{3}z_{6}, \hs  
\{z_{4},z_{5}\} = 3z_{4}z_{5}, \hs 
\{z_{4},z_{6}\} = -3z_{4}z_{6}, \hs 
\{z_{5},z_{6}\} = 3z_{5}z_{6}.
\end{align*}
\hfill $\diamond$
}
\end{example}

\begin{example}\label{ex-sl2-s}
{\em Consider $G = SL(2)$ with the only simple root denoted by $\alpha$ and $s = s_\alpha$ and 
$\la \alpha, \, \alpha \ra = 2$. Let $\bfu = (s, s, s, s, s)$.
For $\gamma = \bfu$, one has
\begin{align*}
\{z_{1},z_{2}\} &= 2z_{1}z_{2}-2, \hs  
\{z_{1},z_{3}\} = -2z_{1}z_{3}, \hs  
\{z_{1},z_{4}\} = 2z_{1}z_{4}, \hs  
\{z_{1},z_{5}\} = -2z_{1}z_{5}, \\ 
\{z_{2},z_{3}\} &= 2z_{2}z_{3}-2, \hs 
\{z_{2},z_{4}\} = -2z_{2}z_{4}, \hs  
\{z_{2},z_{5}\} = 2z_{2}z_{5}, \hs  
\{z_{3},z_{4}\} = 2z_{3}z_{4}-2,\\ 
\{z_{3},z_{5}\} &= -2z_{3}z_{5}, \hs  
\{z_{4},z_{5}\} = 2z_{4}z_{5}-2. 
\end{align*}
For $\gamma = (s, e, e, e, s)$, one has
\begin{align*}
\{z_{1},z_{2}\} &= -2z_{1}z_{2}+2z_{2}^2, \hs  
\{z_{1},z_{3}\} = -2z_{1}z_{3}+4z_{2}z_{3}+2z_{3}^2,\\ 
\{z_{1},z_{4}\} &= -2z_{1}z_{4}+4z_{2}z_{4}+4z_{3}z_{4}+2z_{4}^2, \\ 
\{z_{1},z_{5}\} &= 2z_{1}z_{5}-4z_{2}z_{5}-4z_{3}z_{5}-4z_{4}z_{5}-2,\\ 
\{z_{2},z_{3}\} &= 2z_{2}z_{3}, \hs 
\{z_{2},z_{4}\} = 2z_{2}z_{4}, \hs 
\{z_{2},z_{5}\} = -2z_{2}z_{5},\\ 
\{z_{3},z_{4}\} &= 2z_{3}z_{4}, \hs  
\{z_{3},z_{5}\} = -2z_{3}z_{5}, \hs  
\{z_{4},z_{5}\} = -2z_{4}z_{5}.
\end{align*}
In general, it is easy to see from Theorem \ref{th-zizj} that for the sequence $\bfu = (s, s, \ldots, s)$ of length $n$,
and $\gamma = \bfu$, the Poisson bracket $\{\, , \, \}_\gamma$ on $\Zset[z_1, \ldots, z_n]$ is given by
\begin{align*}
&\{z_i, z_{i+1}\} = 2z_iz_{i+1} - 2, \hs 1 \leq i \leq n-1, \\
&\{z_i, z_k\}  = 2(-1)^{k-j+1} z_iz_k, \hs 1 \leq i < k \leq n, \, k-i \geq 2.
\end{align*}
The coefficient $2$ in all the Poisson brackets results from that fact that $\la \alpha, \alpha \ra = 2$.
}
\end{example}

\appendix
\section{Poisson pairs}\label{sec-Poisson-pairs}
In \cite[$\S$8.5]{Lu-Mou:mixed}, 
a {\it Poisson pair} is defined to be a pair of Poisson maps
\begin{equation}\label{eq-rho-YZ}
\rho_\sY: \;(X, \, \piX) \lrw (Y, \,\piY) \hs \mbox{and} \hs  \; \rho_\sZ: \; (X, \,\piX) \lrw (Z, \,\piZ)
\end{equation}
between Poisson manifolds such that 
the  map
\[
(\rho_\sY, \rho_\sZ): \;\; (X, \,\piX) \lrw (Y \times Z, \; \pi_\sY \times \piZ), \;\;\; x \longmapsto
(\rho_\sY(x), \; \rho_\sZ(x)), \hs x \in X,
\]
is Poisson. If $(Y, \piY)$ and $(Z, \piZ)$ are two Poisson manifolds, the projections from the product
Poisson manifold $(Y \times Z, \, \piY \times \piZ)$ to the two factors
clearly form a Poisson pair. Moreover, for a differentiable map $\phi: Y \to Z$, it is well-known
\cite{weinstein:coisotropic} that $\phi: (Y, \piY) \to (Z, \piZ)$ is
anti-Poisson if and only if the graph of $\phi$, i.e., 
\[
{\rm Graph}(\phi) = \{(y, \, \phi(y): \,y \in Y\} \subset Y \times Z,
\]
is a coisotropic submanifold of $(Y \times Z, \; \piY \times \piZ)$. The following
Lemma \ref{le-graph-coisotropic} is a (partial) generalization of this fact to the case of
Poisson pairs.

\begin{lemma}\label{le-graph-coisotropic}
Let $(\rho_{\sY}, \rho_{\sZ})$ be a Poisson pair as in \eqref{eq-rho-YZ}. Suppose that $X'$ is a coisotropic submanifold of
$(X, \piX)$ such that $\rho_\sY|_{\sX'}: X' \to Y$ is a diffeomorphism. Then
\[
\phi = \rho_\sZ \circ (\rho_\sY|_{\sX'})^{-1}:  \;\; (Y, \piY) \lrw (Z, \piZ)
\]
is an anti-Poisson map.
\end{lemma}

\begin{proof}
Fix $x \in X'$ and let $\rho_\sY(x) = y$ and $z = \rho_\sZ(x) \in Z$. Let 
\[
\rho_{\sY, x}: \;\; T_xX \lrw T_yY \hs \mbox{and} \hs 
\rho_{\sZ, x}: \;\; T_xX \lrw T_zZ
\]
be respectively the
differentials of $\rho_\sY$ and $\rho_\sZ$ at $x$. Lemma \ref{le-graph-coisotropic}
now follows from  the following Lemma \ref{le-linear-algebra} by 
taking
$(V, \pi) = (T_xX, \piX(x))$, $V_1 = \ker \rho_{\sY, x}$, $V_2 = \ker \rho_{\sZ, x}$, and $U = T_xX'$.
\end{proof}

In the following Lemma \ref{le-linear-algebra}, for a finite dimensional vector space $V$ 
and a subspace $U_1 \subset V$, set $U_1^0 = \{\xi \in V^*: \xi|_{{\scriptscriptstyle U}_1} = 0\} \subset V^*$, and $U_1$ is said to be 
coisotropic with respect to $\pi \in \wedge^2 V$ if
 $\pi \in U_1 \wedge V$, where for any subspace $U_2$ of $V$, 
\[
U_1 \wedge U_2 = (\wedge^2 V) \cap (U_1 \otimes U_2 + U_2 \otimes U_1) \subset \wedge^2 V.
\]

\begin{lemma}\label{le-linear-algebra}
Let $V$ be a finite dimensional vector space, let $\pi \in \wedge^2 V$, and let $V_1$ and $V_2$
be two vector subspaces of $V$ such that $\pi(V_1^0, V_2^0) = 0$. For $j = 1, 2$, let $\rho_j: V \to V/V_j$ 
be the projections so that $\rho_j(\pi) \in \wedge^2 (V/V_j)$. 
Assume that $U$ is a coisotropic subspace of $V$ and that $\rho_1|_{{\scriptscriptstyle U}}: U \to V/V_1$ is an isomorphism. 
Let $\psi = \rho_2 \circ (\rho_1|_{{\scriptscriptstyle U}})^{-1}: V/V_1 \to V/V_2$. Then 
$\psi(\rho_1(\pi)) = -\rho_2(\pi)$.
\end{lemma} 

\begin{proof}
For  $\pi' = \sum_j v_j \wedge v_j^\prime \in \wedge^2V$ and $\xi \in V^*$, let
$\xi \rfloor \pi' = \sum_j (\la \xi, v_j\ra v_j^\prime - \la \xi, v_j^\prime\ra v_j)$, 
where $\lara$ denotes the pairing between $V$ and $V^*$. Then the condition $\pi(V_1^0, V_2^0) = 0$ is equivalent to
$\xi \rfloor \pi \in V_2$ for all $\xi \in V_1^0$. 
By assumption, $V = U + V_1$ is a 
direct sum. As $U$ is coisotropic with respect to $\pi$, one can uniquely write $\pi =\pi_{\sU} + \pi_1$, where $\pi_{\sU} \in \wedge^2 U$ and  $\pi_1 \in U \wedge V_1$.
Let $\{u_1, \ldots, u_m\}$ be a basis of $U$ and let $\xi_i \in V_1^0$, $1 \leq i \leq m$, be such that $\la u_i, \xi_j\ra = \delta_{i,j}$ for $1\leq i, j \leq m$.
Then 
\[
\pi_\sU = \frac{1}{2} \sum_{i=1}^m u_i \wedge (\xi_i \rfloor \pi_\sU) \hs \mbox{and} \hs
\pi_1 =\sum_{i=1}^m u_i \wedge (\xi_i \rfloor \pi_1).
\]
For $1 \leq i \leq m$, let $x_i = \xi_i \rfloor \pi = \xi_i \rfloor (\pi_\sU + \pi_1)$. Then
\begin{align*}
\pi &= \frac{1}{2} \sum_{i=1}^m u_i \wedge (\xi_i \rfloor \pi_\sU) + \sum_{i=1}^m u_i \wedge (\xi_i \rfloor (\pi_\sU +\pi_1))
-\sum_{i=1}^m u_i \wedge (\xi_i \rfloor \pi_\sU)\\
& = -\frac{1}{2} \sum_{i=1}^m u_i \wedge (\xi_i \rfloor \pi_\sU) + \sum_{i=1}^m u_i \wedge x_i = -\pi_\sU + 
\sum_{i=1}^m u_i \wedge x_i.
\end{align*}
As $x_i \in V_2$ for each $1 \leq i \leq m$,  $\rho_2(\sum_{i=1}^m u_i \wedge x_i) = 0$, so
$\psi(\rho_1(\pi)) = -\rho_2(\pi)$.
\end{proof}

\end{document}